%% file: burnside_mcg.tex
\documentclass{coulonpaper}

\title{Partial periodic quotient of groups acting on a hyperbolic space.}
\author{R\'emi Coulon}

\begin{document} 

\maketitle

\begin{abstract}
	In this article, we construct partial periodic quotients of groups which have a non-elementary acylindrical action on a hyperbolic space.
	In particular, we provide infinite quotients of mapping class groups where a fixed power of every pseudo-Anosov homeomorphism is identified with a periodic or reducible element.
\end{abstract}

\tableofcontents 

\input{0_introduction}
\input{1_hyperbolic_geometry}

\input{2_cone_off}

\input{3_small_cancellation}

\input{4_applications}

\makebiblio

\noindent
\emph{R\'emi Coulon} \\
Department of Mathematics, Vanderbilt University\\
Stevenson Center 1326, Nashville TN 37240, USA\\
\texttt{remi.coulon@vanderbilt.edu} \\
\texttt{http://www.math.vanderbilt.edu/$\sim$coulonrb/}

\todos

\end{document}

%% file: 0_introduction.tex

\section{Introduction}

\paragraph{}
Let $G$ be a group.
We say that $G$ is \emph{periodic} with \emph{exponent} $n$ if for every $g \in G$, $g^n=1$.
In 1902, W.~Burnside asked whether or not a finitely generated periodic group was necessarily finite.
Despite the simplicity of the statement, this question remained open for a long time and motivated many developments in group theory.
In 1968, P.S.~Novikov and S.I.~Adian achieved a breakthrough by providing the first examples of infinite periodic finitely generated  groups \cite{NovAdj68c}.
See also \cite{Olc82} and \cite{DelGro08}.
We now know that if $G$ is a hyperbolic group which is not virtually cyclic then there exists an integer $n$ such that $G$ has an infinite quotient of exponent $n$ \cite{IvaOlc96}.
As opposed to this situation any finitely generated periodic linear group is finite \cite{zbMATH02628890}.

\paragraph{}
The original motivation for our work was the following question.
What are the finitely generated groups which admit an infinite periodic quotient?
With this level of generality, it is very difficult to understand what could be the periodic quotients of an arbitrary non-hyperbolic group $G$.
In this article we are interested in partial periodic quotients of the form $G/S^n$ where $S^n$ stands for the normal subgroup generated by the $n$-th power of every element in a large subset $S$ of $G$.
Our construction provides various examples of quotients with exotic properties.
Let us mention two applications.

\paragraph{Quotient of amalgamated products.}
Recall that a subgroup $H$ of a group $G$ is \emph{malnormal} if for every $g \in G$, $gHg^{-1} \cap H$ is trivial provided $g$ does not belong to $H$.

\begin{theo}
\label{res: intro free product}
	Let $A$ and $B$ be two groups without involution.
	Let $C$ be a subgroup of $A$ and $B$ malnormal in $A$ or $B$.
	There is an integer $n_1$ such that for every odd exponent $n \geq n_1$ there exists a group $G$ with the following properties.
	\begin{enumerate}
		\item The groups $A$ and $B$ embed into $G$ such that the diagram below commutes.
		\begin{center}
			\begin{tikzpicture}[description/.style={fill=white,inner sep=2pt},] 
				\matrix (m) [matrix of math nodes, row sep=2em, column sep=2.5em, text height=1.5ex, text depth=0.25ex] 
				{ 
					C	& B		\\
					A	& G		\\
				}; 
				\draw[>=stealth, ->] (m-1-1) -- (m-1-2);
				\draw[>=stealth, ->] (m-2-1) -- (m-2-2);
				\draw[>=stealth, ->] (m-1-1) -- (m-2-1);
				\draw[>=stealth, ->] (m-1-2) -- (m-2-2);
			\end{tikzpicture} 
		\end{center}
		\item For every $g \in G$, if $g$ is not conjugated to an element of $A$ or $B$ then $g^n=1$.
		\item There are infinitely many elements in $G$ which are not conjugated to an element of $A$ or $B$.
	\end{enumerate}
\end{theo}
A similar statement has been obtained by K.~Lossov in his Ph.D. dissertation but has never been published though.

\paragraph{Mapping class group.}
Our next example is new and comes from the geometry of surfaces.
Let $\Sigma$ be a compact surface of genus $g$ with $p$ boundary components.
The \emph{mapping class group} $\mcg \Sigma$ of $\Sigma$ is the group of orientation preserving self homeomorphisms of $\Sigma$ defined up to homotopy.
A mapping class $f \in \mcg \Sigma$ is \emph{periodic} if it has finite order; \emph{reducible} if it permutes a collection of essential non-peripheral curves (up to isotopy); \emph{pseudo-Anosov} if there exists an homotopy in the class of $f$ that preserves a pair of transverse foliations and rescale these foliations in an appropriate way.
It follows from Thurston's work that any element of $\mcg \Sigma$ falls into one these three categories \cite{Thu88}.
We produce a quotient of $\mcg \Sigma$ where a fixed power of every pseudo-Anosov element ``becomes'' periodic or reducible.

\begin{theo}
\label{res: intro mcg}
	Let $\Sigma$ be a compact surface of genus $g$ with $p$ boundary components such that $3g +p - 3 >1$.
	There exist integers $\kappa$ and $n_0$ such that for every odd exponent $n \geq n_0$ there is a quotient $G$ of $\mcg \Sigma$ with the following properties.
	\begin{enumerate}
		\item If $E$ is a subgroup of $\mcg \Sigma$ that does not contain a pseudo-Anosov element, then the projection $\mcg \Sigma \twoheadrightarrow G$ induces an isomorphism from $E$ onto its image.
		\item Let $f$ be a pseudo-Anosov element of $\mcg \Sigma$.
		Either $f^{\kappa n} = 1$ in $G$ or there exists a periodic or reducible element $u \in \mcg \Sigma$ such that $f^\kappa = u$ in $G$.
		In particular, for every pseudo-Anosov $f \in \mcg \Sigma$, there exists a non-pseudo-Anosov element $u \in \mcg \Sigma$ such that $f^{\kappa n} = u$ in $G$.
		\item There are infinitely many elements in $G$ which are not the image of a periodic or reducible element of $\mcg \Sigma$.
	\end{enumerate}
\end{theo}

A ping-pong argument shows that $\mcg \Sigma$ contains many free purely pseudo-Anosov subgroups.
By \emph{purely pseudo-Anosov subgroup} we mean that any non-trivial element of this subgroup is pseudo-Anosov.
Until recently it was an open whether $\mcg \Sigma$ had purely pseudo-Anosov \emph{normal} subgroups.
This question was for instance listed in Kirby's book as Problem 2.12(A) \cite{Kirby:1997wt}.
See also \cite[Problem 3]{Ivanov:2006wh} and \cite[Paragraph 2.4]{Farb:2006vi}.
In \cite{Dahmani:2011vu}, F.~Dahmani, V.~Guirardel and D.~Osin provide many examples of such groups.
More precisely they prove the following.
There exists an integer $n$ (that depends only on the surface $\Sigma$) such that if $f \in \mcg \Sigma$ is pseudo-Anosov, then the normal closure of $f^n$ is free and purely pseudo-Anosov \cite[Theorem 8.1]{Dahmani:2011vu}.
One could ask whether or not there is an integer $n$ such that the normal subgroup $N$ of $\mcg \Sigma$ generated by the $n$-th power of \emph{every} pseudo-Anosov element is purely pseudo-Anosov.
However such an integer cannot exist.
Indeed one can find a pseudo-Anosov element $f$ and an infinite order reducible element $u$ such that $f^nu$ is pseudo-Anosov.
If both $f^n$ and $(f^nu)^n$ belong to $N$, then the reducible element 
\begin{equation*}
 u^n 
 = \left(u^{n-1}f^{-n}u^{-(n-1)} \right) \cdots \left(u^2f^{-n}u^{-2} \right)  \left(uf^{-n}u^{-1} \right) f^{-n}\left( f^nu\right)^n,
\end{equation*}
would also belong to $N$.
Nevertheless, if $G$ stands for the quotient given by \autoref{res: intro mcg}, then the kernel $K$ of the projection $\mcg \Sigma \twoheadrightarrow G$ provides a purely pseudo-Anosov normal subgroup that contains a fixed power of most of the pseudo-Anosov elements of $\mcg \Sigma$.
Following \cite{Ivanov:2006wh}, we wonder whether this kernel is a free group.

\begin{coro}
\label{res: intro purely pseudo-Anosov normal subgroup}
	Let $\Sigma$ be a compact surface of genus $g$ with $p$ boundary components such that $3g +p - 3 >1$.
	There exist integers $\kappa$ and $n_0$ such that for every odd exponent $n \geq n_0$ there is a subgroup $K$ of $\mcg \Sigma$ with the following properties.
	\begin{enumerate}
		\item $K$ is normal and purely pseudo-Anosov.
		\item As a normal subgroup, $K$ is not finitely generated.
		\item For every pseudo-Anosov element $f \in \mcg \Sigma$ either $f^{\kappa n}$ belongs to $K$ or there exists a periodic or reducible element $u \in  \mcg \Sigma$ such that $f^\kappa u$ belongs to $K$.
	\end{enumerate}
\end{coro}

\paragraph{}In his seminal paper M.~Gromov introduced the concept of $\delta$-hyperbolic spaces \cite{Gro87}.
Using a simple four point inequality, he captured most of the large scale features of metric spaces with some negative curvature.
For a group $G$ being hyperbolic means that its Cayley is hyperbolic as a metric space.
Generalizing this idea, M.~Gromov also defined the notion of relatively hyperbolic groups.
For many purposes the Cayley graph is not the most appropriate space to work with.
To take advantage of the hyperbolic geometry what really matters though is to have $G$ acting ``nicely'' on a hyperbolic space.
However not all actions will do the job.
Indeed every group admits a proper action on a hyperbolic space.
To make this idea works the action need to satisfy some finiteness condition.
For instance a group $G$ is 
\begin{enumerate}
	\item \emph{hyperbolic} if and only if it acts properly co-compactly on a hyperbolic length space $X$.
	\item \emph{relatively hyperbolic} if and only if it acts properly on a hyperbolic length space $X$ with some finiteness condition for the induced action of $G$ on the boundary at infinity $\partial X$ of $X$.
\end{enumerate}
These two classes already cover numerous examples of groups: geometrically finite Kleinian groups, fundamental groups of finite volume manifolds with pinched sectional curvature, small cancellation groups, amalgamated products over finite groups, etc.
In this article we focus on a weaker condition: acylindricity.
It was first used by Z.~Sela for actions on a tree \cite{Sela:1997gh}. 
The following formulation is due to B.~Bowditch \cite{Bowditch:2008bj}.

\begin{defi}
\label{def: intro - acylindricity}
	The action of a group $G$ on a metric space $X$ is \emph{acylindrical} if for every $l \geq 0$, there exist $d \geq 0$ and $N \geq 0$ with the following property.
	For every $x, x' \in X$ with $\dist x{x'} \geq d$ the set of elements $u \in G$ satisfying $\dist {ux}x \leq l$ and $\dist{ux'}{x'} \leq l$ contains at most $N$ elements.
\end{defi}

Roughly speaking, it means that the stabilizers of long paths are finite with some uniform control on their cardinality.
\begin{exam}
\label{exa: intro example aspherical amalgamated product}
	Let $A$ and $B$ be two groups.
	Let $C$ be a subgroup of $A$ and $B$ which is malnormal in $A$ or $B$.
	The action of the amalgamated product $\amp[C]AB$ on the corresponding Bass-Serre tree is aspherical \cite{Sela:1997gh}.
\end{exam}
\begin{exam}
\label{exa: intro example aspherical mcg}
	Let $\Sigma$ be a compact surface of genus $g$ with $p$ boundary components.
	The \emph{complex of curves} $X$ is a simplicial complex associated to $\Sigma$ introduced by W.~Harvey \cite{Harvey:1981tg}.
	The simplices of $X$ are collections of curves of $\Sigma$ that can be disjointly realized.
	H.~Masur and Y.~Minsky proved that this space is hyperbolic \cite{Masur:1999hc}.
	By construction, $X$ is endowed with an action by isometries of $\mcg \Sigma$.
	Moreover B.~Bowditch showed that this action is acylindrical \cite{Bowditch:2008bj}.
\end{exam}

\paragraph{}
The action of a group on a metric space is \emph{non-elementary} if its orbits are neither bounded or quasi-isometric to a line.
D.~Osin studied the class of groups that admit a non-elementary acylindrical action on a hyperbolic space $X$.
It turns out that this class is very large \cite{Osin:2013te}.
Besides the two examples mentioned previously it also contains hyperbolic groups, relatively hyperbolic groups, outer automorphism groups of free groups, right angle Artin groups which are not cyclic or split as a direct product,  the Cremona group, etc.
More examples are given in the work of A.~Minasyan and D.~Osin \cite{Minasyan:2013wp}.

\paragraph{}
Let $G$ be a group acting acylindrically on hyperbolic space $X$.
Just as with hyperbolic groups, an element $g \in G$ is either elliptic (its orbits are bounded) or loxodromic (given $x \in X$, the map $\Z \rightarrow X$ that sends $m$ to $g^mx$ is a quasi-isometric embedding).
Every elementary subgroup $E$ of $G$ either has bounded orbits or is virtually $\Z$.
The number $e(G,X)$ is the least common multiple of the exponents of  the holomorph $\hol F = \sdp F{\aut F}$, where $F$ describes the maximal finite normal subgroup of all maximal non-elliptic elementary subgroups of $G$.
Provided this number is odd, our main result explains how to build a quotient $G/K$ of $G$ with the following properties.
Any elliptic element is not affected; a fixed power of every loxodromic element is identified with an elliptic one.
More precisely we prove the following statement.

\begin{theo}
\label{res: intro - general acylindrical case}
	Let $X$ be a hyperbolic length space.
	Let $G$ be a group acting by isometries on $X$.
	We assume that the action of $G$ is acylindrical and non-elementary.
	Let $N$ be a normal subgroup of $G$ without involution.
	Assume that $e(N,X)$ is odd.
	There is a critical exponent $n_1$ such that every odd integer $n \geq n_1$ which is a multiple of $e(N,X)$ has the following property.
	There exists a normal subgroup $K$ of $G$ contained in $N$ such that
	\begin{itemize}
		\item if $E$ is an elliptic subgroup of $G$, then the projection $G \twoheadrightarrow G/K$ induces an isomorphism from $E$ onto its image;
		\item for every element $g \in N/K$, either $g^n=1$ or $g$ is the image an elliptic element of $G$;
		\item there are infinitely many elements in $N/K$ which do not belong to the image of an elliptic subgroup of $G$.
	\end{itemize}
\end{theo}

The normal subgroup $N$ in \autoref{res: intro - general acylindrical case} is a technical trick to deal with even torsion in $G$.
For most of our applications we will just take $N = G$.
For instance, \autoref{res: intro - general acylindrical case} applied with the amalgamated product $\amp [C]AB$ of \autoref{exa: intro example aspherical amalgamated product} gives \autoref{res: intro free product}.
The mapping class group $\mcg \Sigma$ of a surface $\Sigma$ does contain elements of order $2$.
However it has a finite index torsion-free normal subgroup $N$.
Thus \autoref{res: intro - general acylindrical case} leads to \autoref{res: intro mcg}.
The constant $\kappa$ in \autoref{res: intro mcg} is exactly the least common multiple of $e(N,X)$ and the index of $N$ in $\mcg \Sigma$.

\paragraph{}Our theorem actually holds in a more general situation (see \autoref{res : SC - partial periodic quotient}).
However the statement requires additional invariants for the action of $G$ on $X$ (see \autoref{sec:group invariants}).
This larger framework allows in particular the group $G$ to contain parabolic subgroup which is never the case for an acylindrical action.

\paragraph{}
The proof of \autoref{res: intro - general acylindrical case} relies on techniques introduced by T.~Delzant and M.~Gromov to study free Burnside groups of odd exponents.
Recall that the \emph{free Burnside group} $\burn rn$ of rank $r$ and exponent $n$ is the quotient of the free group $\free r$ of rank $r$ by the normal subgroup $\free r^n$ generated the $n$-th power of every element.
It is the largest group of rank $r$ and exponent $n$.
In \cite{DelGro08}, T.~Delzant and M.~Gromov provide an alternative proof of the infiniteness of $\burn rn$ for sufficiently large odd integers $n$.
To that end they construct a sequence of non-elementary hyperbolic groups 
\begin{equation*}
	\free r  = G_0 \twoheadrightarrow G_1 \twoheadrightarrow G_2 \twoheadrightarrow \dots \twoheadrightarrow G_k \twoheadrightarrow \dots
\end{equation*}
whose direct limit is $\burn rn$.
Each group $G_k$ is obtained using a geometrical form of small cancellation theory by adjoining to the previous group new relations of the form $g^n$.
The infiniteness of $\burn rn$ follows the from the hyperbolicity of the approximation groups $G_k$.
For a detailed presentation of this approach we refer the reader to the notes written by the author \cite{Coulon:2013tx}.

\paragraph{}
It appears that small cancellation theory can be extended to a larger class of groups.
In the previous process if $G_k$ is a group acting ``nicely'' on a hyperbolic space $X_k$ one can construct a hyperbolic space $X_{k+1}$ on which $G_{k+1}$ acts with similar properties \cite{Coulon:il,Dahmani:2011vu}.
The main difficulty is to make sure that one can indefinitely iterate this construction.
In the case of free Burnside groups of odd exponents T.~Delzant and M.~Gromov used two invariants (the injectivity radius and the invariant $A$, see \autoref{def: injectivity radius} and \autoref{def: invariant A}) to control the small cancellation parameters during the process.
The other key ingredient is the algebraic structure of the approximation groups $G_k$: every elementary subgroup of $G_k$ is cyclic.
This remarkable property explains why the case of odd exponents is much easier than the even one.
If, instead of a free group, we initiate the construction with a group $G$ acting acylindrically on a hyperbolic space, then the algebraic structure of $G$ will never be as simple.
Indeed the elliptic subgroups of $G$ can be anything.
To handle this difficulty we use a new invariant $\nu(G,X)$.
Formally, it is the smallest integer $m$ with the following property.
Given any two elements $g,h \in G$ with $h$ loxodromic, if $g$, $h^{-1}gh,\dots,h^{-m}gh^m$ generate an elliptic subgroup, then $g$ and $h$ generate an elementary subgroup of $G$ (see \autoref{def: invariant nu}).
This new parameter will allow us to control the structure of elementary subgroups which are not elliptic.

\paragraph{Outline of the paper.}
In \autoref{sec: hyperbolic geometry} and \autoref{sec: group acting on a hyperbolic space} we review some of the standard facts on hyperbolic spaces and groups acting on hyperbolic spaces.
In particular, in \autoref{sec:group invariants}, we define all the invariants that are needed to iterate later the small cancellation process.
In \autoref{sec: cone-off over metric space} we recall the cone-off construction which is one of the key tool in the geometrical approach of small cancellation.
\autoref{sec: small cancellation theory} is dedicated to small cancellation theory.
If $G$ is a group acting a hyperbolic space $X$ we explain how to use small cancellation theory to produce a quotient $\bar G$ with an action on a hyperbolic space $\bar X$.
Moreover we show that the invariants associated to the action of $\bar G$ on $\bar X$ can be controlled using the ones describing the action of $G$ on $X$.
In the beginning of \autoref{sec: applications}, we prove a statement (see \autoref{res: SC - induction lemma}) that will be used as the induction step in the proof of the main theorem (see \autoref{res : SC - partial periodic quotient}).
Finally discuss some applications of our results.

\paragraph{Acknowledgment.}
The author is grateful to T.~Delzant who brought the invariant $\nu$ to his attention.
He would like also to thank V.~Guirardel for related discussions.

%% file: 1_hyperbolic_geometry.tex

%
%

\section{Hyperbolic geometry}
\label{sec: hyperbolic geometry}
 
	\paragraph{} In this section we recall some basic ideas about hyperbolic spaces in the sense of M.~Gromov.

%
%

\subsection{Definitions}

	\paragraph{Notations and vocabulary.}
	Let $X$ be a metric length space.
	Unless otherwise stated a path is a rectifiable path parametrized by arclength.
	Given two points $x$ and $x'$ of $X$, we denote by $\dist[X]x{x'}$ (or simply $\dist x{x'}$) the distance between them.
	We write $B(x,r)$ for the open ball of $X$ of center $x $ and radius $r$.
	The space is said to be \emph{proper} if every closed bounded subset is compact.
	Let $Y$ be a subset of $X$.
	We write $d(x,Y)$ for the distance of a point $x \in X$ from $Y$.
	We denote by $Y^{+\alpha}$, the \emph{$\alpha$-neighborhood} of $Y$, i.e. the set of points $x \in X$ such that $d(x, Y) \leq \alpha$.
	The \emph{open $\alpha$-neighborhood} of $Y$ is the set of points $x \in X$ such that $d(x,Y)<\alpha$.
	Let $\eta \geq 0$.
	A point $p$ of $Y$ is an \emph{$\eta$-projection} of $x \in X$ on $Y$ if $\dist xp \leq d(x,Y) +\eta$.
	A 0-projection is simply called a \emph{projection}.

	\paragraph{The four point inequality.}
		The Gromov product of three points $x,y,z \in X$  is defined by 
	\begin{displaymath}
		\gro xyz = \frac 12 \left\{ \fantomB \dist xz + \dist yz - \dist xy \right\}.
	\end{displaymath}
	The space $X$ is \emph{$\delta$-hyperbolic} if for every $x,y,z,t \in X$
	\begin{equation}
	\label{eqn: hyperbolicity condition 1}
		\gro xzt \geq \min\left\{\fantomB \gro xyt, \gro yzt \right\} - \delta,
	\end{equation}
	or equivalently
	\begin{equation}
	\label{eqn: hyperbolicity condition 2}
		\dist xz + \dist yt \leq \max\left\{ \fantomB \dist xy + \dist zt, \dist xt + \dist yz  \right\} +2\delta.
	\end{equation}
	
	\rems Note that in the definition of hyperbolicity we do not assume that $X$ is geodesic or proper.
	For some of the results in this section, the cited reference only provides a proof for the case of geodesic metric spaces.
	However, by relaxing if necessary some constants, which we do here, the same proof works in the more general context of length spaces.

	\paragraph{}
	If $X$ is 0-hyperbolic, then it can be isometrically embedded in an $\R$-tree, \cite[Chapitre 2, Proposition 6]{GhyHar90}.
	For our purpose though, we will always assume that the hyperbolicity constant $\delta$ is positive.
	It is indeed more convenient to define particular subsets (see \autoref{def: hull} of a hull or \autoref{def: axes} of an axis) without introducing other auxiliary positive parameters.
	The hyperbolicity constant of the hyperbolic plane $\H$ will play a particular role.
	We denote it by $\boldsymbol \delta$ (bold delta).
		
	\paragraph{} 
	From now on we assume that $X$ is $\delta$-hyperbolic.
	It is known that triangles in a geodesic hyperbolic space are $4 \delta$-thin (every side lies in the $4\delta$-neighborhood of the union of the two other ones). 
	Since our space is not geodesic, we use instead the following metric inequalities.
	In this lemma the Gromov products $\gro xzt$, $\gro xys$ and $\gro xyt$ should be thought as very small quantities.
	The proof is left to the reader.
	
	\begin{lemm}
	\label{res: metric inequalities}
		Let $x$, $y$, $z$, $s$ and $t$ be  five points of $X$.
		\begin{enumerate}
			\item \label{enu: metric inequalities - thin triangle}
			\begin{math}
				\gro xyt \leq \max \left\{\fantomB \dist xt - \gro yzx , \gro xzt  \right\} + \delta,
			\end{math}
			\item \label{enu: metric inequalities - two points close to a geodesic}
			\begin{math}
				\dist st \leq \left|\fantomB \dist xs - \dist xt \right| + 2\max\left\{\fantomB \gro xys, \gro xyt\right\} + 2\delta,
			\end{math}
			\item \label{enu: metric inequalities - comparison tripod}
			The distance $\dist st$ is bounded above by
			\begin{equation*}
				\max\left\{\fantomB 
					\dist{\fantomB \dist xs}{\dist xt} + 2\max \left\{ \gro xys,\gro xzt\right\},
					 \dist xs +\dist xt - 2 \gro yzx 
				\right\} + 4\delta.
			\end{equation*}
		\end{enumerate}
	\end{lemm}
	
	\paragraph{The boundary at infinity.} 
	Let $x$ be a base point of $X$.
	A sequence $(y_n)$ of points of $X$ \emph{converges to infinity} if $\gro {y_n}{y_m}x$ tends to infinity as $n$ and $m$ approach to infinity.
	The set $\mathcal S$ of such sequences is endowed with a binary relation defined as follows.
	Two sequences $(y_n)$ and $(z_n)$ are related if 
	\begin{displaymath}
		\lim_{n \rightarrow + \infty} \gro {y_n}{z_n}x = + \infty.
	\end{displaymath}
	If follows from (\ref{eqn: hyperbolicity condition 1}) that this relation is actually an equivalence relation.
	The \emph{boundary at infinity} of $X$ denoted by $\partial X$ is the quotient of $\mathcal S$ by this relation.
	If the sequence $(y_n)$ is an element in the class of $\xi \in \partial X$ we say that $(y_n)$  \emph{converges} to $\xi$ and  write
	\begin{displaymath}
		\lim_{n \rightarrow + \infty} y_n = \xi.
	\end{displaymath}
	Note that the definition of $\partial X$ does not depend on the base point $x$.
	If $Y$ is a subset of $X$ we denote by $\partial Y$ the set of elements of $\partial X$ which are limits of sequences of points of $Y$.
	Since $X$ is not proper, $\partial Y$ might be empty even though $Y$ is unbounded.
	
	\paragraph{} The Gromov product of three points can be extended to the boundary.
	Let $x\in X$ and $y,z \in X \cup \partial X$.
	We define $\gro yz x$ as the greatest lower bound of 
	\begin{displaymath}
		\liminf_{n\rightarrow + \infty} \gro {y_n}{z_n}x
	\end{displaymath}
	where $(y_n)$ and $(z_n)$ are two sequences which respectively converge to $y$ and $z$.
	This definition coincides with the original one when $y,z \in X$.
	Two points $\xi$ and $\eta$ of $\partial X$ are equal if and only if $\gro \xi \eta x = + \infty$.
	Let $x \in X$.
	Let $(y_n)$ and $(z_n)$ be two sequences of points of $X$ respectively converging to $y$ and $z$ in $X \cup \partial X$.
	It follows from (\ref{eqn: hyperbolicity condition 1}) that
	\begin{equation}
	\label{eqn: estimate gromov product boundary}
		\gro yzx \leq \liminf_{n\rightarrow + \infty} \gro {y_n}{z_n}x \leq \limsup_{n\rightarrow + \infty} \gro {y_n}{z_n}x \leq \gro yzx + k\delta,
	\end{equation}
	where $k$ is the number of points of $\{y,z\}$ that belongs to $\partial X$.
	Moreover, for every $t \in X$, for every $x,y,z \in X \cup \partial X$, the hyperbolicity condition (\ref{eqn: hyperbolicity condition 1}) leads to
	\begin{displaymath}
		\gro xzt \geq \min\left\{\fantomB \gro xyt, \gro yzt \right\} - \delta.
	\end{displaymath}
	The next lemma is an analogue of \autoref{res: metric inequalities} with one point in the boundary of $X$.
	It will be used in situations where the Gromov products $\gro x\xi s$, $\gro x\xi t$ and $\gro y\xi t$ are small.
	
	\begin{lemm}
	\label{res: metric inequalities with boundary}
		Let $\xi \in \partial X$.
		Let $x$, $y$, $s$ and $t$ be four points of $X$.
		We have the following inequalities
		\begin{enumerate}
			\item \label{enu: metric inequalities  with boundary - thin triangle}
			\begin{math}
				\gro x\xi t \leq \max \left\{\fantomB \dist xt - \gro \xi zx , \gro xzt  \right\} + \delta,
			\end{math}
			\item \label{enu: metric inequalities with boundary - two points close to a geodesic}
			\begin{math}
				\dist st \leq \dist{\dist xs}{\dist xt} + 2\max\left\{\gro x\xi s, \gro x\xi t\right\} + 3\delta,
			\end{math}
			\item \label{enu: metric inequalities with boundary - three points}
			The distance $\dist st$ is bounded above by 
			\begin{equation*}
				\max \left\{ \gro x\xi s + \gro y\xi t + 2\delta , \dist xy + \dist{\dist xs}{\dist yt} + 2\max\{ \gro x\xi s ,\gro y\xi t\} \right\} + 2 \delta.
			\end{equation*}
		\end{enumerate}
	\end{lemm}
	
	\begin{proof}
		Points~\ref{enu: metric inequalities  with boundary - thin triangle} and \ref{enu: metric inequalities with boundary - two points close to a geodesic} follow directly from \autoref{res: metric inequalities}  \ref{enu: metric inequalities - thin triangle} and \ref{enu: metric inequalities - two points close to a geodesic} combined with (\ref{eqn: estimate gromov product boundary}).
		Let us focus on Point~\ref{enu: metric inequalities with boundary - three points}.
		By hyperbolicity we have
		\begin{eqnarray}
			\label{eqn: metric inequalities with boundary - x}
			\min \left\{ \gro xts, \gro t\xi s \right\} & \leq & \gro x\xi s + \delta,\\
			\label{eqn: metric inequalities with boundary - y}
			\min \left\{ \gro yst, \gro s\xi t \right\} & \leq & \gro y\xi t + \delta.
		\end{eqnarray}
		Assume that in (\ref{eqn: metric inequalities with boundary - x}) the minimum is achieved by $\gro xts$.
		It follows that 
		\begin{displaymath}
			\dist st \leq \dist xt - \dist xs + 2 \gro x\xi s + 2 \delta.
		\end{displaymath}
		Combined with the triangle inequality we obtain
		\begin{displaymath}
			\dist st \leq \dist xy + \dist{\dist xs}{\dist yt} + 2 \gro x\xi s + 2 \delta.
		\end{displaymath}
		The same kind of argument holds if the minimum in (\ref{eqn: metric inequalities with boundary - y}) is achieved by $\gro yst$.
		Therefore we can now assume that $\gro t\xi s \leq \gro x\xi s + \delta$ and $\gro s\xi t \leq \gro y\xi t + \delta$.
		For every $z \in X$ we have $\dist st = \gro szt + \gro tzs$.
		If follows from (\ref{eqn: estimate gromov product boundary}) that $\dist st \leq \gro s\xi t + \gro t\xi s + 2\delta$.
		Consequently $\dist st \leq \gro x\xi s + \gro y\xi t + 4\delta$.
	\end{proof}
	
	\begin{lemm}
	\label{res: points close to an infinite geodesic}
		Let $x \in X$ and $\xi \in \partial X$.
		For every $l \geq 0$, for every $\eta >0$, there exists a point $y \in X$ such that $\dist xy = l$ and $\gro x\xi y \leq  \delta + \eta$.
	\end{lemm}
	
	\begin{proof}
		Let $l \geq 0$ and $\eta >0$.
		Let $(z_n)$ be a sequence of points of $X$ which converges to $\xi$.
		In particular there exists $N \in \N$ such that for all $n,m \geq N$, $\gro {z_n}{z_m}x \geq l$.
		We choose for $y$ a point of $X$ such that $\dist xy = l $ and $\gro x{z_N}y \leq \eta$.
		By \autoref{res: metric inequalities}~\ref{enu: metric inequalities - thin triangle}, we get for every $n \geq N$,
		\begin{displaymath}
			\gro x{z_n}y \leq \max \left\{ \dist xy - \gro {z_N}{z_n}x, \gro x{z_N}y \right\} + \delta \leq  \gro x{z_N}y + \delta \leq  \delta + \eta.
		\end{displaymath}
		Consequently $\gro x\xi y \leq \delta + \eta$.
	\end{proof}

%
%

\subsection{Quasi-geodesics}
\label{sec: quasi-geodesics}

	\begin{defi}
	\label{def: quasi-geodesic}
		Let $l \geq 0$, $k \geq 1$ and $L\geq 0$.
		Let $f : X_1 \rightarrow X_2$ be a map between two metric spaces $X_1$ and $X_2$.
		We say that $f$ is a \emph{$(k,l)$-quasi-isometry} if for every $x,x' \in X_1$, 
		\begin{equation*}
			k^{-1}\dist  {f(x)}{f(x')}  - l \leq \dist x{x'}\leq k\dist  {f(x)}{f(x')}   + l.
		\end{equation*}
		We say that $f$ is an \emph{$L$-local $(k,l)$-quasi-isometry} if its restriction to any subset of diameter at most $L$ is a $(k,l)$-quasi-isometry.
		Let $I$ be an interval of $\R$.
		A path $\gamma : I \rightarrow X$ that is a $(k,l)$-quasi-isometry is called a \emph{$(k,l)$-quasi-geodesic}.
		Similarly, we define similarly \emph{$L$-local $(k,l)$-quasi-geodesics}.
	\end{defi}
	
	\rems 
	We assumed that our paths are rectifiable and parametrized by arclength.
	Thus a $(k,l)$-quasi-geodesic $\gamma : I \rightarrow X$ satisfies a more accurate property: for every $t,t' \in I$,
	\begin{equation*}
		\dist{\gamma(t)}{\gamma(t')} \leq \dist t{t'} \leq k\dist{\gamma(t)}{\gamma(t')} +l.
	\end{equation*}
	In particular, if $\gamma$ is a $(1,l)$-quasi-geodesic, then for every $t,t',s \in I$, such that $t\leq s\leq t'$, we have $\gro {\gamma(t)}{\gamma(t')}{\gamma(s)}\leq l/2$.
	Since $X$ is a length space, for every $x,x' \in X$, for every $l>0$, there exists a $(1,l)$-quasi-geodesic joining $x$ and $x'$.
	
	\begin{prop}{\rm \cite[Proposition 2.4]{Coulon:2013tx}} \quad
	\label{res: quasi-convexity of quasi-geodesics}
		Let  $\gamma : I \rightarrow X$ be a $(1,l)$-quasi-geodesic of $X$.
		\begin{enumerate}
			\item \label{enu: quasi-convexity of quasi-geodesics - gromov product}
			Let $x$ be a point of $X$ and $p$ an $\eta$-projection of $x$ on $\gamma(I)$.
			For all $y \in \gamma(I)$, $\gro xyp \leq l + \eta + 2\delta$.
			\item \label{enu: quasi-convexity of quasi-geodesics - quasi-convex}
			For every $x \in X$, for every $y,y'$ lying on $\gamma$, we have $\gro y{y'}x - l \leq d(x,\gamma) \leq \gro y{y'}x + l + 3 \delta$.
		\end{enumerate}
	\end{prop}
	
	\paragraph{}Let $\gamma : \R_+ \rightarrow X$ be a $(k,l)$-quasi-geodesic.
	There exists a point $\xi \in \partial X$ such that for every sequence $(t_n)$ diverging to infinity, $\lim_{n \rightarrow + \infty}\gamma(t_n) = \xi$.
	In this situation we consider $\xi$ as an endpoint (at infinity) of $\gamma$ and write $\lim_{t \rightarrow + \infty} \gamma(t) = \xi$.
		
	\paragraph{Stability of quasi-geodesics.}One important feature of hyperbolic spaces is the stability of quasi-geodesic paths recalled below.
	\begin{prop}[Stability of quasi-geodesics]{\rm \cite[Chapitre 3, Th\'eor\`emes 1.2, 1.4 et 3.1]{CooDelPap90}}\quad
	\label{res: stability quasi-geodesic}
		Let $k \geq 1$, $k' >k$ and $l \geq 0$.
		There exist $L$ and $D$ which only depend on $\delta$, $k$, $k'$ and $l$ with the following properties
		\begin{enumerate}
			\item Every $L$-local $(k,l)$-quasi-geodesic is a (global) $(k',l)$-quasi-geodesic.
			\item The Hausdorff distance between two $L$-local $(k,l)$-quasi-geodesics joining the same endpoints (eventually in $\partial X$) is at most $D$.
		\end{enumerate}
	\end{prop}
	
	\paragraph{}
	In this article we are mostly using $L$-local $(1,l)$-quasi-geodesics.
	For these paths one can provide a precise value for $D$ (see next corollary).
	This is not really necessary but will decrease the number of parameters that we have to deal with in all the proofs.
	
	\begin{coro}{\rm \cite[Corollary 2.6]{Coulon:2013tx}} \quad
	\label{res: stability (1,l)-quasi-geodesic}
		Let $l \geq 0$.
		There exists $L=L(l,\delta)$ which only depends on $\delta$ and $l$ with the following properties.
		Let $\gamma$ be an $L$-local $(1,l)$-quasi-geodesic.
		\begin{enumerate}
			\item The path $\gamma$ is a (global) $(2,l)$-quasi-geodesic.
			\item For every $t,t',s \in I$, such that $t \leq s \leq t'$, we have $\gro{\gamma(t)}{\gamma(t')}{\gamma(s)} \leq l/2 + 5 \delta$.
			\item For every $x \in X$, for every $y,y'$ lying on $\gamma$, we have $d(x,\gamma) \leq \gro y{y'}x + l + 8 \delta$.
			\item The Hausdorff distance between $\gamma$ and an other $L$-local $(1,l)$-quasi-geodesic joining the same endpoints (eventually in $\partial X$) is at most $2l+5\delta$.
		\end{enumerate}
	\end{coro}
		
	\rem Using a rescaling argument, one can see that the best value for the parameter $L=L(l,\delta)$ satisfies the following property: for all $l,\delta \geq 0$ and $\lambda >0$, $L(\lambda l, \lambda \delta) = \lambda L(l,\delta)$.
	For the rest of the article we denote by $L_S$ the smallest positive number larger than $500$ such that $L(10^5\delta,\delta) \leq L_S\delta$.
		
	\paragraph{Quasi-rays.} If $X$ is a proper geodesic space, the Azerl\`a-Ascoli Theorem says that given any two distinct points in $X \cup \partial X$ there exists a geodesic joining them. 
	Here, $X$ is not necessarily proper.
	Therefore we substitute this property for the following lemma.
	
	\begin{lemm}
	\label{res: quasi-rays}
		Let $x \in X$ and $\xi \in \partial X$.
		For every $L > 0$, for every $l >0$, there exists an $L$-local $(1, l+10\delta)$-quasi-geodesic joining $x$ to $\xi$.
	\end{lemm}
	
	\begin{proof}
		Let $L \geq L_S\delta$ and $\eta \in (0, \delta)$.
		According to \autoref{res: points close to an infinite geodesic}, for every $n \in \N$, there exists a point $x_n \in X$ such that $\dist x{x_n} = nL$ and $\gro x\xi{x_n} \leq \eta + \delta$.
		By construction $(x_n)$ converges to $\xi$.
		We claim that for every $n \in \N^*$,
		\begin{equation*}
			\dist {x_n}{x_{n-1}} \geq L 
			\quad \text{and} \quad
			\gro {x_{n+1}}{x_{n-1}}{x_n} \leq 2\eta + 5\delta.
		\end{equation*}
		Let $n \in \N^*$.
		First, the triangle inequality gives $\dist{x_n}{x_{n-1}} \geq L$ and $\dist{x_{n+1}}{x_{n-1}} \geq 2L$.
		On the other hand, applying \autoref{res: metric inequalities with boundary}~\ref{enu: metric inequalities with boundary - two points close to a geodesic} we get $\dist{x_n}{x_{n-1}} \leq L + 2\eta + 5\delta$. 
		The claim is a consequence of these inequalities.
		For every $n\in \N$, we choose a $(1, \eta)$-quasi-geodesic $\gamma_n$ joining $x_n$ to $x_{n+1}$.
		We define $\gamma : \R_+ \rightarrow X$ as the concatenation of these paths.
		It follows from the previous inequalities that $\gamma$ is a $L$-local $(1, 8\eta +10\delta)$-quasi-geodesic.
		By choice of $L$, $\gamma$ is also a $(2,8\eta +10\delta)$-quasi-geodesic, thus it has an endpoint at infinity.
		Since $(x_n)$ lies on $\gamma$, this endpoint is $\xi$.
		If $\eta$ is chosen sufficiently small, $\gamma$ is the desired path.
		\end{proof}

%
%

\subsection{Quasi-convex and strongly quasi-convex subsets}

	\begin{defi}
	\label{def: quasi-convex}
		Let $\alpha \geq 0$.
		A subset $Y$ of $X$ is \emph{$\alpha$-quasi-convex} if for every $x \in X$, for every $y,y' \in Y$, $d(x,Y) \leq \gro y{y'}x + \alpha$.
	\end{defi}
	
	\paragraph{}
	Since $X$ is not a geodesic space our definition of quasi-convex slightly differs from the usual one (every geodesic joining two points of $Y$ remains in the $\alpha$-neighborhood of $Y$).
	However if $X$ is geodesic, an $\alpha$-quasi-convex subset in the usual sense is $(\alpha + 4 \delta)$-quasi-convex in our sense and conversely.
	For instance it follows from the four point inequality~(\ref{eqn: hyperbolicity condition 2}) that any ball is $2\delta$-quasi-convex.
	For our purpose we will also need a slightly stronger version of quasi-convexity.
	
	\begin{defi}
		Let $\alpha \geq 0$.
		Let $Y$ be a subset of $X$ connected by rectifiable paths.
		The length metric on $Y$ induced by the restriction of $\distV[X]$ to $Y$ is denoted by $\distV[Y]$.
		We say that $Y$ is \emph{strongly quasi-convex} if $Y$ is $2 \delta$-quasi-convex and for every $y,y' \in Y$,
		\begin{displaymath}
			\dist[X]y{y'} \leq \dist[Y]y{y'} \leq \dist[X]y{y'} + 8\delta.
		\end{displaymath}
	\end{defi}
	
	\rem The first inequality is just a consequence of the definition of $\distV[Y]$.
	The second one gives a way to compare $Y$ seen as a length space with $X$.
	
	\begin{lemm}{\rm \cite[Chap. 10, Prop. 1.2]{CooDelPap90}} \quad
	\label{res: neighborhood quasi-convex}
		Let $Y$ be an $\alpha$-quasi-convex subset of $X$.
		For every $A \geq \alpha$, the $A$-neighborhood of $Y$ is $2 \delta$-quasi-convex.
	\end{lemm}

	\begin{lemm}
	\label{res: strong neighborhood of quasi-convex}
		Let $Y$ be an $\alpha$-quasi-convex subset of $X$.
		Let $A > \alpha + 2\delta$.
		The \emph{open} $A$-neighborhood of $Y$ is strongly quasi-convex.
	\end{lemm}
		
	\begin{proof}
		Let us denote by $Z$ the \emph{open} $A$-neighborhood of $Y$.
		Let $z_1$ and $z_2$ be two points of $Z$ and $x$ a point of $X$.
		By definition there exist $y_1,y_2 \in Y$ such that $\dist {y_1}{z_1}, \dist {y_2}{z_2} < A$.
		It follows from the four point inequality~(\ref{eqn: hyperbolicity condition 1}) that 
		\begin{equation*}
			\min\left\{\gro {z_1}{y_1}x, \gro{y_1}{y_2}x , \gro {y_2}{z_2}x \right\} \leq \gro {z_1}{z_2}x +2\delta.
		\end{equation*}
		Since $Y$ is $\alpha$-quasi-convex, $d(x,Y) \leq \gro{y_1}{y_2}x + \alpha < \gro{y_1}{y_2}x + A$.
		On the other hand, the triangle inequality gives
		\begin{equation*}
			\gro {z_1}{y_1}x \geq \dist x{y_1} - \dist {y_1}{z_1} > d(x,Y) - A.
		\end{equation*}
		In the same way $\gro {z_2}{y_2}x > d(x,Y) - A$.
		Hence $d(x,Y) < \gro {z_1}{z_2}x + A +2\delta$.
		However $X$ is a length-space.
		Thus
		\begin{equation*}
		\label{eqn: strong neighborhood of quasi-convex - qc}
			d(x,Z) \leq \gro {z_1}{z_2}x + 2\delta.
		\end{equation*}
		Consequently $Z$ is $2\delta$-quasi-convex.

		\paragraph{}
		Let $\eta >0$ such that $\dist {y_1}{z_1} + \eta <A$, $\dist {y_2}{z_2} + \eta <A$ and $A \geq \alpha + 2\delta + \eta$.
		We denote by $\gamma_1$ a $(1,\eta)$-quasi-geodesic joining $y_1$ to $z_1$.
		By choice of $\eta$, this path is contained in $Z$.
		We denote by $x_1$ a point of $\gamma_1$ such that $\dist{x_1}{y_1} = \min\{A-2\delta-\eta, \dist{z_1}{y_1}\}$.
		In particular $\dist{z_1}{x_1} \leq 2\delta + 2\eta$.
		We construct in the same way a $(1,\eta)$-quasi-geodesic $\gamma_2$ joining $y_2$ to $z_2$ and a point $x_2$ lying on $\gamma_2$.
		Let $\gamma$ be a $(1,\eta)$-quasi-geodesic joining $x_1$ to $x_2$.
		Let $p$ be a point lying on $\gamma$.
		By hyperbolicity we get
		\begin{equation}
		\label{eqn: strong neighborhood of quasi-convex - hyperbolicity}
			\min\{\gro {x_1}{y_1}p, \gro{y_1}{y_2}p, \gro{y_2}{x_2}p \} 
			\leq \gro {x_1}{x_2}p + 2\delta 
			\leq \eta/2 + 2\delta.
		\end{equation}
		Since $Y$ is $\alpha$-quasi-convex, we have 
		\begin{equation}
		\label{eqn: strong neighborhood of quasi-convex - Y qc}
			d(p,Y) \leq \gro {y_1}{y_2}p + \alpha \leq \gro {y_1}{y_2}p + A - 2 \delta - \eta
		\end{equation}
		On the other hand, the triangle inequality yields
		\begin{equation}
		\label{eqn: strong neighborhood of quasi-convex - sides}
			d(p,Y) \leq \dist {y_1}p \leq \dist {x_1}{y_1} + \gro {x_1}{y_1}p \leq \gro {x_1}{y_1}p + A - 2 \delta - \eta.
		\end{equation}
		The same inequality holds with $\gro{x_2}{y_2}p$.
		Combining (\ref{eqn: strong neighborhood of quasi-convex - hyperbolicity})-(\ref{eqn: strong neighborhood of quasi-convex - sides}) we get  $d(p,Y) < A$.
		In particular, $\gamma$ is contained in $Z$.
		So are $\gamma_1$ and $\gamma_2$.
		Recall that $\dist{z_1}{x_1} \leq 2\delta + 2\eta$ and $\dist{z_2}{x_2} \leq 2\delta + 2\eta$.
		Hence there is a path of length at most $L(\gamma) + 4\delta + 5\eta$ joining $z_1$ to $z_2$ and contained in $Z$.
		By the triangle inequality $L(\gamma) \leq \dist{z_1}{z_2} + 4\delta + 5 \eta$.
		It follows that 
		\begin{equation*}
		\label{eqn: strong neighborhood of quasi-convex - sqc}
			\dist[Z]{z_1}{z_2} \leq \dist[X]{z_1}{z_2} + 8\delta + 10 \eta.
		\end{equation*}
		This inequality holds for every sufficiently small $\eta$, hence $Z$ is strongly quasi-convex.
	\end{proof}
			
	\begin{lemm}[Projection on a quasi-convex]{\rm \cite[Chapitre 10, Proposition 2.1]{CooDelPap90}}\quad
	\label{res: proj quasi-convex}
		Let $Y$ be an $\alpha$-quasi-convex subset of $X$. 
		\begin{enumerate}
			\item \label{enu: proj quasi-convex - gromov product}
			If $p$ is an $\eta$-projection of $x \in X$ on $Y$, then for all $y \in Y$, $\gro xyp \leq \alpha + \eta$.
			\item \label{enu: proj quasi-convex - distance two points }
			If $p$ (\resp $p'$) is an $\eta$-projection (\resp $\eta'$-projection) of $x \in X$ (\resp $x' \in X$) on $Y$, then 
			\begin{displaymath}
				\dist p{p'} \leq \max \left\{\fantomB \dist x{x'}-\dist xp - \dist {x'}{p'} +2\epsilon, \epsilon \right\},
			\end{displaymath}
			where $\epsilon = 2 \alpha + \eta + \eta' + \delta$.
		\end{enumerate}
	\end{lemm}
	
	\paragraph{}
	The next two lemmas respectively generalize Lemma~2.12 and Lemma~2.13 of \cite{Coulon:2013tx} where they are stated for the intersection of tow quasi-convex subsets. 
	However the proofs work exactly in the same way and are left to the reader.
	
	\begin{lemm}{(\rm compare \cite[Lemma~2.12]{Coulon:2013tx})}\quad
	\label{res: intersection of quasi-convex}
		Let $Y_1, \dots , Y_m$ be a collection of subsets of $X$ such that for every $j \in \intvald 1m$, $Y_j$ is $\alpha_j$-quasi-convex.
		We denote by $Y$ the intersection
		\begin{equation*}
			Y = Y_1^{+ \alpha_1 + 3\delta} \cap \dotsc \cap Y_m^{+ \alpha_m + 3\delta}
		\end{equation*}
	It is a $7\delta$-quasi-convex subset of $X$. 
	\end{lemm}

	\begin{lemm}{\rm (compare \cite[Lemma~2.13]{Coulon:2013tx})}\quad
	\label{res: intersection of thickened quasi-convex}
		Let $Y_1, \dots , Y_m$ be a collection of subsets of $X$ such that for every $j \in \intvald 1m$, $Y_j$ is $\alpha_j$-quasi-convex.
		For all $A \geq 0$ we have
		\begin{displaymath}
			\diam \left( Y_1^{+A} \cap \dotsc \cap Y_m^{+A} \right) 
			\leq \diam \left( Y_1^{+\alpha_1+3\delta} \cap \dotsc \cap Y_m^{+\alpha_m+3\delta} \right) +2A + 4\delta.
		\end{displaymath}
	\end{lemm}
	

	\begin{defi}
	\label{def: hull}
		Let $Y$ be a subset of $X$.
		The \emph{hull} of $Y$, denoted by $\hull Y$, is the union of all $(1, \delta)$-quasi-geodesics joining two points of $Y$.
	\end{defi}

	\begin{lemm}{\rm \cite[Lemma 2.15]{Coulon:2013tx}} \quad
	\label{res: hull quasi-convex}
		Let $Y$ be a subset of $X$. 
		The hull of $Y$ is $6\delta$-quasi-convex.
	\end{lemm}

%
%
	
\section{Group acting on a hyperbolic space}
\label{sec: group acting on a hyperbolic space}

%
%
		
\subsection{Classification of isometries}

	Let $x$ be a point of $X$.
	An isometry $g$ of $X$ is either
	\begin{itemize}
		\item \emph{elliptic}, i.e. the orbit $\langle g \rangle \cdot x$ is bounded,
		\item \emph{loxodromic}, i.e. the map from $\Z$ to $X$ that sends $m$ to $g^m x$ is a quasi-isometry,
		\item or \emph{parabolic}, i.e. it is neither loxodromic or elliptic.		
	\end{itemize}
	Note that these definitions do not depend on the point $x$.
	In order to measure the action of $g$ on $X$, we use two translation lengths.
	By the \emph{translation length} $\len[espace=X]g$ (or simply $\len g$) we mean 
	\begin{displaymath}
		\len[espace=X] g = \inf_{x \in X} \dist {gx}x.
	\end{displaymath}
	The \emph{asymptotic translation length} $\len[stable, espace=X] g$ (or simply $\len[stable]g$) is
	\begin{displaymath}
		\len[espace=X,stable] g = \lim_{n \rightarrow + \infty} \frac 1n \dist{g^nx}x.
	\end{displaymath}
	The isometry $g$ is loxodromic if and only if its asymptotic translation length is positive \cite[Chapitre 10, Proposition 6.3]{CooDelPap90}.
	These two lengths are related as follows.
	
	\begin{prop}{\rm \cite[Chapitre 10, Proposition 6.4]{CooDelPap90}}\quad
	\label{res: translation lengths}
		Let $g$ be an isometry of $X$. 
		Its translation lengths satisfy
		\begin{equation*}
			\len[stable]g \leq \len g \leq \len[stable] g + 32\delta
		\end{equation*}
	\end{prop}
	
	\paragraph{}
	By construction, the group of isometries of $X$ acts on the boundary at infinity $\partial X$ of $X$.
	The different types of isometries of $X$ can be characterized in terms of accumulation points in $\partial X$.
	Given a group $G$ acting by isometries on $X$, we denote by $\partial G$ the set of accumulations points of $G\cdot x$ in $\partial X$.
	Note that it does not depend on $x \in X$.
	It is also $G$-invariant.
	If $g$ is a loxodromic isometry of $X$ then $\partial \langle g \rangle$ contains exactly two elements:
	\begin{displaymath}
		g^-  = \lim_{n \rightarrow - \infty} g^nx \text{ and } g^+  = \lim_{n \rightarrow + \infty} g^nx
	\end{displaymath}
	They are the only points of $\partial X$ fixed by $g$, \cite[Chapitre 10, Proposition 6.6]{CooDelPap90}.
	\begin{lemm}
		Let $g$ be an isometry of $X$.
		Let $l >0$.
		There exist $T \in \R$ with $\len g \leq T < \len g + l$ and a $T$-local $(1,l)$-quasi-geodesic $\gamma : \R \rightarrow X$ such that for every $t \in \R$, $\gamma(t +T) = g\gamma(t)$.
	\end{lemm}
	
	\rem We call such a path an \emph{$l$-nerve of $g$} and $T$ its \emph{fundamental length}.
	This kind of path will be used to simplify some proofs.
	Indeed if $\len g > L_S\delta$ (in particular $g$ is loxodromic) and $l \leq 20\delta$, by stability of quasi-geodesics $\gamma$ is actually $(l + 8\delta)$-quasi-convex.
	Moreover it joins $g^-$ to $g^+$.
	Thus it provides a $g$-invariant line than can advantageously be used as a substitution for an axis or a cylinder (see \autoref{def: axes} and \autoref{def: cylinder loxodromic element}).

	\begin{proof}
		Let $\eta, \eta' >0$.
		There exists $x \in X$ such that $\dist {gx}x < \len g + \eta$.
		Let $\gamma : \intval 0T \rightarrow X$ be a $(1, \eta')$-quasi-geodesic joining $x$ to $gx$.
		In particular $\len g \leq T < \len g + \eta + \eta'$.
		We extend $\gamma$ into a path $\gamma : \R \rightarrow X$ in the following way: for every $t \in [0,T)$, for every $m \in \Z$, $\gamma(t+mT) = g^m\gamma(t)$.
		It turns out that $\gamma$ is a $T$-local $(1,\eta+\eta')$-quasi-geodesic.
		Thus if $\eta$ and $\eta'$ are chosen sufficiently small then $T$ and $\gamma$ satisfy the statement of the lemma.
	\end{proof}

	\paragraph{}
	Recall that we did not assume that $X$ was proper.
	Therefore there might exist unbounded subsets of $Y$ of $X$ such that $\partial Y$ is empty.
	However this pathology does not happen if $Y$ is the orbit of a group $G$.
	To prove this fact we need the following lemma.
	
	\begin{lemm}{\rm \cite[Chapitre 9, Lemme 2.3]{CooDelPap90}}\quad
	\label{res: constructing loxodromic isometry}
		Let $g$ and $h$ be two isometries of $X$ which are not loxodromic.
		If there exists a point $x \in X$ such that $\dist {gx}x \geq 2\gro {gx}{hx}x + 6 \delta$ and  $\dist {hx}x \geq 2\gro {gx}{hx}x + 6 \delta$ then $g^{-1}h$ is loxodromic.
	\end{lemm}
	
	\begin{prop}
	\label{res: non-empty limit set}
		Let $G$ be a group acting by isometries on $X$.
		Either one (and thus every) orbit of $G$ is bounded or $\partial G$ is non-empty.
	\end{prop}

	\begin{proof}
		Let $x$ be point of $X$.
		Assume that, contrary to our claim, $G$ is unbounded and $\partial G$ is empty.
		In particular, $G$ cannot contain a loxodromic element.
		On the other hand, there exists a sequence $(g_n)$ of elements of $G$ such that $\lim_{n \rightarrow + \infty} \dist {g_nx}x = + \infty$ and $\gro{g_nx}{g_mx}x$, $n\neq m$ is bounded.
		It follows from \autoref{res: constructing loxodromic isometry} that if $n$ and $m$ are sufficiently large distinct integers, then $g_n^{-1}g_m$ is a loxodromic element of $G$.
		Contradiction.
	\end{proof}
	
	\begin{prop}
	\label{res: two boundary points give loxodromic isometry}
		Let $G$ be a group acting by isometries on $X$.
		If $\partial G$ has at least two points then $G$ contains a loxodromic isometry.
	\end{prop}
	
	\begin{proof}
		Let us denote by $\xi$ and $\eta$ two distinct points of $\partial G$.
		They are respectively limits of two sequences $(g_n x)$ and $(h_n x)$ where $g_n$ and $h_n$ belong to $G$.
		Thus we have the followings.
		\begin{itemize}
			\item $\displaystyle \lim_{n \rightarrow + \infty} \dist{g_nx}x = + \infty$ and $\displaystyle  \lim_{n \rightarrow + \infty} \dist{h_nx}x = + \infty$
			\item $\displaystyle \limsup_{n \rightarrow + \infty} \gro{g_nx}{h_nx}x \leq \gro \xi\eta x + 2 \delta < + \infty$
		\end{itemize}
		In particular, there exists $n \in \N$ such that $\dist{g_nx}x \geq 2\gro {g_nx}{h_nx} x + 6\delta$ and $\dist{h_nx}x \geq 2\gro {g_nx}{h_nx} x + 6\delta$.
		If $g_n$ and $h_n$ are not already loxodromic, then by \autoref{res: constructing loxodromic isometry}, $g_n^{-1}h_n$ is.
	\end{proof}
	
	\begin{coro}
	\label{res: parabolic have exactly one limit point}
		An isometry $g$ of $X$ is parabolic if and only if $\partial \langle g \rangle$ has exactly one point.
	\end{coro}

	\begin{lemm}
	\label{res: three boundary points gives two distinct loxodromic isometries}
		Let $G$ be a group acting by isometries on $X$.
		If $\partial G$ has at least three points then $G$ contains two loxodromic isometries $g$ and $h$ such that $\{g^-, g^+\} \neq  \{h^-, h^+\}$.
	\end{lemm}

	\begin{proof}
		By \autoref{res: two boundary points give loxodromic isometry}, $G$ contains a loxodromic isometry $g$.
		We denote by $g^-$ and $g^+$ the points of $\partial X$ fixed by $g$. 
		They belong to $\partial G$.
		According to the stability of quasi-geodesics (\autoref{res: stability (1,l)-quasi-geodesic}) the Hausdorff distance between two $L_S\delta$-local $(1, \delta)$-quasi-geodesics with the same endpoints is at most $7\delta$ .
		We denote by $Y$ the union of all $L_S\delta$-local $(1, \delta)$-quasi-geodesics joining $g^-$ and $g^+$.
		This set is non-empty (it contains a nerve of a large power of $g$).
		Moreover $\partial Y = \{g^-, g^+\}$.
		We assume now that for every $u \in G$ we have $u \{g^-, g^+\} = \{g^-, g^+\}$.
		It follows that $Y$ is $G$-invariant.
		Thus every element of $\partial G$ is the limit of a sequence of points of $Y$.
		In other words $\partial G$ is contained in $\{g^-, g^+\}$.
		Contradiction.
		Hence there exists $u \in G$ such that $u \{g^-, g^+\} \neq \{g^-, g^+\}$.
		The isometries $g$ and $h = ugu^{-1}$ satisfy the conclusion of the lemma.
	\end{proof}
	
%
%
	
\subsection{Axis of an isometry}

	\begin{lemm}{\rm \cite[Lemma 2.22]{Coulon:2013tx}} \quad
	\label{res: quasi-convexity distance isometry}
		Let $x$, $x'$ and $y$ be three points of $X$.
		Let $g$ be an isometry of $X$.
		Then $\dist {gy}y \leq \max\left\{ \dist {gx}x, \dist {gx'}{x'} \right\} + 2 \gro x{x'}y  + 6 \delta$.
	\end{lemm}
	
	\begin{defi}
	\label{def: axes}
		Let $g$ be an isometry of $X$.
		The \emph{axis} of $g$ denoted by $A_g$ is the set of points $x \in X$ such that  $\dist {gx}x  < \len g + 8 \delta$.
	\end{defi}
	
	\rems Note that we do not require $g$ to be loxodromic.
	This definition works also for parabolic or elliptic isometries.
	For every $l \in (0,4 \delta)$, every $l$-nerve of $g$ is contained in $A_g$.
	On the other hand for every $x \in A_g$ there is a $16 \delta$-nerve of $g$ going through $x$.
		
	\begin{prop}{\rm \cite[Proposition 2.24]{Coulon:2013tx}} \quad
	\label{res: axes is quasi-convex}
	Let $g$ be an isometry of $X$.
	Let $x$ be a point of $X$.
	\begin{enumerate}
		\item \label{enu: axes - min deplacement}
		$\dist {gx}x \geq 2 d(x, A_g) + \len g - 6\delta$,
		\item \label{enu: axes - deplacement donne distance a l'axe}
		if $\dist{gx}x \leq \len g + A$, then $d(x,A_g) \leq A/2+3\delta$,
		\item \label{enu: axes - quasi-convex}
		$A_g$ is $10\delta$-quasi-convex. 
	\end{enumerate}
	\end{prop}

	\begin{defi}
	\label{def: cylinder loxodromic element}
		Let $g$ be a loxodromic isometry of $X$.
		We denote by $\Gamma_g$ the union of all $L_S\delta$-local $(1,\delta)$-quasi-geodesics joining $g^-$ to $g^+$.
		The \emph{cylinder} of $g$, denoted by $Y_g$, is the \emph{open} $20\delta$-neighborhood of $\Gamma_g$.
	\end{defi}

	\begin{lemm}
	\label{res: cylinder strongly quasi-convex}
		Let $g$ be a loxodromic isometry of $X$.
		The cylinder of $g$ is strongly quasi-convex.
	\end{lemm}
	
	\begin{proof}
		According to \autoref{res: strong neighborhood of quasi-convex}, it is sufficient to prove that the union $\Gamma_g$ of all $L_S\delta$-local $(1, \delta)$-quasi-geodesic joining $g^-$ to $g^+$ is $16\delta$-quasi-convex.
		Let $y,y' \in \Gamma_g$ and $x \in X$.
		By definition there exist $\gamma$ and $\gamma'$ two $L_S\delta$-local $(1,\delta)$-quasi-geodesics joining $g^-$ to $g^+$ such that $y$ and $y'$ lie respectively on $\gamma$ and $\gamma'$.
		We denote by $p$ a projection of $y'$ on $\gamma$.
		By stability of quasi-geodesic, the Hausdorff distance between $\gamma$ and $\gamma'$ is at most $7\delta$.
		Thus $\dist {y'}p \leq 7\delta$.
		As an $L_S\delta$-local $(1, \delta)$-quasi-geodesic $\gamma$ is $9\delta$-quasi-convex, thus
		\begin{equation*}
			d(x, \Gamma_g) \leq d(x, \gamma) \leq \gro ypx  + 9\delta \leq \gro y{y'}x + 16\delta.
		\end{equation*}
		Consequently $\Gamma_g$ is $16\delta$-quasi-convex.
	\end{proof}

	\begin{lemm}{\rm \cite[Lemma 2.27]{Coulon:2013tx}} \quad
	\label{res: Yg in quasi-convex g-invariant}
		Let $g$ be a loxodromic isometry of $X$.
		Let $Y$ be a $g$-invariant $\alpha$-quasi-convex subset of $X$.
		Then the cylinder $Y_g$ is contained in the $(\alpha+28\delta)$-neighborhood of $Y$.
		In particular $Y_g$ is contained in the $38\delta$-neighborhood of $A_g$.
	\end{lemm}

	\begin{lemm}{\rm \cite[Lemma 2.28]{Coulon:2013tx}} \quad
	\label{res: Ag in neighborhood of a nerve}
		Let $g$ be an isometry of $X$ such that $\len g > L_S\delta$.
		Let $l \in \intval 0\delta$.
		Let $\gamma$ be an $L_S\delta$-local $(1,l)$-quasi-geodesic of $X$ joining $g^-$ to $g^+$.
		Then $A_g$ is contained in the $10\delta$-neighborhood of $\gamma$.
		In particular $A_g$ is contained in $Y_g$ and in the $10 \delta$-neighborhood of every $l$-nerve of $g$.
	\end{lemm}
	

	The next lemma explains the following fact.
	Let $g$ be a loxodromic isometry of $X$.
	A quasi-geodesic contained in the neighborhood of the axis of $g$ almost behaves like a nerve of $g$.

	\begin{lemm}{\rm \cite[Lemma 2.29]{Coulon:2013tx}} \quad
	\label{res : quasi-geodesic behaving like a nerve}
		Let $g$ be an isometry of $X$ such that $\len g > L_S \delta$.
		Let $l \in \intval 0 \delta$ and $\gamma : \intval ab \rightarrow X$ be a $L_S\delta$-local $(1,l)$-quasi-geodesic contained in the $C$-neighborhood of $A_g$.
		Then there exists $\epsilon \in \{ \pm 1 \}$ such that for every $s \in \intval ab$  if $s \leq b - \len g$ then 
		\begin{displaymath}
			\dist {g^\epsilon\gamma(s)}{\gamma(s+ \len g)} \leq 4C + 4l + 94\delta.
		\end{displaymath}
	\end{lemm}

\subsection{Weakly properly discontinuous action}

\paragraph{}
From now on we fix a group $G$ acting by isometries of $X$.
Recall that we do not require $X$ to be proper.
Similarly we do not make for the moment any assumption on the action of $G$.
In particular the action of $G$ on $X$ is not necessarily proper.
Instead we use a weak notion of properness introduced by M.~Bestvina and K.~Fujiwara in \cite{Bestvina:2002dr}.

	\begin{defi}
	\label{def: WDP}
		A loxodromic element $g$ of $G$ satisfies the \emph{weak proper discontinuity property} (WPD property) if  for every $x \in X$, for every $l \geq 0$, there exists $n \in \N$ such that the set of elements $u \in G$ satisfying $\dist {ux}x \leq l$ and $\dist {ug^nx}{g^nx} \leq l$ is finite.
		The action of $G$ on $X$ is said to be \emph{weakly properly discontinuous} (WPD) if every loxodromic element of $G$ satisfies the WPD property.
	\end{defi}
	
	\paragraph{}Here the space $X$ is hyperbolic.
	In this situation the WPD property follows from a local condition (see \autoref{res: WPD local implies WPD global}).
	Before proving this statement, we start with the following lemma.
	
	\begin{lemm}
	\label{res: WPD for every exponent}
		Let $g$ be a loxodromic element of $G$.
		Let $l \geq 0$.
		Assume that there exist $y,y' \in Y_g$  such that the set of elements $u \in G$ satisfying $\dist {uy}y \leq l+110\delta$ and $\dist{uy'}{y'}\leq l+110\delta$ is finite.
		Then there exists $n_0$ such that for every $x \in X$, for every $n \geq n_0$, the set of elements $u \in G$ satisfying $\dist {ux}x \leq l$ and $\dist {ug^nx}{g^nx} \leq l$ is finite.
	\end{lemm}

	\begin{proof}
	We write $S$ for the set of elements $u \in G$ satisfying $\dist {uy}y \leq l + 110\delta$ and $\dist{uy'}{y'} \leq l + 110\delta$.
	Since $g$ is loxodromic, there exists $k \in \N$ such that $k \len[stable] g > L_S \delta$.
	We denote by $\gamma : \R \rightarrow X$ a $\delta$-nerve of $g^k$ and $T$ its fundamental length.
	By stability of quasi-geodesics $Y_g$ is contained in the $27\delta$-neighborhood of $\gamma$.
	Therefore there exist $q = \gamma(s)$ and $q' = \gamma(s')$ such that $\dist yq \leq 27\delta$ and $\dist{y'}{q'} \leq 27\delta$.
	We can always assume that $s \leq s'$.
	We choose for $n_0$ an integer such that $n_0\len[stable] g \geq \dist{s'}s + T + 73\delta$.
	
	\paragraph{}
	Let $x$ be a point of $X$ and $n \geq n_0$ an integer.
	We denote by $p$ and $r$  respective projections of $x$ and $g^nx$ on $\gamma$.
	Without loss of generality we can assume that $p = \gamma(0)$.
	We write $r = \gamma(t)$.
	Let $r'$ be a projection of $g^np$ on $\gamma$ (see \autoref{fig: wpd1}).
	\begin{figure}[htbp]
	\centering
		\includegraphics[width=0.9\textwidth]{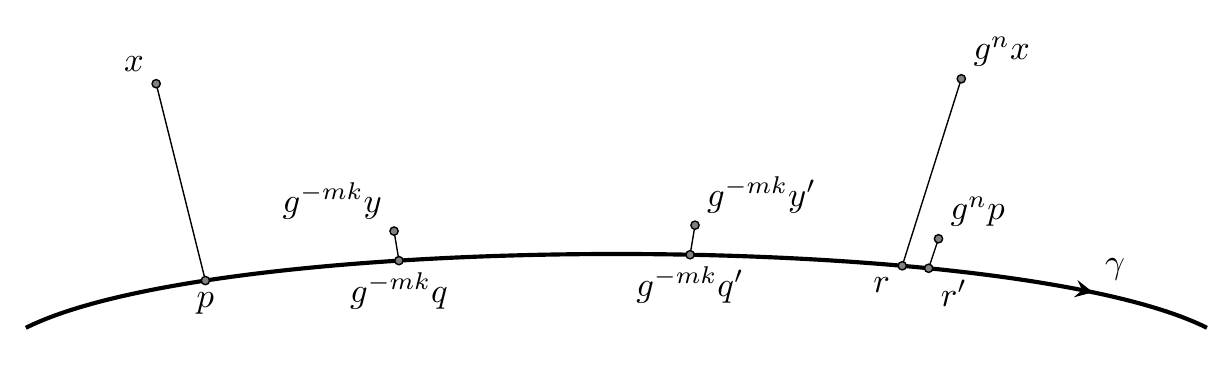}
	\caption{Projections on the $\delta$-nerve $\gamma$}
	\label{fig: wpd1}
	\end{figure}
	By stability of quasi-geodesics, the Hausdorff distance between $\gamma$ and $g^n\gamma$ is at most $7\delta$, thus $\dist {g^np}{r'} \leq 7\delta$.
	Moreover $r'$ is a $14\delta$-projection of $g^nx$ on $\gamma$.
	It follows from the projection on quasi-convex subsets that $\dist r{r'} \leq 66\delta$.
	Consequently $t \geq 0$ and 
	\begin{equation*}
		t \geq \dist rp \geq \dist{g^np}p - 73\delta \geq n \len[stable]g - 73\delta \geq \dist{s'}s + T
	\end{equation*}
	We put $m = \lfloor s/T\rfloor$.
	In particular $0 \leq s - mT \leq s' - mT \leq t$.
	Recall that $\gamma$ is a $\delta$-nerve of $g^k$, hence $g^{-mk}q = \gamma(s-mT)$ and $g^{-mk}q' = \gamma(s'-mT)$ are two points lying on $\gamma$ between $p$ and $r$.
	Using projection on quasi-convex we get 
	\begin{equation*}
		\gro x{g^nx}{g^{-mk}q} \leq 25\delta \quad \text{and} \quad \gro x{g^nx}{g^{-mk}q'} \leq 25\delta.
	\end{equation*}
	
	\paragraph{}
	Let $u \in G$ such that $\dist{ux}x \leq l$ and $\dist{ug^nx}{g^nx} \leq l$.
	\autoref{res: quasi-convexity distance isometry} yields $\dist{ug^{-mk}q}{g^{-mk}q} \leq l + 56\delta$.
	Consequently $\dist{ug^{-mk}y}{g^{-mk}y} \leq l + 110\delta$.
	Similarly we get $\dist{ug^{-mk}y'}{g^{-mk}y'} \leq l + 110\delta$.
	In other words $ug^{-mk}$ belongs to $S$.
	Thus there is only finitely many $u \in G$ such that $\dist{ux}x \leq l$ and $\dist{ug^nx}{g^nx} \leq l$.	
	\end{proof}

	\begin{prop}
	\label{res: WPD local implies WPD global}
		Let $g$ be a loxodromic element of $G$.
		The isometry $g$ satisfies the WPD property if and only if there exist $y,y' \in Y_g$ such that the set of elements $u \in G$ satisfying $\dist {uy}y \leq 486\delta$ and $\dist {uy'}{y'} \leq 486\delta$ is finite.
	\end{prop}
	
	\rem It follows in particular that a loxodromic element $g$ satisfies the WPD property if and only if for every $n \in \N^*$ so does $g^n$.
	The proof follows the idea provided by F.~Dahmani, V.~Guirardel and D.~Osin in \cite{Dahmani:2011vu} for the case of an acylindrical action.
	
	\begin{proof}
		Assume first that $g$ satisfies the WPD property.
		Fix a point $y$ in $Y_g$.
		By assumption there exists $n \in \N$ such that the set of elements $u \in G$ satisfying $\dist{uy}y \leq 486\delta$ and $\dist{ug^ny}{g^ny} \leq 486\delta$ is finite.
		Put $y' = g^ny$. 
		Since $Y_g$ is $g$-invariant it is a point of $Y_g$.
		Consequently $y$ and $y'$ satisfy the the statement of the proposition.
		
		\paragraph{}
		Assume now that there exist $y, y' \in Y_g$ such that the set of elements $u \in G$ satisfying $\dist {uy}y \leq 496\delta$ and $\dist {uy'}{y'} \leq 496\delta$ is finite.
		Let $x \in X$ and $l \geq 0$.
		The element $g$ being loxodromic there exists $k \in \N$ such that $k \len[stable] g > \max\{L_S\delta,  l + 46\delta\}$.
		Let $\gamma$ be a $\delta$-nerve of $g^k$ and $T$ its fundamental length.
		We denote by $p$ a projection of $x$ on $\gamma$.
		For simplicity of notation we put $q = g^kp$ (which also lies on $\gamma$).
		According to \autoref{res: WPD for every exponent}, there exists $n_0 \in \N^*$ such that for every integer $n \geq n_0$ the set of elements $u \in G$ satisfying $\dist {uq}q \leq 376\delta$ and $\dist{ug^{nk}q}{g^{nk}q} \leq 376\delta$ is finite.
		We put $m = n_0k$ and $n = (n_0+2)k$.
		
		\paragraph{}
		We denote by $S$ the set of elements $u \in G$ such that $\dist {ux}x \leq l$ and $\dist {ug^nx}{g^nx} \leq l$.
		We want to prove that $S$ is finite.
		Put $N = \lceil (\dist xq+l)/\delta\rceil$.
		For every integer $i \in \intvald 0N$ we denote by $x_i$ a point of $X$ such that $\dist x{x_i} = i\delta$ and $\gro x{g^nx}{x_i} \leq \delta$ (see \autoref{fig: wpd2}).	
		\begin{figure}[htbp]
		\centering
			\includegraphics[width=0.9\textwidth]{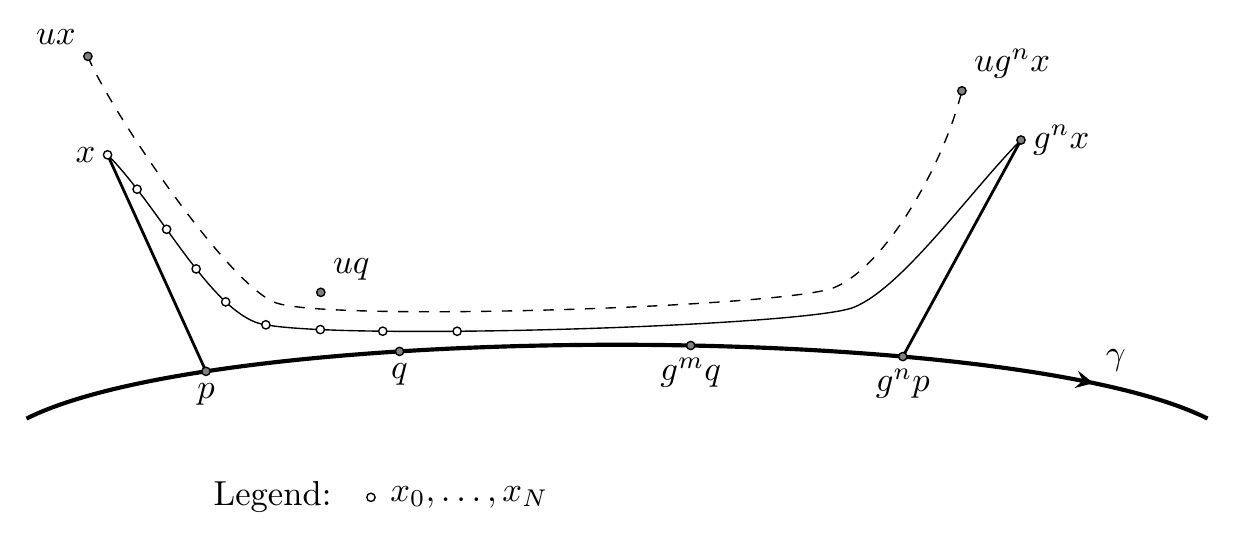}
		\caption{The points $x_0, \dots , x_N$}
		\label{fig: wpd2}
		\end{figure}	
		Such points exist because $\dist{g^nx}x \geq \dist xq + l$.
		Let $u \in S$.
		It follows from the projection on quasi-convex that $\gro x{g^nx}q \leq 25\delta$ and $\gro q{g^nx}{g^mq} \leq 15\delta$ whereas $\dist xq$ and $\dist{g^nx}q$ are larger that $l + 27\delta$.
		By hyperbolicity, we have
		\begin{equation*}
			\min\left\{\dist xq  - \dist x{ux}, \gro x{g^nx}{uq}, \dist{g^nx}q - \dist{g^nx}{ug^nx} \right\} \leq \gro {ux}{ug^nx}{uq} + 2 \delta \leq 27\delta.
		\end{equation*}
		Therefore we get $\gro x{g^nx}{uq} \leq 27\delta$.
		By \autoref{res: metric inequalities}~\ref{enu: metric inequalities - two points close to a geodesic}, $\dist {uq}{x_i} \leq \dist{\dist x{uq}}{i\delta} + 56\delta$.
		However, by triangle inequality $\dist x{uq} \leq \dist xq +l \leq N\delta$.
		Thus there exists $i \in \intvald 0N$ such that $\dist {uq}{x_i} \leq 57\delta$.
		Consequently there is $i \in \intvald 0N$ and a subset $S_i$ of $S$ such that for every $u \in S_i$, $\dist {uq}{x_i} \leq 57\delta $ and $\# S\leq (N+1)\#S_i$ (where $\#S$ denotes the cardinality -- possibly infinite -- of $S$).
		
		\paragraph{}
		Fix now $u_0 \in S_i$.
		Let $v \in u_0^{-1}S_i$.
		By construction $\dist{vq}q \leq 114\delta$ and $\dist{vg^nx}{g^nx} \leq 2l$.
		It follows from the triangle inequality that 
		\begin{equation*}
			\gro q{vg^nx}{vg^mq} \leq \gro {vq}{vg^nx}{vg^mq} + \dist {vq}q \leq 129\delta.
		\end{equation*}
		Applying \autoref{res: metric inequalities}~\ref{enu: metric inequalities - comparison tripod} in the ``triangle'' $[q,g^nx,vg^nx]$ we obtain $\dist {vg^mq}{g^mq} \leq 376\delta$.
		Consequently for every $v \in u_0^{-1}S_i$, $\dist {vq}q \leq 114\delta$ and $\dist {vg^mq}{g^mq} \leq 376\delta$.
		It follows from the definition of $m$ that $S_i$ is finite.
		However $S_i$ has been build in such a way that $\# S \leq (N+1)\#S_i$, therefore $S$ is finite as well, which complete the proof.
	\end{proof}
	
	\paragraph{}From now on we assume that the action of $G$ on $X$ is WPD.
	
	\begin{lemm}
	\label{res: WPD and fixed point in the boundary}
		Let $g$ be a loxodromic element of $G$.
		Let $x \in X$ and $l \geq 0$.
		The set of elements $u \in G$ such that $\dist {ux}x \leq l$ and $ug^+ = g^+$ is finite.
	\end{lemm}

	\begin{proof}
		Without loss of generality we can assume that $\len g > L_S\delta$.
		We denote by $\gamma$ a $\delta$-nerve of $g$.
		Let $p$ be a projection of $x$ on $\gamma$.
		By definition of WPD property, there exists a positive integer $n$ such that the set $S$ of elements $u \in G$ satisfying $\dist{up}p \leq l + 34\delta$ and $\dist{ug^np}{g^np} \leq l + 34\delta$ is finite. 
		By projection on a quasi-convex (\autoref{res: proj quasi-convex}) we have $\gro x{g^+}p \leq 9\delta$.
		Since $\gamma$ is a $\delta$-nerve of $g$, $g^np$ lies on $\gamma$ between $p$ and $g^+$.
		It follows that $\gro x{g^+}{g^np}\leq 15\delta$.
		
		\paragraph{} Let $u$ be an element of $G$ such that $\dist{ux}x \leq l$ and $ug^+=g^+$.
		The estimates of the previous Gromov's products give
		\begin{equation*}
			\gro {ux}{g^+}{up} \leq 9\delta \text{ and } \gro{ux}{g^+}{ug^np} \leq 15\delta.
		\end{equation*}
		Applying \autoref{res: metric inequalities with boundary}~\ref{enu: metric inequalities with boundary - three points} we obtain
		\begin{displaymath}
			\dist {up}{p} \leq \dist {ux}x + 22\delta \leq l + 22\delta 
			\text{ and } 
			\dist {ug^np}{g^np} \leq \dist {ux}x + 34\delta \leq l + 34\delta.
		\end{displaymath}
		Consequently $u$ belongs to the finite set $S$.
	\end{proof}

	\begin{defi}
	\label{def: elementary subgroup}
		A subgroup $H$ of $G$ is called \emph{elementary} if $\partial H$ contains at most two points.
		Otherwise it is said \emph{non-elementary}.
	\end{defi}

	\rem Note that this notion implicitly depends on the action of $G$ on $X$.
	For instance a free group acting trivially on a hyperbolic space is not considered in this sense as a non-elementary groups.
	In the next lemmas we briefly recall how a free group quasi-isometrically embeds into a non-elementary subgroup of $G$.
			
	\begin{prop}
	\label{res: two hyp isom cannot have just one common end point}
		Let $g$ and $h$ be two loxodromic elements of $G$.
		Then $\{ g^-, g^+\}$ and $\{h^-, h^+\}$ are either disjoint or equal.
	\end{prop}
	
	\begin{proof}
		By replacing if necessary $g$ and $h$ by some powers we can assume that $\len g > L_S \delta$ and $\len h > L_S \delta$.
		We suppose that $\{ g^-, g^+\}$ and  $\{h^-, h^+\}$ have one common point that we denote $\xi$.
		Let $\gamma_g$ (\resp $\gamma_h$) be a $\delta$-nerve of $g$ (\resp $h$).
		We denote by $T$ the fundamental length of $\gamma_h$.
		
		\paragraph{} We fix a point $x$ of $\gamma_h$ and $y$ a projection of $x$ on $\gamma_g$.
		Since $\gamma_g$ is $9\delta$-quasi-convex we have $\gro \xi xy \leq 9\delta$.
		In particular there exists a point $z$ on $\gamma_h$ such that $\dist yz \leq 19\delta$.
		Up to reparametrize $\gamma_h$ we can assume that $z = \gamma_h(0)$.
			
		\paragraph{} Let $p \in \N$.
		By replacing if necessary $g$ by its inverse we can assume that $g^py$ is a point of $\gamma_g$ between $y$ and $\xi$.
		In particular $\gro \xi z{g^py} \leq \gro \xi y {g^py} + \dist yz \leq 25\delta$.
		The path $\gamma_h$ being $9\delta$-quasi-convex, there exists a point $s$ on $\gamma_h$ such that $\dist{g^py}s \leq 35\delta$.
		We can write $s = \gamma_h(r-qT)$ where $q \in \Z$ and $r \in [-T/2, T/2]$.
		It follows from the triangle inequality that 
		\begin{displaymath}
			\dist {h^qg^py}{y} 
			\leq \dist {g^py}s + \dist{\gamma_h(r)}{\gamma_h(0)} + \dist zy
			\leq T/2 + 54\delta.
		\end{displaymath}
		The isometries $g$ and $h$ also fix the point $\xi$.
		Using \autoref{res: WPD and fixed point in the boundary} we obtain the following.
		There exists a finite set $S$ such that for every $p \in \N$, there is $q \in \Z$ such that $h^qg^p$ belongs $S$.
		Consequently there exist $p,q \in \Z^*$ such that $g^p = h^q$.
		It implies that $\{ g^-, g^+\} = \{h^-, h^+\}$.
	\end{proof}
	
	\begin{lemm}{\rm \cite[Lemmes 1.1 and 1.2]{Del91} or \cite[Chapitre 5, Th\'eor\`eme 16]{GhyHar90}}\quad
	\label{res: quasi-isometric embeddings of free groups}
		Let $k>0$. Let $g_1,\dots, g_r$ be a collection of isometries of $X$.
		Let $x \in X$.
		We assume that for every $i, j \in \intvald 1r$, for every $\epsilon \in \{\pm1\}$, if $g_i^{-\epsilon} g_j$ is not trivial then
		\begin{displaymath}
			2  \gro {g_i^\epsilon x}{g_j x}x < \min \{\dist {g_ix}x , \dist {g_j}x\}  + \delta.
		\end{displaymath}
		Then $g_1, \dots, g_r$ generate a free group $\free r$ of rank $r$.
		Moreover the map $\free r \rightarrow X$ which send $g \in \free r$ to $gx$ is a quasi-isometric embedding.
	\end{lemm}
	
	\rem One consequence of this lemma is the following. 
	A subgroup $H$ of $G$ is non-elementary if and only if it contains a copy of $\free 2$ such that the map $\free 2 \rightarrow X$ that sends $g$ to $gx$ is a quasi-isometric embedding.
	Given two elements $u$ and $v$ of $G$ we now state a sufficient condition under which they generate a non-elementary subgroup.
	Note that the assumptions allow $u$ and $v$ to be elliptic.
%
%
%
%
	
	\begin{lemm}
	\label{res: non elementary subgroup sufficient condition}
		Let $A \geq 0$.
		Let $u,v \in G$ and $x \in X$.
		We assume that 
		\begin{enumerate}
			\item $2\gro {u^{\pm 1}x}{v^{\pm 1}x}x < \min \{\dist {ux}x, \dist {vx}x \} -A - 6\delta$,
			\item $2\gro{ux}{u^{-1}x}x < \dist {ux}x + A$,
			\item $2\gro{vx}{v^{-1}x}x < \dist {vx}x + A$.
		\end{enumerate}
		Then the subgroup of $G$ generated by $u$ and $v$ is non-elementary.
	\end{lemm}
	
	\begin{proof}
		Put $g_1 = uv$ and $g_2 = vu$.
		We are going to prove that $g_1$ and $g_2$ satisfy the assumptions of \autoref{res: quasi-isometric embeddings of free groups}.
		First note that $\dist {g_1x}x = \dist  {ux}x + \dist {vx}x - 2 \gro{u^{-1}x}{vx}x$. 
		In particular
		\begin{equation*}
			\dist {g_1x}x >  \max\left\{ \dist {ux}x, \dist{vx}x \right\} + A + 6\delta.
		\end{equation*}
		The same inequality holds for $g_2$.
		On the other hand, the hyperbolicity condition~(\ref{eqn: hyperbolicity condition 1}) gives
		\begin{equation*}
			\min\left\{ \gro{vx}{g_2x}x, \gro {g_2x}{g_1^{-1}x}x, \gro{g_1^{-1}x}{v^{-1}x}x \right\} \leq \gro {vx}{v^{-1}x}x + 2\delta,
		\end{equation*}
		which leads to 
		\begin{equation*}
		\label{eqn: non elementary subgroup sufficient condition}
			\min\left\{ \gro x{ux}{v^{-1}x},  \gro {g_2x}{g_1^{-1}x}x, \gro{u^{-1}x}x{vx} \right\} < \frac {\dist {vx}x}2 + \frac A2 + 2\delta.
		\end{equation*}
		Note that the minimum on the left hand side cannot be achieved by $\gro x{ux}{v^{-1}x}$. 
		If it was the case we would have indeed
		\begin{equation*}
			\frac {\dist {vx}x}2 + \frac A2 + 3\delta < \dist {vx}x - \gro{v^{-1}x}{ux}x =  \gro x{ux}{v^{-1}x} < \frac {\dist {vx}x}2 + \frac A2 + 2\delta.
		\end{equation*}
		Similarly it cannot be achieved by $\gro{u^{-1}x}x{vx}$.
		Thus we get 
		\begin{equation*}
			\gro {g_2x}{g_1^{-1}x}x < \frac {\dist {vx}x}2 + \frac A2 + 2\delta \leq \frac 12\min\left\{\dist{g_1x}x, \dist{g_2x}x \right\} - \delta
		\end{equation*}
		With similar arguments we obtain the upper bound for the other Gromov products which are required to apply \autoref{res: quasi-isometric embeddings of free groups}.
		Thus the subgroup of $\langle u ,v \rangle$ generated by $g_1$ and $g_2$ is a free group of rank $2$ which quasi-isometrically embeds into $X$.
		Therefore $\langle u ,v \rangle$ is not elementary.
	\end{proof}

%
%

\subsection{Elementary subgroups}

	\paragraph{}
	Following the classification of isometries, we sort the elementary subgroups of $G$ into three categories.
	A subgroup $H$ of $G$ is
	\begin{enumerate}
		\item \emph{elliptic} if its orbits are bounded
		\item \emph{parabolic} if $\partial H$ contains exactly one point
		\item \emph{loxodromic} if $\partial H$ contains exactly two points.
	\end{enumerate}
	In this section we give a brief exposition of the properties of these subgroups.	
	We still assume that the action of $G$ on $X$ is WPD.
	
	\begin{lemm}
	\label{res: isometry with finite index in a subgroup}
		Let $E$ be a subgroup of $G$ and $g$ an element of $E$.
		Assume that $\langle g \rangle$ is a finite index subgroup of $E$. 
		Then $E$ is elementary.
		Moreover $E$ is elliptic (\resp parabolic, loxodromic) if and only if $g$ is elliptic (\resp parabolic, loxodromic).
	\end{lemm}

	\begin{proof}
		Let $x$ be a point of $X$.
		Since $\langle g \rangle$ is a finite index subgroup of $E$, the Hausdorff distance between the orbits $\langle g \rangle \cdot x$ and $E\cdot x$ is finite.
		Therefore $\partial E = \partial \langle g \rangle$.
		The lemma follows from this equality.
	\end{proof}

	\subsubsection{Elliptic subgroups}
	
	Given an elliptic subgroup $H$ of $G$ we denote by $C_H$ the $H$-invariant subset of $X$ defined by 
	\begin{displaymath}
		C_H = \left\{ x \in X | \forall h \in H, \ \dist{hx}x \leq 11 \delta \right\}
	\end{displaymath}
	
	\begin{prop}{\rm \cite[Corollaries 2.32 and 2.33]{Coulon:2013tx}} \quad
	\label{res: characteristic subset summary}
		The subset $C_H$ is $9 \delta$-quasi-convex.
		Let $Y$ be a non-empty $H$-invariant $\alpha$-quasi-convex subset of $X$.
		For every $A > \alpha$, the $A$-neighborhood of $Y$ contains a point of $C_H$.
	\end{prop}

	\subsubsection{Loxodromic subgroups}
	
	\paragraph{} 
	Let $H$ be a loxodromic subgroup of $G$.
	According to \autoref{res: two boundary points give loxodromic isometry}, $H$ contains a loxodromic isometry $g$.
	In particular $g^-$ and $g^+$ are exactly the two points of $\partial H$.
	Moreover $H$ stabilizes $\partial H$.
	There exists a subgroup $H^+$ of $H$ of index at most 2 which fixes point wise $\partial H$.
	If $H^+ \neq H$ the subgroup $H$ is said to be of \emph{dihedral type}.

	\begin{lemm}
	\label{res: max loxodromic subgroup}
		Let $g$ be a loxodromic element of $G$.
		Let $E$ be the subgroup of $G$ stabilizing $\left\{ g^-,g^+\right\}$.
		Then $E$ is a loxodromic subgroup of $G$. 
		Moreover every elementary subgroup of $G$ containing $g$ lies in $E$.
	\end{lemm}

	\begin{proof}
		By definition $g$ belongs to $E$ therefore $\partial E$ contains $\left\{ g^-,g^+\right\}$.
		If $\partial E$ has an other point, then by \autoref{res: three boundary points gives two distinct loxodromic isometries} it contains an other loxodromic isometry $h$ such that $\left\{ h^-,h^+\right\} \neq \left\{ g^-,g^+\right\}$.
		As an element of $E$, $h^2$ fixes $g^-$ and $g^+$.
		On the other hand, since $h$ is loxodromic the only points of $\partial X$ fixed by $h^2$ are $h^-$ and $h^+$.
		Contradiction.
		Therefore $E$ is a loxodromic subgroup.
		
		\paragraph{}Let $H$ be an elementary subgroup of $G$ containing $g$.
		In particular $g^-$ and $g^+$ belong to $\partial H$.
		Since $H$ is elementary, there is no other point in $\partial H$.
		As we noticed $H$ stabilizes $\partial H$, thus $H$ is contained in $E$.
	\end{proof}
	
	\begin{prop}
	\label{res: normalizer of loxodromic virtually cyclic}
		Let $g \in G$ be a loxodromic isometry and $E$ the subgroup of $G$ which stabilizes $\{g^+,g^-\}$.
		Then $\langle g\rangle$ is a finite index subgroup of $E$.
	\end{prop}
	
	\begin{proof}
		Note that it is sufficient to prove that $\langle g \rangle$ has finite index in $E^+$ the subgroup of $E$ fixing pointwise $\{g^+,g^-\}$.
		The isometry $g$ is loxodromic. 
		Thus, by replacing if necessary $g$ by a power of $g$, we can assume that $\len g > L_S\delta$.
		Let $\gamma : \R \rightarrow X$ be a $\delta$-nerve of $g$ and $T$ its fundamental length.
		The point $x$ stands for $\gamma(0)$.
				
		\paragraph{}Let $u$ be an element of $E^+$.
		By definition of $E^+$, $u\gamma$ is a $T$-local $(1,\delta)$-quasi-geodesic joining $g^-$ to $g^+$.
		According to the stability of quasi-geodesics, there exists a point $p$ on $\gamma$ such that $\dist {ux}p \leq 7\delta$.
		We can write $p = \gamma(r-mT)$ where $m \in \Z$ and $r \in [-T/2, T/2]$.
		It follows from the triangle inequality that 
		\begin{displaymath}
			\dist {g^mux}{x} 
			\leq \dist {ux}p + \dist{\gamma(r)}{\gamma(0)} 
			\leq T/2 + 7\delta.
		\end{displaymath}
		The isometries $u$ and $g$ also fix the point $g^+$.
		Using \autoref{res: WPD and fixed point in the boundary} we obtain the following.
		There exists a finite subset $S$ of $G$ such that for every $u \in E^+$, there is $m \in \Z$ such that  $g^mu$ belongs to $S$.
		Thus $\langle g \rangle$ is a finite index subgroup of $E^+$.
	\end{proof}
	
	The next corollary is a well-known consequence of the previous proposition and a Schur Theorem \cite[Theorem 5.32]{Rotman:1995ud}.
	
	\begin{coro}
	\label{res: structure of loxodromic subgroups}
		Let $H$ be a loxodromic subgroup of $G$.
		The set $F$ of all elements of finite order of $H^+$ is a finite normal subgroup of $H$.
		Moreover there exists a loxodromic element $g \in H^+$ such that $H^+$ is isomorphic to $\sdp F{\Z}$ where $\Z$ is the subgroup generated by $g$ acting by conjugacy on $F$.
	\end{coro}
	
	\rem The subgroup $F$ is the unique maximal finite subgroup of $H^+$.
	In addition, if $H$ is of dihedral type then $H$ is isomorphic to $\sdp F {\dihedral}$ where $\dihedral$ stands for the infinite dihedral group $\dihedral = \Z/2\Z*\Z/2\Z$.
	In particular $F$ is the unique maximal normal finite subgroup of $H$.
	
	\begin{defi}
	\label{def: primitive element}
		 Let $g$ be a loxodromic element of $G$.
		Let $E$ be the subgroup of $G$ stabilizing $\{g^-, g^+\}$ and $F$ its maximal normal finite subgroup.
		We say that $g$ is \emph{primitive} if its image in $E^+/F \equiv \Z$ is $-1$ or $1$.
	\end{defi}

	\begin{coro}
	\label{res: two elliptic generates infinite dihedral}
		Let $A$ and $B$ be two elementary subgroups of $G$ which are not loxodromic.
		If $A$ and $B$ generate a loxodromic subgroup then it is necessarily of dihedral type.
	\end{coro}
	
	\begin{proof}
		Assume that the subgroup $H$ generated by $A$ and $B$ is not of dihedral type.
		It follows from our previous discussion that $H$ is isomorphic to the semi-direct product $\sdp F\Z$ where $F$ is a finite group and $\Z$ is generated by a loxodromic element $g$ acting by conjugacy on $F$.
		Every element $h$ of $H$ can be written $h = g^mu$ with $m \in \Z$ and $u \in F$.
		Moreover $h$ is loxodromic if and only if $m \neq 0$.
		Consequently every elliptic or parabolic element of $H$ belongs to $F$ (and thus has finite order).
		In particular $A$ and $B$ are both contained in $F$.
		Therefore they cannot generate a loxodromic subgroup.
		Contradiction.
	\end{proof}
	
	\begin{lemm}{\rm \cite[Lemma 2.34]{Coulon:2013tx}} \quad
	\label{res : cylinder fixed by normal finite subgroup}
		Let $g$ be a loxodromic element of $G$.
		Let $E$ be the subgroup of $G$ stabilizing $\{g^-, g^+\}$ and $F$ its maximal normal finite subgroup.
		Then $Y_g$ is contained in the $37\delta$-neighborhood of $C_F$.
	\end{lemm}
	
	\begin{proof}
		Since $F$ is a normal subgroup of $E$, $C_F$ is a $g$-invariant $9\delta$-quasi-convex subset of $X$.
		We apply \autoref{res: Yg in quasi-convex g-invariant}.
	\end{proof}

	\subsubsection{Parabolic subgroups}
	
	\begin{lemm}
	\label{res: parabolic subgroup}
		Let $H$ be a parabolic subgroup of $G$.
		Let $E$ be the subgroup of $G$ fixing $\partial H$.
		Then $\partial E = \partial H$. 
		In particular $E$ is parabolic.
	\end{lemm}

	\begin{proof}
		By construction $E$ contains $H$.
		Therefore $\partial H$ is a subset of $\partial E$.
		Assume now that $\partial E$ has at least two points.
		By \autoref{res: two boundary points give loxodromic isometry}, $E$ contains a loxodromic element $g$.
		This element fixes exactly two points of $\partial X$, $g^-$ and $g^+$, one of them being the unique point of $\partial H$.
		Without loss of generality we can assume that $\partial H = \{g^+\}$.
		Let $u$ be an element of $H$.
		The conjugate $ugu^{-1}$ is a loxodromic element of $E$ such that $(ugu^{-1})^+ = g^+$.
		According to \autoref{res: two hyp isom cannot have just one common end point}, $(ugu^{-1})^- = g^-$.
		Hence $u$ fixes pointwise  $\{g^-, g^+\}$.
		By \autoref{res: normalizer of loxodromic virtually cyclic} the stabilizer of $\{g^-,g^+\}$ is virtually $\Z$.
		Moreover it contains a finite subgroup $F$ such that every non-loxodromic element fixing pointwise $\{g^-,g^+\}$ belongs to $F$.
		In particular $H$ lies in $F$, which contradicts the fact that $H$ is parabolic.
	\end{proof}
	
	To every elliptic subgroup $F$ of $G$ we associated a characteristic subset $C_F$.
	We would like to have an analogue of such a set for a parabolic group $H$.
	By definition, there is no point $x \in X$ which is moved by a small distance by all the elements of $H$.
	However this fact remains true for any finite subset of $H$.
	This is the purpose of the next lemma.
	
	\begin{lemm}
	\label{res: partial characteristic subset for parabolic}
		Let $H$ be a parabolic subgroup of $G$ and $\xi$ the unique point of $\partial H$.
		Let $l \in \intval 0\delta$.
		Let $\gamma : \R_+ \rightarrow X$ be a $L_S\delta$-local $(1,l)$-quasi-geodesic such that $ \lim_{t \rightarrow + \infty} \gamma(t) = \xi$.
		Let $S$ be a finite subset of $\stab \xi$.
		There exists $t_0 \geq 0$ such that for every $t \geq t_0$, $\dist {g\gamma(t)}{\gamma(t)} \leq 166\delta$.
	\end{lemm}
	
	\begin{proof}
		Note that it is sufficient to prove the lemma for a set $S$ with a single element.
		Let us call it $g$.
		We denote by $x = \gamma(0)$ the origin of the path $\gamma$.
		By assumption $g\gamma$ is a $L_S\delta$-local $(1,l)$-quasi-geodesic joining $gx$ to $g\xi = \xi$.
		There exists $t_0 \in \R_+$ such that for every $t \geq t_0$, $\dist{\gamma(t)}x > \dist {gx}x + 8\delta$.
		For simplicity of notation we put $y = \gamma(t)$.
		We denote by $z$ a point of $\gamma$ between $y$ and $\xi$ such that $\dist yz > \dist{gx}x + 24\delta$.
		In particular $\gro x\xi z \leq 6 \delta$.
		According to \autoref{res: metric inequalities with boundary}~\ref{enu: metric inequalities with boundary - three points} we have $\dist {gz}z \leq \dist{gx}x+ 16\delta$.
		By hyperbolicity we get 
		\begin{equation*}
			\min\left\{\dist xy - \dist{gx}x, \gro xz{gy}, \dist zy - \dist{gz}z \right\} \leq \gro{gx}{gz}{gy} + 2 \delta \leq 8\delta.
		\end{equation*}
		By choice of $t_0$ and $z$ the minimum is necessarily achieved by $\gro xz{gy}$, hence $\gro xz{gy} \leq 8 \delta$.
		As a quasi-geodesic, $\gamma$ is $9\delta$-quasi-convex.
		Therefore the projection $p = \gamma(s)$ of $gy$ on $\gamma$ is $17 \delta$-close to $gy$.
		We assume that $s \geq t$. 
		A similar argument works in the other case.
		In particular,
		\begin{equation*}
			\gro y\xi{gy} \leq \gro y\xi p +\dist {gy}p \leq 23\delta.
		\end{equation*}
		Let $q$ be an $\delta$-projection of $y$ on $A_g$.
		According to \autoref{res: axes is quasi-convex}, 
		\begin{equation*}
			\dist y{gy} \geq \dist yq + \dist{gq}q + \dist{gq}{gy} - 8\delta.
		\end{equation*}
		In particular $\gro y{gy}q \leq 4\delta$ and $\gro y{gy}{gq} \leq 4\delta$.
		It follows from \autoref{res: metric inequalities with boundary}~\ref{enu: metric inequalities  with boundary - thin triangle} and (\ref{eqn: estimate gromov product boundary}) that 
		\begin{equation*}
			\gro y\xi q 
			\leq \max \left\{ \dist yq - \gro \xi {gy}y, \gro y{gy}q \right\} +\delta 
			\leq \max\left\{ \gro \xi y{gy} + 2 \gro y{gy}q + 2\delta, \gro y{gy}q \right \} + \delta.
		\end{equation*}
		Consequently we get $\gro y\xi q \leq  34 \delta$. 
		On the other hand, the triangle inequality leads to
		\begin{equation*}
			\gro q\xi{y} = \gro{gq}\xi{gy} \leq \gro y\xi{gy} + \gro y{gy}{gq} \leq 27\delta.
		\end{equation*}
		Thus $\dist yq \leq \gro y\xi q + \gro q\xi y +2\delta \leq 63\delta$.
		By \autoref{res: parabolic subgroup}, $g$ is not a loxodromic isometry, thus $\len g \leq 32\delta$.
		Using the triangle inequality we get
		\begin{equation*}
			\dist{g\gamma(t)}{\gamma(t)} = \dist{gy}y \leq 2 \dist yq + \dist{gq}q \leq 2 \dist yq + \len g + 8\delta \leq 166\delta. \qedhere
		\end{equation*}
	\end{proof}

%
%
	
\subsection{Group invariants}
\label{sec:group invariants}
	
	We now introduce several invariants associated to the action of $G$ on $X$.
	During the final induction, they will be useful to ensure that the set of relations we are looking at satisfies a small cancellation assumption.
	In all this section we assume that the action of $G$ on $X$ is WPD.
	
	\begin{defi}
	\label{def: injectivity radius}
		The \emph{injectivity radius} of $G$ on $X$, denoted by $\rinj GX$ is 
		\begin{displaymath}
			\rinj GX = \inf \set{\len[stable] g}{g \in G, \ g \text{ loxodromic}}
		\end{displaymath}
	\end{defi}

	Let $F$ be a finite group.
	Its \emph{holomorph}, denoted by $\hol F$, is the semi-direct product $\sdp F {\aut F}$, where $\aut F$ stands for the automorphism group of $F$.
	The \emph{exponent} of $\hol F$ is the smallest integer $n$ such that for every $g \in \hol F$, $g^n = 1$.
	
	\begin{defi}
	\label{res: invariant e}
		The integer $e(G, X)$ is the least common multiple of the exponents of $\hol F$, where $F$ describes the maximal finite normal subgroups of all maximal loxodromic subgroups of $G$.
	\end{defi}

	\rem If the loxodromic subgroups of $G$ are all cyclic (for instance if $G$ is torsion free) then $e(G,X) = 1$.
	
	\begin{lemm}{\rm Compare \cite[Lemma 19]{IvaOlc96}}\quad
	\label{res: central type element in loxodromic subgroup}
		Let $n$ be an integer, multiple of $e(G,X)$.
		Let $E$ be a loxodromic subgroup of $G$ and $F$ its maximal finite normal subgroup.
		For every loxodromic element $g \in E$, for every $u \in F$ we have the following identities
		\begin{displaymath}
			(ug)^n  = g^n \quad \text{and}  \quad ug^nu^{-1} = g^n.
		\end{displaymath}
	\end{lemm}
	
	\begin{proof}
	Without loss of generality we can assume that $E$ is a maximal loxodromic subgroup of $G$.
		Let $g$ be a loxodromic element of $E$ and $u$ an element of $F$.
		Recall that $g$ acts by conjugacy on $F$.
		We denote by $\psi$ the corresponding automorphism of $F$.
		The first identity is a consequence of the following observations.
		\begin{equation*}
			(ug)^n 
			= u\left(gug^{-1}\right)\left(g^2ug^{-2}\right) \dots \left(g^{n-1}ug^{-(n-1)}\right)g^n
			= u\psi(u)\psi^2(u)\dots \psi^{n-1}(u)g^n 
		\end{equation*}
		However in $\hol F$ we have
		\begin{equation*}
			\left(u\psi(u)\psi^2(u)\dots \psi^{n-1}(u), 1\right) = (u,\psi)^n (1,\psi)^{-n}  =1.
		\end{equation*}
		Thus $(ug)^n =g^n$.
		Since $F$ is a normal subgroup of $F$, $gu^{-1}g^{-1}$ also belongs to $F$.
		The previous identity yields
		\begin{equation*}	
			ug^nu^{-1} = \left(ugu^{-1}\right)^n = \left[(ugu^{-1}g^{-1})g\right]^n = g^n. \qedhere
		\end{equation*}
		
	\end{proof}

	\begin{prop}
	\label{res: non elementary subgroup generated by two elements}
		Let $n$ be an integer, multiple of $e(G,X)$.
		Let $g$ and $h$ be two loxodromic elements of $G$ which are primitive.
		Either $g$ and $h$ generate a non-elementary subgroup or $\langle g^n\rangle = \langle h^n\rangle$.
	\end{prop}
	
	\begin{proof}
		Let $E$ be the subgroup of $G$ stabilizing $\{g^-,g^+\}$.
		We write $F$ for its maximal finite normal subgroup.
		Since $g$ is primitive (see \autoref{def: primitive element}), $E^+$ is isomorphic to the semi-direct product $\sdp F\Z$ where $\Z$ is the subgroup generated by $g$ acting by conjugacy on $F$.
		Assume that $g$ and $h$ generate an elementary subgroup.
		In particular $h$ belongs to $E$ and $\{h^-,h^+\} = \{g^-,g^+\}$.
		However being loxodromic, $h$ fixes pointwise $ \{g^-,g^+\}$ thus $h$ belongs to $E^+$.
		The element $h$ is also primitive, thus there exists $u \in F$ such that $g = uh^{\pm1}$.
		It follows from \autoref{res: central type element in loxodromic subgroup} that $g^n = h^{\pm n}$, hence $\langle g^n\rangle = \langle h^n\rangle$.
	\end{proof}

	\begin{defi}
	\label{def: invariant nu}
		The invariant $\nu (G,X)$ (or simply $\nu$) is the smallest positive integer $m$ satisfying the following property. 
		Let $g$ and $h$ be two isometries of $G$ with $h$ loxodromic.
		If $g$, $h^{-1}gh$,..., $h^{-m}gh^m$ generate an elementary subgroup which is not loxodromic then $g$ and $h$ generate an elementary subgroup of $G$.
	\end{defi}
	
	\ex If $G$ is acting properly co-compactly on a hyperbolic space $X$, then $\nu(G,X)$ is finite.
	Moreover if every loxodromic subgroup of $G$ is cyclic then $\nu(G,X) = 1$.
	Other examples are given in \autoref{sec: examples}.
	
	\begin{prop}
	\label{res: invariant nu'}
		Let $g$ and $h$ be two elements of $G$ with $h$ loxodromic and $m$ an integer such that $g$, $h^{-1}gh$,..., $h^{-m}gh^m$ generate an elementary (possibly loxodromic) subgroup of $G$.
		We assume that $m \geq \nu(G,X)$ and $G$ has no involution.
		Then $g$ and $h$ generate an elementary subgroup of $G$.
	\end{prop}
	
	\begin{proof}
		We write $H$ for the subgroup of $G$ generated by $g$, $h^{-1}gh$,..., $h^{-m}gh^m$.
		We assume first that $g$ is not loxodromic.
		We denote by $p$ the largest integer such that $g$, $h^{-1}gh$,..., $h^{-p} gh^p$ generate an elementary subgroup which is not loxodromic, that we denote $E$.
		If $p \geq \nu(G,X)$, then by definition $g$ and $h$ generate an elementary subgroup.
		Therefore we can assume that $p \leq \nu(G,X) - 1 \leq m-1$.
		Since $p$ is maximal $E$ and $hEh^{-1}$ generate a loxodromic subgroup of $H$.
		According to \autoref{res: two elliptic generates infinite dihedral}, this loxodromic subgroup is of dihedral type.
		This is not possible since $G$ has no involution.
		Consequently, we can assume that $g$ is loxodromic.
		In particular $\partial H$ contains exactly two points $g^-$ and $g^+$ which are also the accumulation points of $h^{-1}gh$.
		It follows that $h$ stabilizes $\{g^-,g^+\}$.
		Consequently $g$ and $h$ are contained in the elementary subgroup of $G$ which stabilizes $\{g^-, g^+\}$ (see \autoref{res: max loxodromic subgroup}).
	\end{proof}
	
	\nota If $g_1$, \dots, $g_m$ are $m$ elements of $G$ we denote by $A(g_1,\dots, g_m)$ the quantity
	\begin{displaymath}
		A(g_1,\dots, g_m) = \diam \left(A_{g_1}^{+13\delta} \cap \dotsc \cap A_{g_m}^{+13\delta}\right).
	\end{displaymath}
	
	\begin{defi}
	\label{def: invariant A}
		Assume that $\nu = \nu(G,X)$ is finite.
		We denote by $\mathcal A$ the set of $(\nu+1)$-uples $(g_0,\dots,g_\nu)$ such that $g_0,\dots,g_\nu$ generate a non-elementary subgroup of $G$ and for all $j \in \intvald 0\nu$, $\len{g_j} \leq L_S\delta$. 
		The parameter $A(G,X)$ is given by
		\begin{displaymath}
			A(G,X) = \sup_{(g_0,\dots,g_\nu) \in \mathcal A} A\left(g_0, \dots,g_\nu\right)
		\end{displaymath}
	\end{defi}

	\begin{prop}
	\label{res: overlap two axes}
		Let $g$ and $h$ be two elements of $G$ which generate a non-elementary subgroup.
		\begin{enumerate}
			\item \label{enu: overlap axes short}
			If $\len g \leq L_S \delta$, then $A(g,h) \leq \nu \len h + A(G,X) + 156\delta$.
			\item \label{enu: overlap axes general}
			Without assumption on $g$ we have, 
			\begin{displaymath}
				A(g,h) \leq \len g + \len h + \nu\max\{\len g, \len h\} + A(G,X) + 684\delta.
			\end{displaymath}
		\end{enumerate}
	\end{prop}
	
	\rem If $\len g \leq L_S \delta$ and loxodromic, the same proof shows that 
	\begin{equation*}
		A(g,h) \leq \len h + A(G,X) + 156\delta.
	\end{equation*}

	\begin{proof}
		 We prove Point~\ref{enu: overlap axes short} by contradiction. 
		 Assume that
		 \begin{displaymath}
		 	A(g,h) > \nu \len h + A(G,X)  + 156\delta.
		 \end{displaymath}
		 Let $\eta \in (0, \delta)$ such that
		 \begin{displaymath}
		 	A(g,h) > \nu(\len h +\eta) + A(G,X) + 2\eta + 156\delta.
		 \end{displaymath}
		By definition of $A(G,X)$ we have $\len h > L_S\delta$, otherwise $g$ and $h$ would generate an elementary subgroup.
		We denote by $\gamma : \R \rightarrow X$ an $\eta$-nerve of $h$ and by $T$ its fundamental length.
		In particular $T \leq \len h + \eta$.
		By \autoref{res: Ag in neighborhood of a nerve}, its $10\delta$-neighborhood contains $A_g$, therefore applying \autoref{res: intersection of thickened quasi-convex}, we get
		\begin{displaymath}
			\diam\left(A_g^{+13\delta}\cap \gamma^{+12\delta}\right) >\nu (\len h+\eta) + A(G,X) + 2\eta + 106\delta.
		\end{displaymath}
		In particular there exist  $x=\gamma(s)$ and $x'=\gamma(s')$ two points of $\gamma$ which also belong to the $25\delta$-neighborhood of $A_g$ and such that 
		\begin{equation}
		\label{eqn: overlap axes short}
			\dist x{x'} > \nu (\len h + \eta) + A(G,X) + 2\eta + 82\delta \geq \nu T + A(G,X) + 2\eta +82\delta.
		\end{equation}
		By replacing if necessary $h$ by $h^{-1}$ we can assume that $s \leq s'$.
		By stability of quasi-geodesics, for all $t \in \intval s{s'}$, $\gro x{x'}{\gamma(t)} \leq \eta/2 +5\delta$.
		Since the $25\delta$-neighborhood of $A_g$ is $2\delta$-quasi-convex (see \autoref{res: neighborhood quasi-convex}), it follows that $\gamma(t)$ lies in the $(\eta/2 + 32\delta)$-neighborhood of $A_g$.
		Thus $\dist{g\gamma(t)}{\gamma(t)} \leq \len g + \eta + 72\delta$.
		
		\paragraph{}
		According to (\ref{eqn: overlap axes short}) there exists $t \in \intval s{s'}$ such that $\dist x{\gamma(t)} = A(G,X) + 2\eta +82\delta$.
		We put $y= \gamma(t)$.
		Note that
		\begin{displaymath}
			\dist {s'}t \geq \dist {x'}y \geq \dist x{x'} - \dist xy  \geq \nu T.
		\end{displaymath}
		Let $m \in \intvald 0\nu$.
		By construction $h^mx = \gamma(s+ mT)$ and $h^my = \gamma(t + mT)$.
		Using our remark $s + mT$ and $t + mT$ belong to $\intval s{s'}$.
		Hence 
		\begin{displaymath}
			\max\left\{\dist{gh^mx}{h^mx}, \dist{gh^my}{h^my} \right\} \leq \len {h^mgh^{-m}} + \eta + 72\delta.
		\end{displaymath}
		It follows from \autoref{res: axes is quasi-convex}, that $x$ and $y$ belong to the $(\eta/2 + 39\delta)$-neighborhood of $h^mA_g$.
		This holds for every $m \in \intvald 0 \nu$.
		Consequently $x$ and $y$ are two points of 
		\begin{equation*}
			A_g^{+ \eta/2 + 39\delta} \cap \dotsc \cap h^\nu A_g^{+ \eta/2 +39\delta}.
		\end{equation*}
		Applying \autoref{res: intersection of thickened quasi-convex}, we obtain
		\begin{displaymath}
			A\left(g,hgh^{-1},\dots, h^\nu gh^{-\nu}\right) \geq \dist xy  -\eta - 82 \delta> A(G,X).
		\end{displaymath}
		Moreover, for every $m \in \intvald 0\nu$, we have $\len {h^mgh^{-m}} \leq L_S\delta$.
		By definition of $A(G,X)$ the isometries $g$, $hgh^{-1}$, \dots, $h^\nu gh^{-\nu}$ generate an elementary group.
		It follows from \autoref{res: invariant nu'} that $g$ and $h$ also generate an elementary group.
		Contradiction.
		
		\paragraph{} 
		We now prove Point~\ref{enu: overlap axes general}.
		According to the previous point we can assume that $\len g > L_S\delta$ and $\len h > L_S \delta$.
		Without loss of generality we can suppose $\len h \geq \len g$.
		Assume that contrary to our claim
		\begin{displaymath}
			A(g,h) > \len g + (\nu +1)\len h + A(G,X) + 684\delta.
		\end{displaymath}
		Let $\eta \in (0,\delta)$ such that 
		\begin{displaymath}
			A(g,h) > \len g + (\nu +1)\len h + A(G,X) + 13\eta + 684\delta.
		\end{displaymath}
		We denote by $\gamma$ an $\eta$-nerve of $h$ and by $T$ its fundamental lengths.
		Its $10\delta$-neighborhood contains $A_h$ thus
		\begin{displaymath}
			\diam \left(\gamma^{+ 12\delta} \cap A_g^{+ 13\delta}\right) > \len g + (\nu +1)\len h + A(G,X) +13\eta + 634\delta.
		\end{displaymath}
		In particular there exits $x = \gamma(s)$, $x'= \gamma(s')$ lying in the $25\delta$-neighborhood of $A_g$ such that
		\begin{displaymath}
			\dist x{x'}> \len g + (\nu +1)\len h + A(G,X) + 13\eta + 610\delta.
		\end{displaymath}
		Without loss of generality we can assume that $s \leq s'$.
		As previously, the restriction of $\gamma$ to $\intval s{s'}$ is contained in the $(\eta/2+32\delta)$-neighborhood of $A_g$.
		We apply \autoref{res : quasi-geodesic behaving like a nerve}.
		By replacing if necessary $g$ by $g^{-1}$, for every $t \in \intval s{s'}$ if $t \leq s'-\len g$ then 
		\begin{displaymath}
			\dist {g\gamma(t)}{\gamma(t+ \len g)} \leq 6 \eta + 222\delta.
		\end{displaymath}
		Consequently, for every $t \in \intval s{s'}$ such that $t \leq s' - \len g - T$ we have
		\begin{displaymath}
			\dist{gh\gamma(t)}{hg\gamma(t)}
			\leq \dist{g\gamma(t + T)}{h\gamma(t+\len g)} + 6 \eta + 222\delta
			\leq 12\eta + 444\delta.
		\end{displaymath}
		It follows that the translation length of the isometry $u = h^{-1}g^{-1}hg$ is at most $L_S\delta$ and for all $t \in \intval s{s'}$ if  $t \leq s' - \len g - T$ then $\gamma(t)$ is in the $(6\eta +225\delta)$-neighborhood of $A_u$.
		Let $y = \gamma(t)$ be a point such that $t \in \intval s{s'}$ and $\dist {x'}y = \len g + T$.
		In particular, 
		\begin{equation*}
			\dist xy \geq \dist x{x'} - \dist {x'}y > \nu\len h + A(G,X) + 12\eta + 594\delta.
		\end{equation*}
		Moreover $x$ and $y$ belong to the $(6\eta +225\delta)$-neighborhood of $A_u$ and $A_h$.
		Therefore
		\begin{displaymath}
			A(g,u) \geq \dist xy - 12\eta - 454 \delta > \nu\len h + A(G,X) + 156\delta.
		\end{displaymath}
		It follows from the previous point that $h$ and $u$ generate an elementary group.
		Hence so do $h$ and $g^{-1}hg$.
		However $h$ is a loxodromic isometry. 
		Consequently $g$ and $h$ generate an elementary group. 
		Contradiction.
	\end{proof}
	
	\begin{coro}
	\label{res: overlap multiples axes}
		Let $m$ be an integer such that $m \leq \nu (G,X)$.
		Let $g_1, \dots, g_m$ be $m$ elements of $G$.
		If they do not generate an elementary subgroup, then  
		\begin{displaymath}
			A\left(g_1, \dots,g_m\right) \leq (\nu + 2)\sup_{1 \leq i \leq m} \len {g_i} + A(G,X) + 684\delta.
		\end{displaymath}
	\end{coro}
	
	\begin{proof}
		We distinguish two cases.
		If for every $i \in \intvald 1m$ we have $\len {g_i} \leq L_S\delta$, then it follows from the definition of $A(G,X)$ that $A\left(g_1, \dots,g_m\right) \leq A(G,X)$.
		Assume now that there exists $i \in \intvald 1m$ such that $\len {g_i} > L_S\delta$.
		In particular $g_i$ is loxodromic.
		Suppose that the corollary is false. 
		Then by \autoref{res: overlap two axes}, for every $j \in \intvald 1m$ the elements $g_i$ and $g_j$ generate an elementary subgroup.
		Therefore $g_j$ belongs to the maximal elementary subgroup containing $g_i$.
		Consequently $g_1$, \dots, $g_m$ cannot generate a non-elementary subgroup.
		Contradiction.
	\end{proof}

%% file: 2_cone_off.tex

%
%

\section{Cone-off over a metric space}
\label{sec: cone-off over metric space}

\paragraph{}
In this section we recall the so called \emph{cone-off} construction.
The goal is to build a metric space $\dot X$ obtained by attaching a family of cones on a base space $X$.
In particular we need to understand its curvature.
Most of the result of this section follows from the general exposition given by the author in \cite{Coulon:2013tx}.

\paragraph{}
Let $\rho$ be a positive number.
Its value will be made precise later.
It should be thought as a very large parameter.

%
%

\subsection{Cone over a metric space}
\label{sec: cone}

	\begin{defi}
		Let $Y$ be a metric space. 
		The \emph{cone over $Y$ of radius $\rho$}, denoted by $Z_\rho(Y)$ or simply $Z(Y)$, is the topological quotient of $Y\times \left[0,\rho\right]$ by the equivalence relation that identifies all the points of the form $(y,0)$.
	\end{defi}
	
	The equivalence class of $(y,0)$, denoted by $v$, is called the \emph{apex} of the cone. 
	By abuse of notation, we still write $(y,r)$ for the equivalence class of $(y,r)$.
	The cone over $Y$ is endowed with a metric characterized as follows \cite[Chapter I.5, Proposition 5.9]{BriHae99}.
	Let $x=(y,r)$ and $x'=(y',r')$ be two points of $Z(Y)$ then
	\begin{equation*}
		\cosh \dist x{x'} = \cosh r \cosh r' - \sinh r\sinh r' \cos \angle y{y'},
	\end{equation*}
	where $\angle y{y'}$ is the \emph{angle at the apex} defined by $\angle y{y'} = \min \left\{ \pi ,  \dist y{y'}/\sinh \rho\right\}$.
	If $Y$ is a length space, then so is $Z(Y)$.
	This metric is modeled on the one of the hyperbolic place $\H$ (see \cite{Coulon:2013tx} for the geometric interpretation).
	In particular the cone $Z(Y)$ is $2\boldsymbol \delta$-hyperbolic, where $\boldsymbol \delta$ is the hyperbolicity constant of $\H$ \cite[Proposition 4.6]{Coulon:2013tx}.
	
	\paragraph{} In order to compare the cone $Z(Y)$ and its base $Y$ we introduce a map $\iota : Y \rightarrow Z(Y)$ which sends $y$ to $(y,\rho)$.
	It follows from the definition of the metric on $Z(Y)$ that for all $y,y' \in Y$, 
	\begin{equation*}
		\dist[Z(Y)]{\iota(y)}{\iota(y')} = \mu\left( \dist[Y] y{y'}\right),
	\end{equation*}
	where $\mu  \colon \R_+ \rightarrow \R_+$ is a map characterized as follows.
	For every $t \geq 0$,
	\begin{equation*}
		\cosh \mu(t) = \cosh^2 \rho - \sinh^2 \rho \cos \left(\min \left\{ \pi, \frac t{\sinh \rho}\right\}\right).
	\end{equation*}
	In addition, the map $\mu$ satisfies the following proposition whose proof is Calculus exercise.

	\begin{prop}
	\label{res: map mu}
		The map $\mu$ is continuous, concave, non-decreasing.
		Moreover, we have the followings.
		\begin{enumerate}
			\item \label{enu: mu - comparison map}
			for all $t \geq 0$,
			\begin{math}
			\displaystyle
				t - \frac 1{24}\left(1+\frac 1{\sinh^2\rho}\right) t^3 \leq \mu(t) \leq t.
			\end{math}
			\item \label{emu: mu - lower bound}
			for all $t \in \intval 0{\pi \sinh \rho}$, $t \leq \pi \sinh (\mu(t)/2)$.
		\end{enumerate}
	\end{prop}

	\begin{lemm}
	\label{res: cone - small curves close to the apex}
		Let $r \in \intval 0\rho$.
		The map from $Y$ to $Z(Y)$ which sends $y$ to $(y,r)$ is $\kappa$-Lipschtiz, with $\kappa = \sinh r / \sinh \rho$.
		In particular if $\gamma : I \rightarrow Y$ is a rectifiable path then the path $\tilde \gamma : I \rightarrow Z(Y)$ defined by $\tilde \gamma(t) = (\gamma(t),r)$ is rectifiable and $L(\tilde \gamma) \leq \kappa L(\gamma)$.
	\end{lemm}

	\begin{proof}
		Let $y$ and $y'$ be two points of $Y$ such that $\dist y{y'} \leq \pi \sinh \rho$.
		We put $x = (y,r)$ and $x' = (y',r)$.
		By definition of the metric on $Z(Y)$ we have
		\begin{equation*}
			\cosh\left(\dist x{x'}\right) 
			= 1 + \sinh^2r \left[1-\cos\left(\frac {\dist y{y'}}{\sinh \rho}\right) \right]
			\leq 1 + \frac 12 \cdot \frac {\sinh^2r}{\sinh^2 \rho} \dist y{y'}^2
		\end{equation*}
		It follows that $\dist x{x'} \leq\kappa\dist y{y'}$, where $\kappa = \sinh r / \sinh \rho$.
		The same inequality holds if $\dist y{y'} > \pi \sinh \rho$.
		Thus the map $Y \rightarrow Z(Y)$ which sends $y$ to $(y,r)$ is $\kappa$-Lipschitz.
		The property about the path $\tilde \gamma$ follows from this fact.
	\end{proof}

	
	\paragraph{Group action on a cone}
	Let $Y$ be a metric space endowed with an action by isometries of a group $H$.
	This action naturally extends to an action by isometries on $Z(Y)$ in the following way.
	For every point $x=(y,r)$ of $Z(Y)$, for every $h\in H$, we let $h\cdot x=(hy,r)$.
		
	\begin{lemm}{\rm \cite[Lemma 4.7]{Coulon:2013tx}} \quad
	\label{res: group acting as a rotation on cone}
		Let $Y$ be a metric space and $H$ a group acting by isometries on $Y$.
		Assume that for every $h \in H$, $\len h \geq \pi \sinh \rho$.
		Then for every point $x \in Z(Y)$, for every $h \in H\setminus\{1\}$, $\dist{hx}x = 2\dist xv$.
	\end{lemm}
	
	\paragraph{}
	Note that $H$ fixes the apex $v$ of the cone.
	Therefore this action is not necessarily proper (even if the one of $H$ on $Y$ is).
	One should think as $H$ as a rotation group with center $v$.
	Nevertheless if $H$ acts properly on $Y$, then the metric on $Z(Y)$ induces a distance on $Z(Y)/H$.
	Moreover the spaces $Z(Y)/H$ and $Z(Y/H)$ are isometric.
	For every point $x$ in $Z(Y)$, we denote by $\bar x$ its image in $Z(Y)/H$.
	
	\begin{lemm}{\rm \cite[Lemma 4.8]{Coulon:2013tx}} \quad
	\label{res: metric quotient cone}
		Let $l \geq 2\pi \sinh \rho$.
		We assume that for every $h \in H\setminus \{1\}$, $\len h \geq l$.
		Let $x=(y,r)$ and $x'=(y',r')$ be two points of $Z(Y)$.
		If $\dist[Y]y{y'} \leq l - \pi \sinh \rho$ then $\dist{\bar x}{\bar x'} = \dist x{x'}$.
	\end{lemm}

%
%

\subsection{The cone-off construction. Definition and curvature}
\label{sec: cone-off}

\paragraph{}
We now explain how the cones introduced in the previous section can be attached on a metric space.
Let $X$ be a $\delta$-hyperbolic length space.
We consider a collection $\mathcal Y$ of strongly quasi-convex subsets of $X$.
Let  $Y \in \mathcal Y$.
We denote by $\distV[Y]$ the length metric on $Y$ induced by the restriction of $\distV[X]$ to $Y$.
We write $Z(Y)$ for the cone of radius $\rho$ over $(Y,\distV[Y])$.
Its comes with a natural map $\iota : Y \hookrightarrow Z(Y)$ as defined in \autoref{sec: cone}.

\begin{defi}
\label{def:pre-cone off}
	The \textit{cone-off of radius $\rho$ over $X$ relative to $\mathcal Y$} denoted by $\dot X_\rho(\mathcal Y)$ (or simply $\dot X$) is obtained by attaching for every $Y \in \mathcal Y$, the cone $Z(Y)$ on $X$ along $Y$ according to $\iota$.
\end{defi}

In other words the space $\dot X$ is the quotient of the disjoint union of $X$ and all the $Z(Y)$ by the equivalence relation which identifies every point $y\in Y$ with its image $\iota(y) \in Z(Y)$.
By abuse of notation, we use the same letter to designate a point of this disjoint union and its image in $\dot X$.


\paragraph{Metric on the cone-off.}
For the moment $\dot X$ is just a set of points.
We now define a metric on $\dot X$ and recall its main properties.
Note that we did not require the attachment maps $\iota$ to be isometries.
We endow the disjoint union of $X$ and all the $Z(Y)$ (where $Y \in \mathcal Y$) with the distance induced by $\distV[X]$ and $\distV[Z(Y)]$.
This metric is not necessary finite: the distance between two points in distinct components is infinite.
Let $x$ and $x'$ be two points of $\dot X$.
We define $\dist[SC] x{x'}$ to be the infimum of the distances between two points in the previous disjoint union whose images in $\dot X$ are respectively $x$ and $x'$.
\begin{enumerate}
	\item Let $Y \in \mathcal Y$.
	If $x \in Z(Y)\setminus Y$ and $x' \notin Z(Y)$, then $\dist[SC] x{x'} = + \infty$.
	In particular $\distV[SC]$ is not a distance on $\dot X$ (it does not satisfy the triangle inequality).
	\item Let $x$ and $x'$ be two points of $X$.
	Using the properties of $\mu$, 
	\begin{equation*}
		\mu\left( \dist[X]x{x'}\right) \leq \dist[SC] x{x'} \leq \dist[X]x{x'}.
	\end{equation*}
\end{enumerate}
Let $x$ and $x'$ be two points of $\dot X$.
A \emph{chain} between $x$ and $x'$ is a finite sequence $C=\left(z_1,\dots,z_m\right)$ such that $z_1=x$ and $z_m=x'$.
Its length, denoted by $l(C)$, is
\begin{equation*}
	l(C) = \sum_{j=1}^{m-1} \dist[SC]{z_{j+1}}{z_j}.
\end{equation*}
The following map endows $\dot X$ with a length metric  \cite[Proposition 5.10]{Coulon:2013tx}.
\begin{equation*}
	\begin{array}{ccccl}
		\dot X \times \dot X & \rightarrow & \R_+ && \\
		(x,x') & \rightarrow & \dist[\dot X] x{x'} & = & \inf \left\{ l(C) | C \text{ chain between $x$ and $x'$}\right\}.
	\end{array}
\end{equation*}
For every $Y \in \mathcal Y$ the natural map $Z(Y) \rightarrow \dot X$ is a $1$-Lipschitz embedding.
The same holds for the map $X \rightarrow \dot X$.
The next lemmas detail the relationship between the metric of these spaces.

\begin{lemm}{\rm \cite[Lemma 5.8]{Coulon:2013tx}} \quad
\label{res : comparison metric X and dot X}
	For every $x, x' \in X$, $\mu\left(\dist[X] x{x'}\right) \leq \dist[\dot X] x{x'} \leq \dist[X] x{x'}$.
\end{lemm}

\begin{lemm}{\rm \cite[Lemma 5.7]{Coulon:2013tx}} \quad
\label{res: metric on  dot X and Zi coincide}
	Let $Y\in \mathcal Y$.
	Let $x \in Z(Y)\setminus Y$.
	Let $d(x,Y)$ be the distance between $x$ and $\iota(Y)$ computed with $\distV[Z(Y)]$.
	For all $x' \in \dot X$, if $\dist[\dot X] x{x'} < d(x,Y)$ then $x'$ belongs to $Z(Y)$.
	Moreover $\dist[\dot X] x{x'} = \dist[Z(Y)] x{x'}$.
\end{lemm}

\rem If $v$ stands for the apex of the cone $Z(Y)$, then the previous lemma implies that $Z(Y) \setminus Y$ is exactly the ball of $\dot X$ of center $v$ and radius $\rho$.


\paragraph{Large scale geometry of the cone-off.}	
In \cite{DruSap05} C~Drutu and M~Sapir introduced the notion of tree-graded spaces.
If $X$ is tree-graded with respect to $\mathcal Y$, then $\dot X$ has a very precise geometry.
For instance, it is tree-graded with respect to $\set{Z(Y)}{Y \in \mathcal Y}$ and $2 \boldsymbol \delta$-hyperbolic.
From a qualitative point of view some of the metric features of $\dot X(\mathcal Y)$ still hold after a small ``perturbation'' of the geometry of $X$.
To make this statement precise we need to introduce a parameter that control the overlap between two elements of $\mathcal Y$.
We put
\begin{equation*}
	\Delta(\mathcal Y) = \sup_{Y_1\neq Y_2 \in \mathcal Y} \diam\left(Y_1^{+ 5\delta}\cap Y_2^{+5\delta}\right)
\end{equation*}

\begin{theo}{\rm \cite[Proposition 6.4]{Coulon:2013tx}} \quad
\label{res: curvature of dot X - general case}
	There exist positive numbers $\delta_0$, $\Delta_0$ and $\rho_0$ with the following property.
	Let $X$ be a $\delta$-hyperbolic length space with $\delta \leq \delta_0$.
	Let $\mathcal Y$ be a family of strongly quasi-convex subsets of $X$ with $\Delta(\mathcal Y) \leq \Delta_0$.
	Let $\rho \geq \rho_0$.
	Then the cone-off $\dot X_\rho(\mathcal Y)$ of radius $\rho$ over $X$ relative to $\mathcal Y$ is $\dot \delta$-hyperbolic with $\dot \delta = 900 \boldsymbol \delta$.	
\end{theo}

\rem It is important to note that in this statement the constants $\delta_0$, $\Delta_0$ and $\rho_0$ do not depend on $X$ or $\mathcal Y$.
Moreover $\delta_0$ and $\Delta_0$ (\resp $\rho_0$) can be chosen arbitrary small (\resp large).

%
%
	
\subsection{Group action on the cone-off}

\paragraph{}
In this section $\rho$ is a real number, $X$ a $\delta$-hyperbolic length space and $\mathcal Y$ a collection of strongly quasi-convex subsets of $X$.
We assume that $\delta \leq \delta_0$, $\Delta(\mathcal Y) \leq \Delta_0$ and $\rho \geq \rho_0$ where $\delta_0$, $\Delta_0$ and $\rho_0$ are the constants given by \autoref{res: curvature of dot X - general case}.
In particular $\dot X$ is $\dot \delta$-hyperbolic, with $\dot\delta = 900 \boldsymbol \delta$.
Without loss of generality we can assume that $\rho_0 \geq 10^{10}\boldsymbol \delta$.

\paragraph{}
Let $G$ be a group acting by isometries on $X$.
We assume that $G$ acts by left translation on $\mathcal Y$.
The action of $G$ on $X$ can be extended by homogeneity into an action on $\dot X$ as follows.
Let $Y \in \mathcal Y$ and $x = (y,r)$ be a point of the cone $Z(Y)$.
Let $g$ be an element of $G$.
Then $g\cdot x$ is the point to the cone $Z(gY)$ defined by $g\cdot x = (gy,r)$.
It follows from the definition of the metric of $\dot X$ that $G$ acts by isometries on $\dot X$.

\paragraph{}Recall that the map $X \rightarrow \dot X$ is $1$-Lipschitz. 
Therefore if an element of $G$ is elliptic (\resp parabolic) for the action of $G$ on $X$, then it is elliptic (\resp parabolic or elliptic) for the action on $\dot X$.

\begin{prop}
\label{res: WPD action on the cone-off}
	If the action of $G$ on $X$ is WPD so is the one on $\dot X$.
\end{prop}

\begin{proof}
	We apply here the criterion provided by \autoref{res: WPD local implies WPD global}.
	Let $g$ be an element of $G$ which is loxodromic for its action on $\dot X$.
	Its cylinder $Y_g$ in the cone-off $\dot X$ cannot be bounded, therefore it contains a point $y$ in $X$.
	Being loxodromic as an isometry of $\dot X$, $g$ is also loxodromic as an isometry of $X$. 
	In particular it satisfies the WPD property.
	Consequently there exists $n \in \N$ such that the set $S$ of elements of $u \in G$ satisfying $\dist[X]{uy}y \leq \pi\sinh(\dot \delta)$ and $\dist[X]{ug^ny}{g^ny} \leq \pi\sinh(\dot \delta)$ is finite.
	Note that the point $y' = g^ny$ also belongs to $Y_g \subset \dot X$.
	Let $u \in G$ such that $\dist[\dot X]{uy}y \leq 2\dot \delta$ and $\dist[\dot X]{uy'}{y'} \leq 2\dot \delta$.
	It follows from \autoref{res : comparison metric X and dot X} that  
	\begin{equation*}
		\mu\left( \dist[X]{uy}y\right) \leq \dist[\dot X]{uy}y \leq 2\dot \delta < 2 \rho.
	\end{equation*}
	By  \autoref{res: map mu},  $\dist[X]{uy}y \leq \pi\sinh(\dot \delta)$.
	Similarly we get $\dist[X]{uy'}{y'} \leq \pi\sinh(\dot \delta)$.
	Thus $u$ belongs to the finite set $S$.
	By \autoref{res: WPD local implies WPD global}, $g$ is WPD for the action of $G$ on $\dot X$.
\end{proof}

\paragraph{}
For the rest of this section, we assume that the action of $G$ on $X$ (and thus on $\dot X$) is WPD.
We now study how the type of an elementary subgroup of $G$ for its action on $X$ is related to the one for its action on $\dot X$

\begin{lemm}
\label{res: pre type elementary subgroup cone-off}
	Let $H$ be a subgroup of $G$.
	If $H$ is elliptic (\resp parabolic, loxodromic) for the action on $X$, then $H$ is elliptic (\resp parabolic or elliptic, elementary) for the action on $\dot X$.
\end{lemm}

\begin{proof}
	We use one more time the fact that the map $X \rightarrow \dot X$ is $1$-Lipschitz.
	In particular, it directly gives that if $H$ is elliptic for the action on $X$ so is it for the action on $\dot X$.
	Assume now that $H$ is parabolic for the action on $X$.
	Since $\partial H \subset \partial X$ has only one point, $H$ does not contain a loxodromic element for the action on $X$, and thus for the action on $\dot X$.
	According to \autoref{res: two boundary points give loxodromic isometry} (applied in $\dot X$) $H$ is either parabolic or elliptic.
	Assume now that $H$ is loxodromic for the action on $X$.
	By \autoref{res: normalizer of loxodromic virtually cyclic}, $H$ contains a loxodromic element $g$ such that $\langle g \rangle$ has finite index in $H$.
	It follows from \autoref{res: isometry with finite index in a subgroup} that $H$ is elementary for the action of $G$ on $\dot X$.
\end{proof}

%

\begin{prop}
\label{res: type elementary subgroup cone-off - parabolic}
	Let $H$ be a subgroup of $G$.
	If $H$ is parabolic for its action on $\dot X$, then so is it for its action on $X$.
\end{prop}

\begin{proof}
	We denote by $\xi$ the unique point of $\partial H \subset \partial \dot X$.
	Let $\gamma : \R_+ \rightarrow \dot X$ be a $L_S\dot\delta$-local $(1,\dot\delta)$-quasi-geodesic such that $\lim_{t \rightarrow + \infty} \gamma(t) = \xi$.
	Let $g \in H$. 
	By \autoref{res: partial characteristic subset for parabolic} there exists $t_0$ such that for every $t \geq t_0$, $\dist[\dot X]{g\gamma(t)}{\gamma(t)} \leq 166\dot \delta$.
	Since the path $\gamma$ is infinite there exists $t \geq t_0$ such that $x = \gamma(t)$ lies in $X$.
	We obtain 
	\begin{equation*}
		\mu \left(\dist[X] {gx}x\right) \leq \dist[\dot X] {gx}x \leq 166\dot \delta < 2 \rho.
	\end{equation*}
	Hence $\dist[X] {gx}x \leq \pi \sinh (83\dot \delta)$ (see \autoref{res: map mu}).
	Consequently for every $g \in H$, $\len[espace=X] g \leq \pi \sinh (83\dot \delta)$.
	Therefore $H$ cannot contain a loxodromic element for its action of $G$ on $X$.
	By \autoref{res: two boundary points give loxodromic isometry} $H$ is either elliptic or parabolic for this action.
	It follows from \autoref{res: pre type elementary subgroup cone-off} that $H$ is parabolic for this action.
\end{proof}

%% file: 3_small_cancellation.tex
\section{Small cancellation theory}
\label{sec: small cancellation theory}

%
%

\subsection{Small cancellation theorem}

\paragraph{}
In this section $X$ is a $\delta$-hyperbolic length space, endowed with an action by isometries of a group $G$.
We assume that the action of $G$ on $X$ is WPD and that $G$ is non-elementary.
We consider a family $\mathcal Q$ of pairs $(H,Y)$ such that $Y$ is a strongly quasi-convex subset of $X$ and $H$ a subgroup of $\stab Y$.
We suppose that $G$ acts on $\mathcal Q$  and $\mathcal Q/G$ is finite. 
The action of $G$ on $\mathcal Q$ is defined as follows: for every $g \in G$, for every $(H,Y) \in \mathcal Q$, $g \cdot(H,Y) = (gHg^{-1}, gY)$.
We denote by $K$ the (normal) subgroup generated by the subgroups $H$ with $(H,Y) \in \mathcal Q$.
The goal is to understand the action of the quotient $\bar G = G/K$ on an appropriate space.
We use here small cancellation theory.

\paragraph{}
In order to control the small cancellation parameters at each step of the final induction (see \autoref{res: SC - induction lemma} and \autoref{res : SC - partial periodic quotient}), we will not use the properties of the whole group $G$ but only of a normal subgroup .
To that end, we need additional assumptions on the subgroups $H$ that can be stated as follows.
Let $N$ be a normal subgroup of $G$ \emph{without involution} and containing $K$.
We denote by $\bar N$ the image $N/K$ of $N$ in $\bar G$.
As a subgroup of $G$, the action of $N$ on $X$ is WPD.
Note that the definition of a primitive element (see \autoref{def: primitive element}) depends on the ambient group.
Let $g$ be a loxodromic element of $N$.
The maximal loxodromic subgroup of $N$ containing $g$ is a priori smaller than the one of $G$ with the same property.
Consequently $g$ might be primitive viewed as an element of $N$ but a proper power as an element of $G$.
With this idea in mind we can now state our last assumptions.
For every $(H,Y) \in \mathcal Q$, we suppose that there exists a loxodromic element $h \in N$ which is primitive \emph{as an element of $N$} and an odd integer $n \geq 100$ such that
\begin{enumerate}
	\item $H$ is the cyclic subgroup generated by $\langle h^n \rangle$.
	\item $Y$ is the cylinder $Y_h$ of $h$.
\end{enumerate}
For the rest of this section, we will refer to $h$ as a \emph{primitive root} of $H$.

\paragraph{}
Let $(H,Y) \in \mathcal Q$. 
By construction $\stab Y$ is a loxodromic subgroup of $G$.
In particular it admits a maximal normal finite subgroup $F$ (see \autoref{res: structure of loxodromic subgroups}).
Every element $u \in F$ fixes pointwise $\partial Y$.
Since $N$ has no involution, every element of $\stab Y \cap N$  also fixes pointwise $\partial Y$.
In particular it is either elliptic and thus belongs to $F$ or loxodromic.
We will very often use this property later.
According to \autoref{res: normalizer of loxodromic virtually cyclic}, $H$ has finite index in $\stab Y$.
Thus $\stab Y/ H$ is finite.

\paragraph{}
Let $\rho >0$.
We denote by $\dot X$ the cone-off of radius $\rho$ over $X$ relative to the collection $\set{Y}{(H,Y) \in \mathcal Q}$.
As we explained previously, $G$ acts by isometries on $\dot X$.
The space $\bar X$ is defined to be the quotient of $\dot X$ by $K$.
It is endowed with an action on $\bar G$.
We denote by $\zeta : \dot X \rightarrow \bar X$ the canonical map from $\dot X$ to $\bar X$.
We write $v(\mathcal Q)$ for the subset of $\dot X$ consisting in all apices of the cones $Z(Y)$ where $(H,Y) \in \mathcal Q$.
Its image in $\bar X$ is denoted by $\bar v(\mathcal Q)$.

\paragraph{}
To study the action of $\bar G$ on $\bar X$ we consider two parameters which respectively play the role of the length of the largest piece and the length of the smallest relation in the usual small cancellation theory.
Both quantities are measured with the metric of $X$.
\begin{eqnarray*}
	\Delta (\mathcal Q) & = & \sup \set{\diam\left(Y_1^{+ 5\delta} \cap Y_2^{+ 5\delta}\right)}{(H_1,Y_1)\neq (H_2,Y_2) \in \mathcal Q} \\
	T(\mathcal Q) & = & \inf\set{\len h}{h \in H, (H,Y) \in \mathcal Q}. 
\end{eqnarray*}

\begin{theo}[Small cancellation theorem]{\rm \cite[Proposition 6.7]{Coulon:2013tx}} \quad
\label{res: SC - small cancellation theorem}
	There exist positive constants $\delta_0$, $\Delta_0$ and $\rho_0$ which do not depend on $X$, $G$ or $\mathcal Q$ and  satisfying the following property.
	Assume that $\delta \leq \delta_0$, $\rho \geq \rho_0$.
	If in addition $\Delta(\mathcal Q) \leq \Delta_0$ and $T(\mathcal Q) \geq 8\pi \sinh \rho$ then the following holds.
	\begin{enumerate}
		\item The cone-off $\dot X$ is a $\dot \delta$-hyperbolic length space with $\dot \delta = 900 \boldsymbol \delta$.
		\item The space $\bar X$ is a $\bar \delta$-hyperbolic length space with $\bar \delta = 64.10^4 \boldsymbol \delta$.
		\item The group $\bar G$ acts by isometries on $\bar X$
		\item For every $(H,Y) \in \mathcal Q$, the projection $G \twoheadrightarrow \bar G$ induces an isomorphism from $\stab Y/H$ onto its image.
	\end{enumerate}
\end{theo}

\rems 
Note that $\dot \delta \leq \bar \delta$, thus $\dot X$ is also $\bar \delta$-hyperbolic.
This is not really accurate, however it will allow us to decrease the number of parameters we have to deal with.
As in \autoref{res: curvature of dot X - general case}, the constants $\delta_0$ and $\Delta_0$ (\resp $\rho_0$) can be chosen arbitrary small (\resp large).
From now on, we will always assume that $\rho_0 > 10^{20} L_S\boldsymbol \delta$ whereas $\delta_0, \Delta_0 < 10^{-10}\boldsymbol\delta$.
These estimates are absolutely not optimal.
We chose them very generously to be sure that all the inequalities that we might need later will be satisfied.
What really matters is their orders of magnitude recalled below.
\begin{displaymath}
	\max\left\{\delta_0, \Delta_0\right\} \ll \boldsymbol \delta  \ll \rho_0 \ll \pi \sinh \rho_0.
\end{displaymath}
An other important point to remember is the following.
The constants $\delta_0$, $\Delta_0$ and $\pi \sinh \rho_0$ are used to describe the geometry of $X$ whereas $\boldsymbol \delta$ and $\rho_0$ refers to the one of $\dot X$ or $\bar X$.
From now on and until the end of \autoref{sec: small cancellation theory} we assume that $X$, $G$ and $\mathcal Q$ are as in \autoref{res: SC - small cancellation theorem}.
In particular $\dot X$ and $\bar X$ are $\bar \delta$-hyperbolic.

\notas In this section we work with three metric spaces namely $X$, its cone-off $\dot X$ and the quotient $\bar X$.
Since the map $X \hookrightarrow \dot X$ is an embedding we use the same letter $x$ to designate a point of $X$ and its image in $\dot X$.
We write $\bar x$ for its image in $\bar X$.
Unless stated otherwise, we keep the notation $\distV$ (without mentioning the space) for the distances in $X$ or $\bar X$.
The metric on $\dot X$ will be denoted by $\distV[\dot X]$.

%
%

\subsection{The geometry of $\bar X$}

\paragraph{}
In this section we look more closely at the geometric features of the space $\bar X$.

\paragraph{Quasi-geodesics in $\bar X$.}
We look here at the quasi-geodesics of $\bar X$.
We explain how to build quasi-geodesic path of $\bar X$ that avoid the set of apices $\bar v(\mathcal Q)$.
In addition, we prove that the set $\bar v(\mathcal Q)$ of apices of $\bar X$ contains at least 2 elements.

\begin{prop}{\rm \cite[Corollary 3.12]{Coulon:2013tx}} \quad
\label{res: dot X covers bar X}
	The space $\dot X \setminus v(\mathcal Q)$ is a covering space of $\bar X \setminus \bar v(\mathcal Q)$.
	Let $l >0$ and $x \in \dot X$.
	If for every $v \in v(\mathcal Q)$, $\dist[\dot X] vx \geq l$, then for every $g \in K\setminus\{1\}$, $\dist[\dot X]{gx}x \geq \min\{ 2l, \rho/5\}$.
\end{prop}

\begin{prop}{\rm \cite[Proposition 3.15]{Coulon:2013tx}} \quad
\label{res: dot X - bar X - isometry far away from the apices}
	Let $r \in (0, \rho/20]$.
	Let $x \in \dot X$ in the $(\rho -2r)$-neighborhood of $X$.
	The map $\zeta : \dot X \rightarrow \bar X$ induces an isometry from $B(x,r)$ onto $B(\bar x, r)$.
\end{prop}

\rem On important consequence of this proposition is the following.
If $\gamma \colon I \rightarrow \dot X$ is a $(1,l)$-quasi-geodesic of $\dot X$ that stays in the $d$-neighborhood of $X$, then for every $L < (\rho-d)/2$, the path $\bar \gamma \colon I \rightarrow \bar X$ induced by $\gamma$ is an $L$-local $(1,l)$-quasi-geodesic of $\bar X$.
In particular, if $d$ and $l$ are sufficiently small, we can apply the stability of quasi-geodesics (see \autoref{res: stability (1,l)-quasi-geodesic}) to the path $\bar \gamma$.

\begin{lemm}
\label{res: bar X - small curves close to the apex}
	Let $(H,Y) \in \mathcal Q$ and $r \in [0, \rho)$.
	We denote by $v$ the apex of the cone $Z(Y)$ and by $h$ a primitive root of $H$.
	Let $\bar x$ and $\bar x'$ be two points of $\bar X$ such that $\dist{\bar x}{\bar v} = \dist{\bar x'}{\bar v} = r$.
	There exists a path $\bar \gamma : I \rightarrow \bar X$ joining $\bar x$ to $\bar x'$ such that
	\begin{enumerate}
		\item for every $t \in I$, $\dist{\bar \gamma(t)}{\bar v} = r$,
		\item $\bar \gamma$ is rectifiable and its length is at most $(\sinh r /\sinh \rho) \len {h^n}$.
	\end{enumerate}
\end{lemm}

\begin{proof}
	By construction, the ball $B(\bar v, \rho)$ is the image of $Z(Y)\setminus Y$ in $\bar X$.
	In particular $\bar x$ and $\bar x'$ are the respective images of points $x = (y,r)$ and $x' = (y',r)$ of $Z(Y)$.
	The cylinder $Y$ is $27\delta$-close to any $\delta$-nerve of $h^n$.
	Therefore, by translating if necessary $x'$ by $h^n$ we can always assume that $\dist y{y'} \leq \len {h^n} /2 + 55\delta$.
	Since $Y$ is strongly quasi-convex there exists a path $\gamma : I \rightarrow Y$ whose length (as a path of $X$) is at most
	\begin{equation*}
		L(\gamma) \leq \dist y{y'} + 9\delta \leq \frac 12 \len {h^n} + 64\delta \leq \len {h^n}.
	\end{equation*}
	We define the path $\tilde \gamma : I \rightarrow Z(Y)$ by $\tilde \gamma(t) = (\gamma(t),r)$.
	By \autoref{res: cone - small curves close to the apex}, the length of $\tilde \gamma$ (as a path of $\dot X$) is at most $(\sinh r/\sinh \rho) \len {h^n}$.
	Moreover for every $t \in I$, $\dist [\dot X]{\tilde \gamma(t)}{v} = r$.
	We choose for $\bar \gamma$ the path of $\bar X$ induced by $\tilde \gamma$.
	It satisfies the statement of the lemma.
\end{proof}

\begin{lemm}
\label{res: bar X - quasi-geodesics avoiding apices - point in X}
	For every $\bar x, \bar x' \in \bar X \setminus \bar v(\mathcal Q)$, for every $l >0$, there exists a $(1,l)$-quasi-geodesic of $\bar \gamma : I \rightarrow \bar X$ joining $\bar x$ to $\bar x'$ such that for every $t \in I$, $\bar \gamma(t)$ does not belong to $\bar v(\mathcal Q)$.
\end{lemm}

\begin{proof}
	By assumption $\mathcal Q/G$ is finite.
	Therefore there exists $D \geq 0$ such that for every $(H,Y) \in \mathcal Q$, if $h$ is a primitive root of $H$ then $\len{h^n} \leq D$.
	Let $\bar x$ and $\bar x'$ be two points of $\bar X$.
	Two apices of $\bar v(\mathcal Q)$ are at least at a distance $2 \rho$ far apart from each other.
	Therefore there are only a finite number, say $M$, of points $\bar v \in \bar v(\mathcal Q)$ such that $\gro {\bar x}{\bar x'}{\bar v} \leq \bar \delta$.
	
	\paragraph{}
	Fix $\eta \in (0,2\bar \delta)$ such that $M \sinh(2\eta)D/\sinh \rho + (M+1)\eta \leq l$.
	Let $\bar \gamma : \intval ab \rightarrow \bar X$ be a $(1, \eta)$-quasi-geodesic joining $\bar x$ to $\bar x'$.
	For every $t \in \intval ab$, $\gro {\bar x}{\bar x'}{\bar \gamma(t)} \leq \eta/2$.
	Hence by choice of $\eta$, there are at most $M$ distinct points of $\bar v(\mathcal Q)$ lying on $\bar \gamma$.
	We denote them $\bar v_1 = \bar \gamma(t_1), \dots, \bar v_m = \bar \gamma(t_m)$ (with $m \leq M$). 
	Without loss of generality we can assume that $t_1 < t_2 < \dots < t_m$.
	Note that for every $j \in \intvald 1{m-1}$, $\dist {t_{j+1}}{t_j} \geq 2 \rho$.
	Let $j \in \intvald 1m$.
	The path $\bar \gamma$ is not a geodesic, thus it can go through the same apex several times.
	However if we let $s_j = \max\{t_j - 2\eta, a\}$ and $s'_j = \min \{t_j + 2\eta, b\}$, then $ \bar \gamma$ restricted to $\intval a{s_j}$ or $\intval{s'_j}b$ does not contain $\bar v_j$.
	Moreover, by \autoref{res: bar X - small curves close to the apex} there exists a path $\bar \gamma_j$ joining $\bar \gamma(s_j)$ to $\bar \gamma(s'_j)$ whose length is at most $\sinh(2\eta)D/\sinh \rho + \eta$ that does not contain any apex.
	We now define a new path $\bar \gamma'$ joining $\bar x$ to $\bar x'$ as follows.
	For every $j \in \intvald 1m$, we replace the subpath of $\bar \gamma$ between times $s_j$ and $s'_j$ by the path $\bar \gamma_j$.
	By construction, $\bar \gamma'$ does not contain any apex.
	Moreover its length is at most
	\begin{equation*}
		L(\bar \gamma') \leq L(\bar \gamma) + M \sinh(2\eta)D/\sinh \rho + M\eta \leq L(\bar \gamma) + l-\eta.
	\end{equation*}
	Since $\bar \gamma$ is a $(1, \eta)$-quasi-geodesic, $\bar \gamma'$ is a $(1, l)$-quasi-geodesic.
\end{proof}

\begin{lemm}
\label{res: bar X - quasi-geodesics avoiding apices - point in  partial X}
	Let $\bar x \in \bar X \setminus \bar v(\mathcal Q)$ and $\bar \xi \in \partial \bar X$.
	For every $L>0$, for every $l >0$, there exists a $L$-local $(1,l+10\bar \delta)$-quasi-geodesic $\bar \gamma : \R_+ \rightarrow \bar X$ joining $\bar x$ to $\bar \xi$ such that for every $t \in \R_+$, $\bar \gamma(t)$ does not belong to $\bar v(\mathcal Q)$.
\end{lemm}

\begin{proof}
	The proof works just as the one of \autoref{res: quasi-rays}, using \autoref{res: bar X - quasi-geodesics avoiding apices - point in X} to avoid the apices of $\bar X$.
\end{proof}

\begin{prop}
\label{res: two distinct apices in bar X}
	The set $\bar v(\mathcal Q)$ contains at least two distinct apices.
\end{prop}

\begin{proof}
	Let $(H,Y) \in \mathcal Q$.
	We assumed that the action of $G$ on $X$ is non elementary.
	Therefore there exists $g \in G$ such that $\stab Y \neq g \stab Y g^{-1}$.
	In particular $(H,Y) \neq g(H,Y)$.
	In other words $\mathcal Q$ contains at least two elements.
	Let $\eta \in(0,\delta)$.
	We now fix two distinct apices $v$ and $v'$ in $v(\mathcal Q)$ such that for every $w,w' \in v(\mathcal Q)$, $\dist[\dot X] v{v'} \leq \dist[\dot X] w{w'} + \eta$.
	Let $\gamma : I \rightarrow \dot X$ be a $(1,\eta)$-quasi-geodesic joining $v$ to $v'$.
	Recall that two distinct points of $v(\mathcal Q)$ are at least $2\rho$ far apart from each other.
	Therefore there exist $t$ and $t'$ in $I$ such that $\dist[\dot X]{\gamma(t)}v = \rho/4+\eta$ and $\dist[\dot X]{\gamma(t')}{v'} = \rho/4+\eta$.
	For simplicity of notation, we put $x = \gamma(t)$ and $x' = \gamma(t')$.
	It follows from the triangle inequality that $\dist[\dot X] x{x'}\geq 3\rho/2- 2\eta$.
	We claim that $\gamma$ restricted to $\intval t{t'}$ lies in the $3\rho/4$-neighborhood of $X$.
	First, $\gamma$ being a $(1,\eta)$-quasi-geodesic, for every $s \in \intval t{t'}$, $\dist[\dot X]{\gamma(s)}v \geq \rho/4$ and $\dist[\dot X]{\gamma(s)}{v'} \geq \rho/4$.
	We now focus on the other apices of $\dot X$.
	Let $w \in v(\mathcal Q) \setminus\{v,v'\}$.
	Assume that $w$ lies in the $\rho/4$-neighborhood of $\gamma$.
	It follows that 
	\begin{equation*}
		\min\left\{ \dist[\dot X]vw, \dist[\dot X]{v'}w \right\} \leq \frac 12 \dist[\dot X]v{v'} + \rho/4 + \eta.
	\end{equation*}
	However two distinct apices of $v(\mathcal Q)$ are at a distance at least $2\rho$ apart, hence
	\begin{equation*}
		\min\left\{ \dist[\dot X]vw, \dist[\dot X]{v'}w \right\} < \dist[\dot X]v{v'}- \eta,
	\end{equation*}
	which contradicts our choice of $v$ and $v'$.
	Consequently $\gamma$ restricted to $\intval t{t'}$ lies in the $3\rho/4$-neighborhood of $X$.
	Let $\bar \gamma : \intval t{t'} \rightarrow \bar X$ be the path of $\bar X$ induced by the restriction of $\gamma$ to $\intval t{t'}$.
	According to \autoref{res: dot X - bar X - isometry far away from the apices} $\bar \gamma$ is a $\rho/10$-local $(1,\eta)$-quasi-geodesic.
	By stability of quasi-geodesics it is a (global) $(2,\eta)$-quasi-geodesic.
	Consequently 
	\begin{equation*}
		\dist{\bar x}{\bar x'} \geq \frac12 \dist t{t'} - \eta  \geq \frac 12 \dist[\dot X] x{x'} - \eta \geq 3\rho/4 -2\eta > \rho/2 + 2\eta.
	\end{equation*}
	It implies that $\bar v \neq \bar v'$.
	Indeed by construction $\dist{\bar x}{\bar v} \leq \rho/4 +\eta$ and $\dist{\bar x'}{\bar v'} \leq \rho/4 +\eta$.
	Thus if $\bar v$ and $\bar v'$ were the same apex we would have $\dist{\bar x}{\bar x'} \leq \rho/2 + 2\eta$.
\end{proof}

\paragraph{Stabilizers of apices.}
The next results deals with the stabilizers of the apices in $\bar X$.
In particular given an apex $\bar v \in \bar v(\mathcal Q)$, we are interested in how an element $\bar g \in \stab{\bar v}$ acts on the ball $B(\bar v, \rho)$.

\begin{prop}
\label{res: SC - unique fixed apex and co - prelim}
	Let $(H,Y) \in \mathcal Q$.
	We denote by $v$ the apex of the cone $Z(Y)$, $F$ the maximal finite normal subgroup of $\stab Y$ and $h \in N$ a primitive root of $H$.
	Let $u \in F$.
	Let $g \in \stab Y$ such that $(n/4)\len[stable] h \leq \len[stable]g \leq (3n/4)\len [stable]h$.
	\begin{enumerate}
		\item For every $\bar x \in B(\bar v, \rho)$, $\dist{\bar u\bar x}{\bar x} \leq \bar \delta$.
		\item For every $\bar x \in \bar X$, $\gro {\bar x}{\bar u\bar g\bar x}{\bar v} \leq 2\bar \delta$.
	\end{enumerate}
\end{prop}

\begin{proof}
	According to \autoref{res : cylinder fixed by normal finite subgroup}, $u$ moves the points of $Y$ by a distance at most $85\delta$.
	Let $\bar x$ be a point of $B(\bar v,\rho)$.
	In particular, $\bar x$ is the image of a point $x=(y,r)$ of the cone $Z(Y)$.
	Since the map $\zeta : \dot X \rightarrow \bar X$ shortens the distances, we get
	\begin{equation*}
		\dist{\bar u\bar x}{\bar x} \leq \dist[\dot X]{u x}{x}  \leq \dist[Y] {uy}y \leq \dist {uy}y + 8\delta \leq 93\delta \leq \bar \delta,
	\end{equation*}
	which proves the first point.
	Moreover we have
	\begin{equation*}
		\dist{ugy}y \geq \dist{gy}y - 85\delta \geq \len{g} - 85\delta \geq n\len[stable]h/4 - 85\delta  \geq T(\mathcal Q)/4 - 85\delta \geq \pi \sinh \rho.
	\end{equation*}
	It follows that $\dist[\dot X]{ugx}x = 2r$.
	On the other hand, $y$ is a point of the cylinder of $h$ and therefore is contained in the $38\delta$-neighborhood the axis of $g$ (see \autoref{res: Yg in quasi-convex g-invariant}). 
	Hence 
	\begin{equation*}
		\dist{ugy}y \leq \dist{gy}y + 85\delta \leq \len g + 169\delta \leq 3n\len[stable]h/4 + 233\delta \leq \len{h^n} -n\len[stable]h/4 + 233\delta.
	\end{equation*}
	Consequently $\dist[Y]{ugy}y \leq  \len{h^n} - \pi \sinh \rho$.
	According to \autoref{res: metric quotient cone}, $\dist{\bar u \bar g\bar x}{\bar x} = \dist[\dot X]{ugx}x  = 2r$.
	By construction $\dist{\bar x}{\bar v} = \dist{\bar u\bar g\bar x}{\bar v} = r$, thus $\gro {\bar x}{\bar u\bar g\bar x}{\bar v} =0$.
	
	\paragraph{}
	Assume now that $\bar x$ is a point of $\bar X \setminus B(\bar v,\rho)$.
	Let $\bar z$ be an $\bar \delta$-projection of $\bar x$ on $B(\bar v,\rho)$.
	It follows from the hyperbolicity condition~(\ref{eqn: hyperbolicity condition 1}) combined with the previous observation that 
	\begin{equation*}
		\min \left\{ \dist{\bar v}{\bar z} - \gro{\bar x}{\bar v}{\bar z}, \gro{\bar x}{\bar u \bar g \bar x}{\bar v}, \dist{\bar v}{\bar u \bar g\bar z} - \gro{\bar u \bar g \bar x}{\bar v}{\bar u \bar g\bar z} \right\}
		\leq \gro {\bar z}{\bar u \bar g\bar z}{\bar v} + 2 \bar \delta = 2 \bar \delta.
	\end{equation*}
	Since $\bar x$ does not belong to $B(\bar v,\rho)$ we have $\dist {\bar z}{\bar v} \geq \rho - \bar \delta$.
	By projection on a quasi-convex, $\gro{\bar x}{\bar v}{\bar z} \leq 3\bar \delta$.
	The minimum in the previous inequality is therefore achieved by $\gro {\bar x}{\bar u \bar g \bar x}{\bar v}$.
	Hence $\gro {\bar x}{\bar u \bar g\bar x}{\bar v} \leq 2\bar \delta$.
\end{proof}

\begin{coro}
\label{res: SC - unique fixed apex and co}
	Let $(H,Y) \in \mathcal Q$ and $v$ be the apex of $Z(Y)$.
	Let $\bar g \in \stab {\bar v}$.
	If $\bar g$ is not the image of an elliptic element of $\stab Y$ then there exists $k \in \Z$ such that the axis of  $\bar g^k$ is contained in the $6\bar\delta$-neighborhood of $\{\bar v\}$.
	In particular, $\bar v$ is the unique apex of $\bar X$ fixed by $\bar g$.
\end{coro}

\begin{proof}
	Let $F$ be the maximal finite normal subgroup of $\stab Y$.
	We denote by $r$ a primitive element of $\stab Y$ and $h \in N$ a primitive root of $H$.
	Recall that as an element of $G$, $h$ is not necessarily primitive.
	Let $g$ be a preimage of $\bar g$ in $\stab Y$.
	By assumption $h$ and $g$ are loxodromic elements, thus they fix pointwise $\partial Y$.
	Consequently there is $u,u' \in F$ and $p,q \in \Z$ such that $h = r^pu$ and $g = r^qu'$.
	Since $g$ is not the image of an elliptic element of $\stab Y$, $q \neq 0 \mod np$.
	Thus there exist integers $k, l \in \Z$ such that $m = kq +lnp$ is  between $np/3$ and $2np/3$.
	Since $F$ is a normal subgroup of $\stab Y$, there exists $f \in F $ such that $h^{ln}g^k = r^mf$.
	In particular $\bar g^k  = \bar r^m\bar f$.
	By construction $(n/4)\len[stable]h \leq \len[stable] {r^m} \leq (3n/4)\len[stable] h$.
	Let $\bar x$ be a point of $\bar X$.
	According to \autoref{res: SC - unique fixed apex and co - prelim}, $\gro{\bar x}{\bar g^k \bar x}{\bar v} \leq  2\bar\delta$.
	Thus $\dist{\bar g^k \bar x}{\bar x} \geq 2 \dist{\bar v}{\bar x} - 4\delta$.
	However $\bar g$ fixes $\bar v$, thus $\len {\bar g^k} = 0$.
	Consequently the points of $\bar X$ which belong to the axis of $\bar g^k$ are $6\bar \delta$-close to $\bar v$.
\end{proof}

\begin{coro}
\label{res: large rotation around an apex in bar X}
	Let $v \in v(\mathcal Q)$.
	There exists $\bar g \in \stab {\bar v}$  such that for every $\bar x  \in \bar X$, $\gro{\bar x}{\bar g \bar x}{\bar v} \leq 2\bar \delta$ and $\gro{\bar g^{-1}\bar x}{\bar g \bar x}{\bar x} \leq \dist{\bar g \bar x}{\bar x}/2 +4 \bar \delta$.
\end{coro}

\begin{proof}
	By construction there exists $(H,Y) \in \mathcal Q$ such that $v$ is the apex of $Z(Y)$.
	We denote by $h \in N$ a primitive root of $H$.	
	According to our assumption $H$ is the cyclic group generated by $h^n$ with  $n \geq 100$.
	Thus there exists an integer $m$ such that $n/4 \leq m \leq 3m/8$.
	We put $\bar g = \bar h^m$.
	Let $\bar x \in \bar X$.
	By \autoref{res: SC - unique fixed apex and co - prelim} we get that $\gro {\bar g \bar x}{\bar x}{\bar v} \leq 2\bar \delta$ and $\gro{\bar g^{-1}\bar x}{\bar g\bar x}{\bar v} \leq 2 \bar \delta$.
	It follows from the triangle inequality that
	\begin{equation*}
		\gro{\bar g^{-1}\bar x}{\bar g\bar x}{\bar x} \leq \dist {\bar x}{\bar v} + 2\bar \delta \leq \dist{\bar g \bar x}{\bar x}/2 + 4\bar \delta. \qedhere
	\end{equation*}
\end{proof}

\paragraph{Lifting figures.}
The next propositions are two key ingredients for the coming study of $\bar G$.
We explain how some figure in $\bar X$ can be lift into a picture of $\dot X$.

\begin{prop}{\rm \cite[Proposition. 3.21]{Coulon:2013tx}} \quad
\label{res: SC - lifting quasi-convex}
	Let $\alpha \geq 0$ and $d \geq \alpha$.
	Let $\bar Z$ be an $\alpha$-quasi-convex subset of $\bar X$.
	Let $\bar z_0$ be a point of $\bar Z$ and $z_0$ a preimage of $\bar z_0$ in $\dot X$.
	We assume that for every $\bar v \in \bar v(\mathcal Q)$, $\bar Z$ does not intersect $B(\bar v, \rho/20 + d + 10\bar \delta)$.
	Then there exists a subset $Z$ of $\dot X$ satisfying the following properties.
	\begin{enumerate}
		\item The map $\zeta : \dot X \rightarrow \bar X$ induces an isometry from $Z$ onto $\bar Z$.
		\item For every $\bar g \in \bar G$, for every subset $Z'$ of $Z$ if $\bar g\bar Z'$ lies in the $d$-neighborhood of $\bar Z$ then there exists a preimage $g \in G$ of $\bar g$ such that for every $z \in Z$ and $z' \in Z'$, $\dist[\dot X]{gz'}z = \dist{\bar g\bar z'}{\bar z}$.
		\item The projection $\pi : G \rightarrow \bar G$ induces an isomorphism from $\stab Z$ onto $\stab {\bar Z}$ 
	\end{enumerate}
\end{prop}

Let $\bar \gamma : I \rightarrow \bar X$ be a quasi-geodesic $\bar X$.
If  $\bar \gamma$ stays far away from the apices (e.g. if it is a small path with endpoints in $\zeta(X)$) \autoref{res: SC - lifting quasi-convex} provides a tool to lift it in an appropriate manner as a path $\gamma$ of $\dot X$ with the same length.
In particular if an isometry $\bar g \in \bar G$ moves the endpoints of $\bar \gamma$ by a small distance, one can find a preimage $g \in G$ of $\bar g$ that moves the endpoints of $\gamma$ by a small distance. 
This property might fail if $\bar \gamma$ is an arbitrary long path (take a path with loops  around apices).
The next proposition explain how to handle that case.

\begin{prop}
\label{res: SC - lifting group prelim}
	Let $x$ and $y$ be two points of $X$.
	Let $\gamma : \intval ab \rightarrow \dot X$ be a path joining $x$ to $y$ such that the path $\bar \gamma : \intval ab \rightarrow \bar X$ that it induces is an $L_S \bar \delta$-local $(1, 100\bar \delta)$-quasi-geodesic.
	Let $S$ be a subset of $G$ such that for every $g \in S$, $\dist[\dot X] {gx}x \leq \rho/50$ and $\dist{\bar g \bar y}{\bar y} \leq \rho/50$.
	In addition, we suppose that $S$ satisfies the following property.
	Let $(H,Y) \in \mathcal Q$.
	Let $v$ be the apex of $Z(Y)$ and $F$ the maximal finite normal subgroup of $\stab Y$.
	If $\bar v$ is $9\rho/10$-close to $\bar \gamma$, then for every $g \in S$, $\bar g$ is the image of an element of $F$.
	Under these assumptions, for every $g \in S$, $\dist[\dot X]{gy}y = \dist{\bar g\bar y}{\bar y}$.
\end{prop}

	\rem Let $g \in S$.
	By assumption $\bar \gamma$ is a local quasi-geodesic
	It follows from \autoref{res: quasi-convexity distance isometry} that for every $t \in \intval ab$, $\dist{\bar g \bar \gamma(t)}{\bar \gamma(t)} \leq \rho/50 + 116\bar \delta$.
	If this path was entirely contained in the neighborhood of $\zeta(X)$ we could apply \autoref{res: SC - lifting quasi-convex} to lift it in $\dot X$.
	However $\bar \gamma$ might go through the cones.
	Therefore we need to subdivide $\bar \gamma$ into subpaths of two types: the ones which stay far away from the apices and the ones contained in a cone.
	Once this is done, we lift them one after the other.

\begin{proof}
	Let $v_1, \dots, v_m$ be the apices of $v(\mathcal Q)$ which are $9\rho/10$-close to $\gamma$.
	For every $j \in \intvald 1m$, we denote by $\gamma(c_j)$ a projection of $v_j$ on $\gamma$.
	By reordering the apices we can always assume that $c_1 \leq c_2 \leq \dots \leq c_m$. 
	For simplicity of notation we put $c_0 = a$ and $c_{m+1} = b$.
	Let $j \in \intvald 1m$.
	Since $\bar \gamma$ is an $L_S \bar \delta$-local $(1,l)$-quasi-geodesic of $\bar X$ so is $\gamma$.
	I particular, it is a (global) $(2, 100 \bar \delta)$-quasi-geodesic.
	Hence we can find $b_{j-1} \in (c_{j-1},c_j]$ and  $a_j \in [c_j,c_{j+1})$ with the following properties.
	\begin{enumerate}
		\item $\dist{v_j}{\gamma (b_{j-1})} = 9\rho/10$ and $\dist{v_j}{\gamma (a_j)} = 9\rho/10$,
		\item $\gamma \cap B(v_j, 2\rho/5)$ is contained in $\gamma((b_{j-1},a_j))$
	\end{enumerate}
	In addition, we put $a_0 = a$, $a_{m+1} = b_m=b$ (see \autoref{fig: lift}).
	\begin{figure}[htbp]
	\centering
		\includegraphics[width=0.9\textwidth]{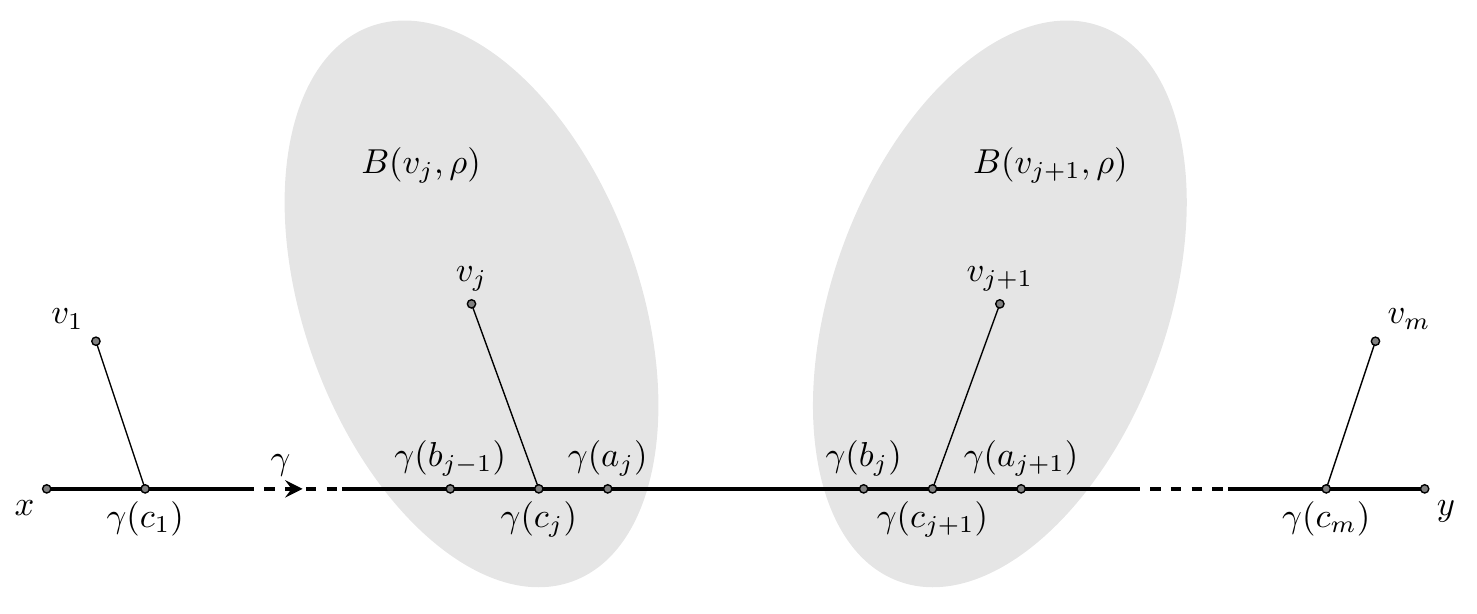}
	\caption{The cones intersecting $\gamma$.}
	\label{fig: lift}
	\end{figure}
	We claim that for every $j \in \intvald 0{m+1}$, for every $g \in S$,  we have
	\begin{equation*}
		\dist{\bar g\bar \gamma(a_j)}{\bar \gamma(a_j)} = \dist[\dot X]{g\gamma(a_j)}{\gamma(a_j)}.
	\end{equation*}
	The proof is by induction on $j$.
	If $j=0$ then $\gamma(a_j) = x$.
	The claim follows from the fact that the map $\zeta : \dot X \rightarrow \bar X$ induces an isometry from $B(x, \rho/20)$ onto $B(\bar x, \rho/20)$ (see \autoref{res: dot X - bar X - isometry far away from the apices}).
	Assume now that our claim is true for $j \in \intvald 0m$. 
	Since $\gamma$ is a local quasi-geodesic, $a_j \leq b_j$.
	We denote by $\bar \gamma_j$ the restriction of $\bar \gamma$ to $\intval{a_j}{b_j}$.
	By construction $\bar \gamma_j$ is $9\bar \delta$-quasi-convex and contained in the $3\rho/5$-neighborhood of $\zeta(X)$.
	Applying \autoref{res: SC - lifting quasi-convex} there exists a continuous path $\gamma_j : \intval{a_j}{b_j} \rightarrow \dot X$ starting at $\gamma(a_j)$ and lifting $\bar \gamma_j$ with the following property.
	Given $\bar g \in \bar G$, if $\bar g \bar \gamma_j$ lies in the $\rho/10$-neighborhood of $\bar \gamma_j$ then there exists $g \in G$ such that for every $t \in \intval{a_j}{b_j}$, $\dist{\bar g \bar \gamma_j(t)}{\bar \gamma_j(t)} = \dist[\dot X]{g\gamma_j(t)}{\gamma_j(t)}$.
	According to \autoref{res: dot X covers bar X}, $\dot X \setminus v(\mathcal Q)$ is a covering space of $\bar X \setminus \bar v(\mathcal Q)$.
	Thus $\gamma_j$ is exactly the restriction of $\gamma$ to $\intval {a_j}{b_j}$.
		
	\paragraph{}
	Take now an element $g$ in $S$ and write $\bar g$ for its image in $\bar G$.
	By assumption $\dist{\bar g\bar x}{\bar x} \leq \rho/50$ and $\dist{\bar g\bar y}{\bar y} \leq \rho/50$.
	It follows from \autoref{res: quasi-convexity distance isometry} that for every $t \in \intval ab$, $\dist{\bar g \bar \gamma(t)}{\bar \gamma(t)} \leq \rho/50+116 \bar \delta$.
	In particular $\bar g$ moves the points of $\bar \gamma_j$ by a distance at most $\rho/10$.
	Using the properties of the lift $\gamma_j$, there exists $u \in K$ such that for every $t \in \intval{a_j}{b_j}$, $\dist{\bar g \bar \gamma(t)}{\bar \gamma(t)} = \dist[\dot X]{gu\gamma(t)}{\gamma(t)}$.
	Thus $\dist{\bar g\bar \gamma(a_j)}{\bar \gamma(a_j)}=\dist[\dot X]{gu\gamma(a_j)}{\gamma(a_j)}$.
	On the other hand, using the induction assumption $\dist{\bar g\bar \gamma(a_j)}{\bar \gamma(a_j)} =\dist[\dot X]{g\gamma(a_j)}{\gamma(a_j)}$.
	It follows from the triangle inequality that
	\begin{equation*}
		\dist[\dot X]{u\gamma(a_j)}{\gamma(a_j)} \leq \dist[\dot X]{gu\gamma(a_j)}{\gamma(a_j)} + \dist[\dot X]{g\gamma(a_j)}{\gamma(a_j)} = 2 \dist{\bar g\bar \gamma(a_j)}{\bar \gamma(a_j)} \leq \rho/25 + 232\bar \delta.
	\end{equation*}
	However, $K\setminus\{1\}$ moves the points of the $\rho/10$-neighborhood of $X \subset \dot X$ by a distance at least $\rho /5$ (see \autoref{res: dot X covers bar X}).
	Consequently $u = 1$.
	In particular $\dist[\dot X]{g\gamma(b_j)}{\gamma(b_j)} = \dist{\bar g\bar\gamma(b_j)}{\bar\gamma(b_j)}$ is at most $\rho/50+116 \bar \delta$.
	If $j = m$, then $a_{m+1} = b_m$, thus the claim holds for $j+1$.
	Otherwise, $\gamma(b_j)$ is a point in the ball $B(v_{j+1}, 9\rho/10)$, thus $g$ necessarily belongs to $\stab {v_{j+1}}$.
	Moreover by assumption, $\bar g$ is the image of an element in the maximal normal finite subgroup $F_{j+1}$ of $\stab {v_{j+1}}$.
	Since $g$ moves the point $\gamma(b_j) \in B(v_{j+1}, \rho)$ by a small distance, $g$ is the elliptic preimage of $\bar g$.
	Therefore it moves all the points of $B(v_{j+1}, \rho)$ by a distance at most $\bar \delta$ (see \autoref{res: SC - unique fixed apex and co - prelim}).
	In particular, $\dist[\dot X]{g\gamma(a_{j+1})}{\gamma(a_{j+1})} \leq \bar \delta$.
	However, the map $\zeta : \dot X \rightarrow \bar X$ induces an isometry from the ball $B(\gamma(a_{j+1}), \rho/20)$ onto its image, hence $\dist[\dot X]{g\gamma(a_{j+1})}{\gamma(a_{j+1})} = \dist{\bar g \bar \gamma(a_{j+1})}{\bar \gamma(a_{j+1})}$.
	This proves our claim for $j+1$.
	The statement of the lemma follows from our claim for $j = m+1$.
\end{proof}

%
%

\subsection{Elementary subgroups}

\begin{prop}
\label{res: WPD action on bar X}
	The action of $\bar G$ on $\bar X$ is WPD.
\end{prop}

\begin{proof}
	Let $\bar g$ be a loxodromic element of $\bar G$.
	We claim that there exist $\bar y$ and $\bar y'$ in $Y_{\bar g}$ such that the set of elements $\bar u \in \bar G$ satisfying $\dist{\bar u\bar y}{\bar y} \leq \bar \delta$ and $\dist {\bar u \bar y'}{\bar y'} \leq \bar \delta$ is finite.
	\autoref{res: WPD local implies WPD global} will imply that $\bar g$ satisfies the WPD property.
	By replacing if necessary $\bar g$ by a power of $\bar g$ we can assume that $\len{\bar g} > L_S \bar \delta$.
	Let $\bar \gamma : \R \rightarrow \bar X$ be a $\bar \delta$-nerve of $\bar g$ and $T$ its fundamental length.
	By definition $\bar \gamma$ is contained in the cylinder $Y_{\bar g}$ of $\bar g$.
	We now distinguish two cases.
	
	\paragraph{}
	Assume first that there exists $\bar v \in \bar v(\mathcal Q)$ lying in the $\rho/10$-neighborhood of $\bar \gamma(\intval 0T)$.
	There is $(H,Y) \in \mathcal Q$ such that $\bar v$ is the image in $\bar X$ of the apex of $Z(Y)$.
	Let $\bar y = \bar y' = \bar \gamma(s)$ be a projection of $\bar v$ on $\bar \gamma(\intval 0T)$.
	Let $\bar u$ be an element of $\bar G$ such that $\dist {\bar u \bar y}{\bar y} \leq \bar \delta$.
	It follows from the triangle inequality that
	\begin{equation*}
		\dist {\bar u \bar v}{\bar v} \leq 2 \dist {\bar v}{\bar y} + \dist{\bar u \bar y}{\bar y} < 2\rho.
	\end{equation*}
	However two distinct apices of $\bar X$ are at a distance at least $2\rho$ apart.
	Thus $\bar u \bar v = \bar v$.
	Hence $\bar u$ belongs to the finite group $\stab {\bar v} = \stab Y/ H$, which proves our claim.
	
	\paragraph{}
	Assume now that for every $\bar v \in \bar v(\mathcal Q)$, $\bar \gamma(\intval 0T)$ does not intersect $B(\bar v, \rho/10)$.
	We put $\bar y = \bar \gamma(0)$ and denote by $y$ a preimage of $\bar y$ in $\dot X$.
	The set $\bar v(\mathcal Q)$ being $\bar G$-invariant, for every $\bar v \in \bar v(\mathcal Q)$, $\bar \gamma$ does not intersect $B(\bar v, \rho/10)$.
	Since $\len{\bar g} > L_S\bar \delta$, $\bar \gamma$ is a $9\bar \delta$-quasi-convex subset of $\bar X$.
	According to \autoref{res: SC - lifting quasi-convex}, there exists a map $\gamma : \R \rightarrow \dot X$ and a preimage $g$ of $\bar g$ with the following properties.
	\begin{enumerate}
		\item $y=\gamma(0)$.
		\item For every $t \in \R$, $\gamma(t)$ is a preimage in $\dot X$ of $\bar \gamma(t)$.
		\item For every $t \in \R$, $\gamma(t + T) = g \gamma(t)$.
		\item For every $s,t \in \R$, for every $\bar u \in \bar G$ satisfying $\dist{\bar u \bar \gamma(s)}{\bar \gamma(s)} \leq \bar \delta$ and $\dist{\bar u \bar \gamma(t)}{\bar \gamma(t)} \leq \bar \delta$, there exists a preimage $u\in G$ of $\bar u$ such that $\dist{u\gamma(s)}{\gamma(s)} \leq \bar \delta$ and $\dist{u\gamma(t)}{\gamma(t)} \leq \bar \delta$
	\end{enumerate}
	Recall that the map $\zeta : \dot X \rightarrow \bar X$ is $1$-Lipschitz.
	Thus $\bar g$ being a loxodromic isometry of $\bar X$, $g$ is a loxodromic isometry of $\dot X$.
	According to \autoref{res: WPD action on the cone-off}, the action of $G$ on $\dot X$ is WPD.
	Hence there exists $n \in \N$ such that the set $S$ of elements $u \in G$ satisfying $\dist[\dot X]{uy}y \leq \bar \delta$ and $\dist{ug^ny}{g^ny} \leq \bar \delta$ is finite.
	We put $y' = g^ny = \gamma(nT)$. 
	By construction $\bar y'$ is a point on $\bar \gamma \subset Y_{\bar g}$.
	Let $\bar u$ be an element of $\bar G$ such that $\dist{\bar u\bar y}{\bar y} \leq \bar \delta$ and $\dist {\bar u \bar y'}{\bar y'} \leq \bar \delta$.
	Using the last property of $\gamma$, we get that $\bar u$ is the image of an element in the finite set $S$, which proves our claim in the second case.
\end{proof}

\begin{prop}
\label{res: bar G non elementary}
	The group $\bar G$ is non-elementary (for its action on $\bar X$).
\end{prop}

\begin{proof}
	The idea of the proof is to exhibit two elements of $\bar G$ satisfying the criterion provided by \autoref{res: non elementary subgroup sufficient condition}.
	According to \autoref{res: two distinct apices in bar X}, $\bar v(\mathcal Q)$ contains two distinct apices $\bar v_1$ and $\bar v_2$.
	By \autoref{res: large rotation around an apex in bar X}, for each $j \in \{1,2\}$ there exists $\bar g_j \in \stab {\bar v_j}$ such that for every $\bar x \in \bar X$, 
	\begin{equation}
	\label{eqn: bar G non elementary - prelim}
		\gro{\bar g_j \bar x}{\bar x}{\bar v} \leq 2\bar \delta 
		\quad \text{and} \quad
		2\gro{\bar g_j^{-1}\bar x}{\bar g_j \bar x}{\bar x} \leq \dist{\bar g_j \bar x}{\bar x} + 8\bar \delta.
	\end{equation}
	Let $\bar x$ be a $\bar \delta$-projection of $\bar v_2$ on $B(\bar v_1, \rho)$.
	Recall that $B(\bar v_1, \rho)$ is $2 \bar \delta$-quasi-convex.
	Thus $\gro{\bar v_1}{\bar v_2}{\bar x} \leq 3\bar \delta$.
	Applying the hyperbolicity condition~(\ref{eqn: hyperbolicity condition 1}) we get
	\begin{equation}
	\label{eqn: bar G non elementary}
		\min\left\{ \dist {\bar v_1}{\bar x} - \gro{\bar x}{\bar g_1 \bar x}{\bar v_1} , \gro{\bar g_1\bar x}{\bar g_2 \bar x}{\bar x} , \dist{\bar v_2}{\bar x} - \gro{\bar g_2 \bar x}{\bar x}{\bar v_2}\right\} 
		\leq \gro{\bar v_1}{\bar v_2}{\bar x} + 2\bar \delta 
		\leq 5\bar \delta.
	\end{equation}
	By construction, $ \rho - \bar \delta \leq \dist{\bar v_1}{\bar x} \leq \rho$.
	Since $\bar v_1$ and $\bar v_2$ are $2\rho$ far apart we get $\dist{\bar v_2}{\bar x} \geq \rho$.
	Consequently the minimum in~(\ref{eqn: bar G non elementary}) can only be achieved by $\gro{\bar g_1\bar x}{\bar g_2 \bar x}{\bar x}$.
	Thus $\gro{\bar g_1\bar x}{\bar g_2 \bar x}{\bar x} \leq 5\bar \delta$.
	Similarly we prove that $\gro{\bar g_1^{\pm 1}\bar x}{\bar g_2^{\pm 1} \bar x}{\bar x} \leq 5\bar \delta$.
	However by construction $\gro{\bar g_1 \bar x}{\bar x}{\bar v_1} \leq 2\bar \delta$ and $\gro{\bar g_2 \bar x}{\bar x}{\bar v_2} \leq 2\bar \delta$.
	Thus $\dist{\bar g_1 \bar x}{\bar x} \geq 2 \dist{\bar x}{\bar v_1} - 4\bar \delta \geq 2\rho - 6\bar \delta$ and  $\dist{\bar g_2 \bar x}{\bar x} \geq 2\rho - 4\bar \delta$.
	Consequently,
	\begin{equation*}
		2\gro{\bar g_1^{\pm 1}\bar x}{\bar g_2^{\pm 1} \bar x}{\bar x} <  \min\left\{ \dist{\bar g_1 \bar x}{\bar x}, \dist{\bar g_2 \bar x}{\bar x}\right\} - 15\bar \delta.
	\end{equation*}
	The other inequalities needed to apply \autoref{res: non elementary subgroup sufficient condition} is given by (\ref{eqn: bar G non elementary - prelim}).
	It follows that $\bar g_1$ and $\bar g_2$ generate an non-elementary subgroup of $\bar G$.
\end{proof}

\begin{prop}
\label{res: SC - image of elementary subgroup}
	The image in $\bar G$ of an elliptic (\resp parabolic, loxodromic) subgroup of $G$ is elliptic (\resp parabolic or elliptic, elementary).
\end{prop}

\begin{proof}
	The map $X \rightarrow \bar X$ shortens the distance.
	Hence the proof works exactly as the one of \autoref{res: pre type elementary subgroup cone-off}.
\end{proof}

\begin{prop}
\label{res: SC - proj one-to-one on non-hyp elem sg}
	Let $E$ be a non-loxodromic elementary subgroup of $G$.
	Then the projection $\pi : G \twoheadrightarrow \bar G$ induces an isomorphism from $E$ onto its image.
\end{prop}

\begin{proof}
	Let $g$ be a non-trivial element of $E$.
	Since $E$ is not loxodromic, $g$ cannot be loxodromic, (see \autoref{res: parabolic have exactly one limit point}).
	In particular $\len[stable] g= 0$, thus $\len g \leq 32\delta$ (see \autoref{res: translation lengths}).
	We distinguish two cases.
	Assume first that $g$ does not act trivially on $X$.
	In particular, there exists a point $x \in X$ such that $\dist {gx}x >0$.
	Without loss of generality we can assume that $\dist{gx}x \leq 33\delta$.
	It follows that
	\begin{displaymath}
		0 < \mu\left(\dist {gx}x\right) \leq \dist[\dot X] {gx}x \leq \dist {gx}x \leq 33\delta.
	\end{displaymath}
	However the map $\zeta : \dot X \rightarrow \bar X$ induces an equivariant isometry from $B(x,\rho/20)$ onto its image.
	Therefore $\dist{\bar g \bar x}{\bar x} \neq 0$, hence $\bar g \neq 1$.
	Assume now that $g$ acts trivially on $X$.
	Let $(H,Y)\in \mathcal Q$.
	Then $g$ belongs to the stabilizer of $Y$.
	Moreover, being non-loxodromic $g$ does not belong to $H$, thus it induces a non-trivial element of $\stab Y/H$.
	However we know that $\stab Y/H$ embeds into $\bar G$.
	Therefore $\bar g \neq 1$.
\end{proof}

\paragraph{}From now on we are interested in the elementary subgroups of $\bar N$.
Our goal is to find a way, to lift any elementary subgroup of $\bar N$ in an elementary subgroup of $N$.
Recall that we assumed that $N$ is a normal subgroup \emph{without involution}.
Hence for every $(H,Y) \in \mathcal Q$, the elements of $\stab Y \cap N$ are either loxodromic or in the maximal normal finite subgroup of $\stab Y$.
On the other hand, the kernel $K$ of the projection $G \twoheadrightarrow \bar G$ is contained in $N$.
Thus for every $\bar g \in \bar N$, any preimage $g \in G$ of $\bar g$ belongs to $N$.

\paragraph{Elliptic subgroups.}
The following result follows the ideas of T~Delzant and M~Gromov in \cite{DelGro08}.

\begin{prop}
\label{res: SC  - lifting elliptic subgroups}
	Let $\bar E$ be an elliptic subgroup of $\bar N$.
	One of the following holds.
	\begin{enumerate}
		\item The subgroup $\bar E$ is isomorphic to an elliptic subgroup of $N$.
		\item There exists $\bar v \in \bar v(\mathcal Q)$ such that $\bar E$ is contained in $\stab{\bar v}$.
		Moreover there exists $\bar g \in \bar E$ such that $A_{\bar g}$ lies in the $6\bar \delta$-neighborhood of $\{\bar v\}$.
	\end{enumerate}
\end{prop}

\begin{proof}
	Recall that $C_{\bar E}$ is the set of points $\bar x \in \bar X$ such that for every $\bar g \in \bar E$, $\dist{\bar g\bar x}{\bar x} \leq 11\bar \delta$. 
	It is an $\bar E$-invariant $9\bar \delta$-quasi-convex (see \autoref{res: characteristic subset summary}).
	We distinguish two cases.
	Assume first that $C_{\bar E}$ contains a point $\bar x$ in the $50\bar \delta$-neighborhood of $\zeta(X)$.
	We write $\bar Z$ for the hull of the $\bar E$-orbit of $\bar x$ (see \autoref{def: hull}).
	It is an $\bar E$-invariant $6\bar \delta$-quasi-convex contained in the $56\bar \delta$-neighborhood of $\zeta(X)$.
	By \autoref{res: SC - lifting quasi-convex}, there exists a subset $Z$ of $\dot X$ such that the map $\zeta : \dot X \rightarrow \bar X$ induces an isometry from $Z$ onto $\bar Z$ and the projection $G \twoheadrightarrow \bar G$ induces an isomorphism from $\stab Z$ onto $\stab {\bar Z}$.
	In particular $\bar E$ is isomorphic to a subgroup $E$ of $\stab Z$.
	Let $x$ be the preimage of $\bar x$ in $Z$ and $y$ a projection of $x$ on $X$.
	Thus $\dist[\dot X] xy \leq 50\bar \delta$.
	Let $g \in E$ we have then 
	 \begin{displaymath}
	 	\mu\left(\dist{gy}y\right) \leq \dist[\dot X]{gy}y \leq \dist[\dot X]{gx}x + 100 \bar \delta = \dist[\bar X]{\bar g \bar x}{\bar x} +100\bar \delta \leq 111\bar \delta < 2\rho.
	 \end{displaymath}
 	It follows that $\dist[X]{gx}x \leq \pi \sinh(56\bar \delta)$ (see \autoref{res: map mu}).
	In particular $E$ has a bounded orbit in $X$, thus it is an elliptic subgroup of $G$. 
	
	\paragraph{} 
	Assume now that $C_{\bar E}$ does not contain a point $\bar x$ in the $50\bar \delta$-neighborhood of $\zeta(X)$.
	Since $C_{\bar E}$ is $9\bar \delta$-quasi-convex, there exists $\bar v \in \bar v(\mathcal Q)$ such that $C_{\bar E}$ lies in the ball $B(\bar v, \rho - 50\bar \delta)$.
	Let $\bar x$ be a point of $C_{\bar E}$.
	Any element $\bar g$ of $\bar E$ moves $\bar x$ by a distance at most $11\bar \delta$.
	The triangle inequality yields $\dist{\bar g\bar v}{\bar v} < 2 \rho$, hence $\bar g$ fixes $\bar v$.
	Consequently $\bar E$ is a subgroup of $\stab {\bar v}$.
	Note that there exists an element of $\bar g \in \bar E$ which is not the image of an elliptic element of $\stab Y$.
	Otherwise, \autoref{res: SC - unique fixed apex and co - prelim} would force $B(\bar v, \rho)$ to be contained in $C_{\bar E}$.
	It follows then from \autoref{res: SC - unique fixed apex and co} that there exists $k \in \Z$ such that the axis of $\bar g^k$ is contained in the $6\bar \delta$-neighborhood of $\{\bar v\}$.
\end{proof}

\begin{coro}
\label{res: SC - bar G no involution}
	The subgroup $\bar N$ has no involution.
\end{coro}

\begin{proof}
	Let $\bar g$ be an element of $\bar N$ and assume that $\bar g$ has order $2$.
	According to \autoref{res: SC  - lifting elliptic subgroups} there are two cases.
	\begin{enumerate}
		\item There exists a preimage $g \in N$ of $\bar g$ with order $2$, which contradicts the fact that $N$ has no involution.
		\item There exists $\bar v \in \bar v(\mathcal Q)$ such that $\bar g$ belongs to $\stab {\bar v}$.
		There is $(H,Y) \in \mathcal Q$ such that $\bar v$ is the image of the apex of the cone $Z(Y)$.
		Let $g\in N$ be a preimage of $\bar g$ in $\stab Y$.
		Let $h$ be a primitive root of $(H,Y)$.
		By definition, $\stab Y\cap N$ is isomorphic to the semi-direct product $\sdp F \Z$ where $F$ is the maximal normal finite subgroup of $\stab Y \cap N$ and $\Z$ the subgroup generated by $h$ acting by conjugacy on $F$.
		In particular there exists $u \in F$ and $m \in \Z$ such that $g = h^mu$.
		We noticed that $\stab {\bar v}$ is isomorphic to $\stab Y /H$.	
		Consequently there exists $p \in \Z$ such that $h^{pn}=g^2 = h^{2m}(h^{-m}uh^mu)$.
		Thus $h^{-m}uh^mu = 1$ and $pn = 2m$.
		However $n$ is odd, thus $n$ divides $m$.
		It follows $\bar g$ is the image of $u$.
		Restricted to $F$ the projection $G \twoheadrightarrow \bar G$ is one-to-one, hence $u$ has order $2$.
		It contradicts again the fact that $N$ has no involution.
		\end{enumerate}
	Thus $\bar N$ cannot contain an involution.
\end{proof}

\begin{prop}
\label{res: elliptic subgroup equal in the quotient}
	Let $E$ be an elliptic subgroup of $N$ (for its action on $X$).
	Let $S$ be a subset of $G$ and $y$ a point of $X$ such that for every $u \in S$, $\dist[\dot X]{uy}y < \rho/100$.
	If the image $\bar S$ of $S$ in $\bar G$ is contained in $\bar E$, then there exists $g \in K$ such that $gSg^{-1}$ lies in $E$.
\end{prop}

\begin{proof}
	We fix a point $x$ in $C_E \subset X$.
	There exists $g \in K$ such that $\dist[\dot X]{gy}x \leq \dist{\bar y}{\bar x} + \bar \delta$.
	By \autoref{res: SC - proj one-to-one on non-hyp elem sg}, the map $G \rightarrow \bar G$ induces an isomorphism from $E$ onto its image.
	We denote by $S'$ the preimage of $\bar S$ in $E$. 
	We claim that $z=gy$ is hardly moved by the elements of $S'$.
	Let $\gamma : I \rightarrow \dot X$ be a $(1,\bar \delta)$-quasi-geodesic joining $x$ to $z$.
	Let $\bar \gamma : I \rightarrow \bar X$ the path of $\bar X$ induced by $\gamma$.
	By choice of $g$ the length of $\bar \gamma$ satisfies
	\begin{equation*}
		L(\bar \gamma) \leq L(\gamma) \leq \dist[\dot X]zx + \bar \delta \leq \dist{\bar z}{\bar x} + 2\bar \delta.
	\end{equation*}
	Hence $\bar \gamma$ is a $(1, 2\bar \delta)$-quasi-geodesic of $\bar X$.
	Let $u$ be an element of $S$ and $u'$ the preimage of $\bar u$ in $S'$.
	We are going to apply \autoref{res: SC - lifting group prelim} with the path $\gamma$ and the set $\{u'\}$.
	Since $u'$ belongs to $E$ we have $\dist[\dot X]{u'x}x \leq \dist{u'x}x \leq 11\delta\leq \bar \delta$.
	On the other hand $g$ lies in $K$ and $\bar u = \bar u'$ in $\bar S$, thus
	\begin{equation*}
		\dist{\bar u' \bar z}{\bar z} = \dist{\bar u \bar y}{\bar y} \leq \dist[\dot X]{uy}y < \rho/100.
	\end{equation*}
	Let $(H,Y) \in \mathcal Q$.
	Let $v$ be the apex of $Z(Y)$.
	Assume that $\bar u'$ belongs to $\stab {\bar v}$.
	If $\bar u'$ is not the image of an element in the maximal normal finite subgroup of $\stab Y$ then by \autoref{res: SC - unique fixed apex and co}, the characteristic subset $C_{\bar E}$ lies in the $15\bar \delta$-neighborhood of $\{\bar v\}$.
	However $\bar x$ is by construction a point of this characteristic subset. 
	Contradiction.
	It follows then from \autoref{res: SC - lifting group prelim} that $\dist[\dot X]{u'z}z = \dist{\bar u'\bar z}{\bar z} \leq \rho/100$, which proves our claim.
	Applying the triangle inequality we get   
	\begin{equation*}
		\dist[\dot X]{gug^{-1}z}{u'z} \leq \dist[\dot X]{gug^{-1}z}z + \dist[\dot X]{u'z}z =  \dist[\dot X]{uy}y+ \dist[\dot X]{u'z}z \leq \rho/50.
	\end{equation*}
	However $\bar u = \bar u'$, thus $u'gu^{-1}g^{-1}$ belongs to $K$.
	Applying \autoref{res: dot X covers bar X}, we get $u' = gug^{-1}$.
	In particular $gug^{-1}$ belongs to $E$.
\end{proof}

\begin{coro}
\label{res: elliptic equal in the quotient}
	Let $u$ and $u'$ be two elements of $N$.
	We assume that $\len u < \rho/100$ and $u'$ is elliptic (for the action on $X$).
	If $\bar u = \bar u'$ then $u$ and $u'$ are conjugated in $G$.
\end{coro}

\begin{proof}
	We apply \autoref{res: elliptic subgroup equal in the quotient} with the elliptic subgroup $E = \langle u' \rangle$ and the set $S = \{u\}$.
	In particular there exists $g \in K$ such that $gug^{-1}$ belongs to $E$.
	However by \autoref{res: SC - proj one-to-one on non-hyp elem sg}, the map $G \twoheadrightarrow \bar G$ induces an isomorphism from $E$ onto its image.
	It follows that $gug^{-1} = u'$.
\end{proof}

\begin{coro}
\label{res: elliptic conjugated in the quotient}
	Let $u$ and $u'$ be two elements of $N$.
	We assume that $\len u < \rho/100$ and $u'$ is elliptic (for the action on $X$).
	If $\bar u$ and $\bar u'$ are conjugated in $\bar G$ then $u$ and $u'$ are conjugated in $G$.
\end{coro}

\begin{proof}
	Assume that $\bar u$ and $\bar u'$ are conjugated in $\bar G$.
	In particular there exists $g \in G$ such that $\bar u= \bar g \bar u' \bar g^{-1}$.
	However $gu'g^{-1}$ is also an elliptic element of $N$.
	The corollary follows from \autoref{res: elliptic equal in the quotient} applied to $u$ and $gu'g^{-1}$.
\end{proof}

\paragraph{Parabolic subgroups.}
\autoref{res: SC  - lifting elliptic subgroups} explains how we can lift an elliptic subgroup of $\bar N$ into a particular subgroup of $N$.
We need a similar procedure for parabolic subgroups of $\bar N$.
This the purpose of \autoref{res: lifting parabolic subgroup - key proposition} to \autoref{res: lifting parabolic subgroup - isomorphism}.
Let $\bar E$ be a parabolic subgroup of $\bar N$ (for its action on $\bar X$).
We denote by $\bar \xi$ the unique point of $\partial \bar E \subset \partial \bar X$.
By \autoref{res: parabolic subgroup}, $\stab{\bar \xi}$ is a parabolic subgroup of $\bar G$.
We also fix a point $x_0$ in $X$ and write $\bar x_0$ for its image in $\bar X$.
According to \autoref{res: bar X - quasi-geodesics avoiding apices - point in  partial X}, there exits an $L_S\bar \delta$-local $(1, 11\bar \delta)$-quasi-geodesic $\bar \gamma : \R_+ \rightarrow \bar X$ joining $\bar x_0$ to $\bar \xi$ and avoiding the points of $\bar v(\mathcal Q)$.
Recall that $\dot X \setminus v(\mathcal Q)$ is a covering space of $\bar X \setminus \bar v(\mathcal Q)$ (see \autoref{res: dot X covers bar X}).
Therefore there exists a continuous path $\gamma : \R_+ \rightarrow \dot X$ starting at $x_0$ such that for every $t \in \R_+$, $\gamma(t)$ is a preimage of $\bar \gamma(t)$.
Since the map $\dot X \setminus v(\mathcal Q) \rightarrow \bar X \setminus \bar v(\mathcal Q)$ is a local isometry (see \autoref{res: dot X covers bar X}), $\gamma$ is a $L_S\bar\delta$-local $(1,11\bar \delta)$-quasi-geodesic of $\dot X$.
In particular it defines a point $\xi = \lim_{t \rightarrow + \infty}\gamma(t)$ in the boundary at infinity of $\dot X$.
Our goal is to prove that $\stab \xi$ is a parabolic subgroup of $G$ (for its action on $\dot X$ and thus on $X$) and that the map $G \twoheadrightarrow \bar G$ induces an isomorphism from $\stab \xi \cap N$ onto $\stab {\bar \xi} \cap \bar N$.
The next proposition is the key result for our proof.

\begin{prop}
\label{res: lifting parabolic subgroup - key proposition}
	Let $\bar g \in \stab {\bar \xi} \cap \bar N$.
	There exists a preimage $g \in N$ of $\bar g$ and $t_0 \in \R_+$ such that for every $t \geq t_0$, $\dist[\dot X]{g\gamma(t)}{\gamma(t)} \leq 183\bar \delta$.
	In particular $g$ belongs to $\stab \xi$.
\end{prop}

\begin{proof}
	By \autoref{res: partial characteristic subset for parabolic}, there exists $t_0 \in \R_+$ such that for every $t \geq t_0$, $\dist{\bar g \bar \gamma(t)}{\bar \gamma(t)} \leq 166\bar \delta$.
	Without loss of generality, we can assume that $\gamma(t_0)$ lies in $X$.
	However the map $\zeta: \dot X \rightarrow \bar X$ induces an isometry from $B(\gamma(t_0), \rho/20)$ onto $B(\bar \gamma(t_0), \rho/20)$ (see \autoref{res: dot X - bar X - isometry far away from the apices}).
	Therefore there exists a preimage $g \in N$ of $\bar g$ such that $\dist[\dot X]{g\gamma(t_0)}{\gamma(t_0)} = \dist{\bar g\bar \gamma(t_0)}{\bar \gamma(t_0)}$.
	Let $t \geq t_0$.
	Since $\gamma$ is an infinite continuous path, there exists $t_1 \geq t$ such that $\gamma(t_1)$ belongs to $X$.
	In addition, $\dist{\bar g\bar \gamma(t_1)}{\bar \gamma(t_1)} \leq 166\bar \delta$.
	Let $(H,Y) \in \mathcal Q$.
	We denote by $v$ the apex of the cone $Z(Y)$ and $F$ the maximal normal finite subgroup of $\stab Y$.
	Assume that $\bar v$ lies in the $9\rho/10$-neighborhood of $\bar \gamma$ restricted to $\intval {t_0}{t_1}$.
	It follows from the triangle inequality that $\dist{\bar g \bar v}{\bar v} <2 \rho$, thus $\bar g$ belongs to $\stab{\bar v}$.
	We claim that $\bar g$ is the image of an element of $F$.
	Assume on the contrary that this is false.
	According to \autoref{res: SC - unique fixed apex and co}, there exists $k \in \Z$ such that the axis of $\bar g^k$ is contained in  the $6\bar \delta$-neighborhood of $\{\bar v\}$.
	However $\bar g^k$ is also an element of $\stab {\bar \xi}$.
	Thus by \autoref{res: partial characteristic subset for parabolic}, there exists $t_2 \in \R_+$ such that for every $t \geq t_2$, $\dist{\bar g^k\bar \gamma(t)}{\bar \gamma(t)} \leq 166\bar \delta$, which leads to a contradiction.
	It follows then from \autoref{res: SC - lifting group prelim} that $\dist[\dot X]{g\gamma(t_1)}{\gamma(t_1)} = \dist{\bar g \bar \gamma(t_1)}{\bar \gamma(t_1)}$.
	Applying \autoref{res: quasi-convexity distance isometry}, we get that 
	\begin{equation*}
		\dist[\dot X]{g \gamma(t)}{\gamma(t)} \leq \max\left\{\dist[\dot X]{g\gamma(t_0)}{\gamma(t_0)} ,\dist[\dot X]{g\gamma(t_1)}{\gamma(t_1)} \right\}  + 2\gro {\gamma(t_0)}{\gamma(t_1)}{\gamma(t)} + 6\bar \delta \leq 183\bar \delta. \qedhere
	\end{equation*}
\end{proof}

\begin{prop}
\label{res: lifting parabolic subgroup - stab xi parabolic}
	The subgroup $\stab \xi$ is parabolic for the action of $G$ on $X$.
\end{prop}

\begin{proof}
	According to \autoref{res: type elementary subgroup cone-off - parabolic} it is sufficient to prove that $\stab \xi$ is parabolic for the action of $G$ on $\dot X$.
	Let $\bar g$ be an element of the parabolic subgroup $\bar E$.
	In particular $\bar g$ belongs to $\stab {\bar \xi} \cap \bar N$.
	We denote by $g \in \stab \xi$ the preimage of $\bar g$ given by \autoref{res: lifting parabolic subgroup - key proposition}.
	We write $E$ for the set of all preimages of elements of $\bar E$ obtained in this way.
	It is a subset of $\stab \xi$.
	Since $\bar E$ is parabolic the set $\bar E \cdot \bar x_0$ is not bounded.
	The map $\zeta: \dot X \rightarrow \bar X$ being $1$-Lipschitz $E\cdot x_0$ is unbounded as well (in $\dot X$).
	Consequently, $\stab \xi$ cannot be an elliptic subgroup of $G$.
	Therefore it is sufficient to show that $\stab \xi$ does not contain a loxodromic element.
	Assume on the contrary that there exists $g \in \stab \xi$ which is a loxodromic isometry of $\dot X$.
	By replacing if necessary $g$ by a power of $g$ we can assume that $\len[espace=\dot X] g > L_S\bar \delta$.
	As a loxodromic isometry $g$ fixes exactly two points of $\partial \dot X$, namely $g^-$ and $g^+$.
	Being an element of $\stab \xi$, $g$ also fixes $\xi$, thus $\xi \in \{ g^-, g^+\}$.
	We denote by $\sigma : \R \rightarrow \dot X$ a $\delta$-nerve of $g$.
	By hyperbolicity there exists $t_0 \in \R_+$ such that for every $t \geq t_0$, $\gamma(t)$ is $40\bar \delta$-close to $\sigma$, and thus in the $40\bar \delta$-neighborhood of $A_g$.
	It follows from \autoref{res : quasi-geodesic behaving like a nerve} that there is $\epsilon \in \{\pm 1\}$, such that for every $t \geq t_0$, $\dist[\dot X]{g^\epsilon\gamma(t)}{\gamma(t+ \len[espace=\dot X] g)} \leq 298\bar \delta$.
	Without loss of generality we can assume that $\epsilon = 1$.
	In particular for every $t \geq t_0$, $\dist{\bar g \bar \gamma(t)}{\bar \gamma(t+\len[espace=\dot X] g)} \leq 298 \bar \delta$.
	Hence $\bar g$ belongs to $\stab {\bar \xi}$.
	On the other hand $\bar \gamma$ is an $L_S\bar \delta$-local $(1,11\bar\delta)$-quasi-isometry.
	Thus for every $t \geq t_0$,
	\begin{equation*}
		\dist{\bar g\bar \gamma(t)}{\bar \gamma(t)} \geq \dist{ \bar \gamma(t+\len[espace=\dot X] g)}{\bar \gamma(t)} -298\bar \delta 
		\geq \min\{(L_S-22)\bar \delta ,\len[espace=\dot X] g -11\bar \delta\} -298\bar \delta> 166\bar \delta.
	\end{equation*}
	This last point contradicts \autoref{res: partial characteristic subset for parabolic} applied with the path $\bar \gamma$ and the parabolic subgroup $\stab {\bar \xi}$.	
\end{proof}

\begin{prop}
\label{res: lifting parabolic subgroup - isomorphism}
	The projection $G \twoheadrightarrow \bar G$ induces a one-to-one map from $\stab \xi$ into $\stab {\bar \xi}$.
	It sends $\stab \xi \cap N$ onto $\stab {\bar \xi}\cap \bar N$.
	The preimage $E$ of $\bar E$ in $\stab \xi \cap N$ is a parabolic subgroup of $G$ for its action on $X$.
\end{prop}

\begin{proof}
	Let $g$ be an element of $\stab \xi$.
	According to \autoref{res: lifting parabolic subgroup - stab xi parabolic} $\stab \xi$ is parabolic for the action of $G$ on $\dot X$.
	By \autoref{res: partial characteristic subset for parabolic}, there exits $t_0 \in \R_+$ such that for every $t \geq t_0$, $\dist[\dot X]{g\gamma(t)}{\gamma(t)} \leq 166\bar \delta$.
	It follows that for every $t \geq t_0$, $\dist{\bar g \bar \gamma(t)}{\bar \gamma(t)} \leq 166\bar \delta$.
	In particular $\bar g$ belongs to $\stab {\bar \xi}$.
	The subgroup $\stab \xi$ is elementary and not loxodromic, thus \autoref{res: SC - proj one-to-one on non-hyp elem sg} says that the map $G \twoheadrightarrow \bar G$ restricted to $\stab \xi$ is one-to-one.
	The surjectivity follows from  \autoref{res: lifting parabolic subgroup - key proposition}.
	According to \autoref{res: lifting parabolic subgroup - stab xi parabolic}, $E$ is elementary either elliptic or parabolic.
	However it cannot be elliptic otherwise its image $\bar E$ in $\bar G$ would be elliptic too.
\end{proof}

\paragraph{Loxodromic subgroups.}
We finish this study with the case of loxodromic subgroups.

\begin{prop}
\label{res: SC - loxodromic elementary subgroups of the quotient}
	Let $\bar E$ be a loxodromic subgroup of $\bar N$ (for its action on $\bar X$).
	Then $\bar E$ is isomorphic to a loxodromic subgroup $E$ of $N$ (for its action on $X$).
	Moreover if $\bar E$ is a \emph{maximal} loxodromic subgroup of $\bar N$, then $E$ is a also a \emph{maximal} loxodromic subgroup of $N$.
\end{prop}

\begin{proof}
	By \autoref{res: SC - bar G no involution}, $\bar N$ has no involution, thus $\bar E$ is not of dihedral type.
	We denote by $\bar F$ its maximal normal finite subgroup.
	There exists a loxodromic element $\bar g \in \bar E$ such that $\bar E$ is isomorphic to the semi-direct product $\sdp{\bar F}{\Z}$, where $\Z$ is the cyclic group generated by $\bar g$  acting by conjugacy on $\bar F$.
	According to \autoref{res : cylinder fixed by normal finite subgroup}, the cylinder $Y_{\bar g}$ of $\bar g$ is contained in the $37\bar \delta$-neighborhood of $C_{\bar F}$.
	Since $Y_{\bar g}$ contains bi-infinite local quasi-geodesics it cannot be a subset of a ball $B(\bar v, \rho)$ with $v \in v(\mathcal Q)$.
	Therefore we can find a point $\bar x$ in $C_{\bar F}$ which is at the same time in the $37\bar \delta$-neighborhood of $\zeta(X)$.
	Let $\bar Z$ be the hull of $\bar F \cdot \bar x$.
	It is an $\bar F$-invariant $6 \bar \delta$-quasi-convex subset of $\bar X$ contained in the $43\bar\delta$-neighborhood of $\zeta(X)$.
	It follows from \autoref{res: SC - lifting quasi-convex} that there exits a subset $Z$ of $\dot X$ with the following properties.
	\begin{enumerate}
		\item The map $\zeta : \dot X \rightarrow \bar X$ induces an isometry from $Z$ onto $\bar Z$.
		\item The projection $G \twoheadrightarrow \bar G$ induces an isomorphism from $\stab Z$ onto $\stab{\bar Z}$.
	\end{enumerate}
	We denote by $x$ the preimage of $\bar x$ in $Z$ and by $F$ the preimage of $\bar F$ in $\stab Z$.
	In particular, for every $u \in F$, $\dist[\dot X]{ux}x \leq 11\bar \delta$.
	There exists a preimage $g \in N$ of $\bar g$ such that $\dist[\dot X]{gx}x \leq \dist{\bar g\bar x}{\bar x} +\bar \delta$.
	As a preimage of $\bar g$, $g$ is loxodromic (for its action on $\dot X$ and thus on $X$).
	Let $\gamma : I \rightarrow \dot X$ be a $(1, \bar \delta)$-quasi-geodesic between $x$ and $gx$.
	We denote by $\bar \gamma$ the path of $\bar X$ induced by $\gamma$.
	Its length satisfies the following
	\begin{equation*}
		L(\bar \gamma) \leq L(\gamma) \leq \dist[\dot X]{gx}x +\bar \delta \leq \dist{\bar g \bar x}{\bar x} + 2 \bar \delta.
	\end{equation*}
	Thus $\bar \gamma$ is a $(1,2 \bar \delta)$-quasi-geodesic.
	Recall that $\bar F$ is a normal subgroup of $\bar E$, consequently $C_{\bar F}$ is $\bar g$-invariant.
	In particular, for every $\bar u \in \bar F$, $\dist{\bar u \bar g \bar x}{\bar g \bar x} \leq 11 \bar \delta$.
	We want to apply \autoref{res: SC - lifting group prelim}, with the path $\gamma$ and the whole group $F$ for the subset $S$.
	Let $v \in  v(\mathcal Q)$ such that $\bar v$ is in the $9\rho/10$-neighborhood of $\bar \gamma$.
	Let $u \in F$.
	\autoref{res: quasi-convexity distance isometry} combined with the triangle inequality says that $\bar u$ belongs to $\stab {\bar v}$.
	If $\bar u$ is not the image of an elliptic element of $\stab v$, then by \autoref{res: SC - unique fixed apex and co}, the characteristic subset $C_{\bar F}$ is contained in the $15 \bar \delta$-neighborhood of $\{\bar v\}$.
	This contradicts the fact that $\bar x$ belongs to this characteristic subset.
	Consequently, by \autoref{res: SC - lifting group prelim} for every $u \in F$, $\dist[\dot X]{ugx}{gx} = \dist{\bar u \bar g \bar x}{\bar g \bar x}$.
	Let $u$ be an element of $F$.
	Since $\bar g$ normalizes $\bar F$, the image of $g^{-1}ug$ in $\bar N$ is an element of $\bar F$.
	We denote by $u'$ its preimage in $F$.
	We claim that $g^{-1}ug= u' $.
	Using the conclusions of \autoref{res: SC - lifting quasi-convex} and \autoref{res: SC - lifting group prelim} we have
	\begin{equation*}
		\dist[\dot X]{u'x}x = \dist{\bar g^{-1}\bar u \bar g \bar x}{\bar x}= \dist{\bar u \bar g \bar x}{\bar g \bar x} \text{ and } \dist[\dot X]{g^{-1}ugx}{x} = \dist[\dot X]{ugx}{gx} = \dist{\bar u \bar g \bar x}{\bar g \bar x}.
	\end{equation*}
	However $\bar g \bar x$ belongs to $C_{\bar F}$.
	We get from the triangle inequality that
	\begin{equation*}
		\dist{g^{-1}u^{-1}gu'x}x  \leq \dist {u'x}x + \dist x{g^{-1}ugx} =  2\dist{\bar u \bar g \bar x}{\bar g \bar x} \leq 22\bar \delta.
	\end{equation*}
	Recall that $u'$ and $g^{-1}ug$ are two preimages of the same element of $\bar N$.
	Hence $g^{-1}u^{-1}gu'$ belongs to $K$.
	By \autoref{res: dot X covers bar X}, we have $g^{-1}ug = u'$, which completes the proof of our claim.
	Not only $g$ normalizes $F$ but the projection $G \twoheadrightarrow \bar G$ identifies the action by conjugacy of $g$ on $F$ and the one of $\bar g$ on $\bar F$.
	Consequently the subgroup $E$ of $N$ generated by $g$ and $F$ is a loxodromic subgroup isomorphic to $\bar E$.
	
	\paragraph{}Assume now that $\bar E$ is a maximal loxodromic subgroup of $\bar N$.
	Let us denote by $E'$ the maximal loxodromic subgroup of $N$ containing $E$.
	According to \autoref{res: SC - image of elementary subgroup}, the image $\bar E'$ of $E'$ in $\bar G$ is an elementary subgroup of $\bar N$.
	By maximality $\bar E' = \bar E$.
	Let $g'$ be an element of $E'$ whose image in $\bar G$ is trivial.
	According to \autoref{res: structure of loxodromic subgroups} $g'$ is either elliptic or loxodromic.
	If it is loxodromic, then $\langle g' \rangle$ has finite index in $E'$, thus $\bar E'$ is finite, which is impossible.
	Hence $g'$ is elliptic.
	Applying \autoref{res: SC - proj one-to-one on non-hyp elem sg} we get $g'=1$.
	In other words, the projection $G \twoheadrightarrow \bar G$ restricted to $E'$ is also one-to-one, which completes the proof of the last assertion.
\end{proof}

%
%

\subsection{Invariants of the action on $\bar X$.}

\paragraph{}
In \autoref{sec:group invariants} we associated several invariants to the action of a group on a hyperbolic space.
In this section we explain how the invariants for the action of $\bar N$ on $\bar X$ are related to the ones for the action of $N$ on $X$.

\begin{prop}
\label{res: SC - invariant exponent}
	The number $e(\bar N, \bar X)$ divides $e(N,X)$.
\end{prop}

\begin{proof}	
	It follows directly from \autoref{res: SC - loxodromic elementary subgroups of the quotient} and the definition of $e(\bar N, \bar X)$ (see \autoref{res: invariant e}).
\end{proof}

\begin{prop}
\label{res: SC - contol nu}
	The invariant $\nu(\bar N, \bar X)$ is at most $\nu(N,X)$.
\end{prop}

\begin{proof}
	Let $m \geq \nu(N,X)$ be an integer.
	Let $\bar g$ and $\bar h$ be two elements of $\bar N$ with $\bar h$ loxodromic such that $\bar g$, $\bar h^{-1}\bar g\bar h$,\dots, $\bar h^{-m}\bar g\bar h^m$ generate an elementary subgroup $\bar E$ of $\bar N$ which is not loxodromic.
	For every $j \in \intvald 0m$, we let $\bar g_j = \bar h^{-j}\bar g \bar h^j$. 
	We distinguish two cases.
	
	\paragraph{Case 1.}\emph{The subgroup $\bar E$ is elliptic and there exists $v \in v(\mathcal Q)$ such that $C_{\bar E}$ is contained in $B(\bar v, \rho - 50\bar \delta)$.} 
	The elements of $\bar E$ moves the points of $C_{\bar E}$ by a  distance at most $11\bar \delta$.
	Thus $\bar E$ is contained in $\stab{\bar v}$.
	Since $N$ has no involution, the set of elliptic elements of $\stab v \cap N$ forms a subgroup $F$ of $\stab v$ whose image in $\bar N$ will be denoted by $\bar F$.
	Note that at least one of the elements $\bar g_0, \dots, \bar g_m$ does not belong to $\bar F$.
	Indeed, if it was the case, $\bar E$ would be a subgroup of $\bar F$ and thus by \autoref{res: SC - unique fixed apex and co - prelim}, $B(\bar v, \rho)$ should lie in $C_{\bar E}$, which contradicts the assumption of Case 1. 
	Assume that $\bar g_0$ does not belong to $\bar F$ (the proof works similarly for the other elements).
	According to \autoref{res: SC - unique fixed apex and co}, $\bar v$ is the only apex fixed by $\bar g_0$.
	However $\bar g_1 = \bar h^{-1} \bar g_0 \bar h$ also belongs to $\bar E$ and thus $\stab {\bar v}$.
	It follows that $\bar h\bar v$ is also an apex fixed by $\bar g_0$.
	Hence $\bar h\bar v = \bar v$.
	Consequently $\bar g$ and $\bar h$ belong to $\stab {\bar v}$.
	Therefore they generate an elliptic subgroup of $\bar N$.
	
	\paragraph{Case 2.}\emph{Either $\bar E$ is elliptic and there is no $v \in v(\mathcal Q)$ such that $C_{\bar E}$ is contained in $B(\bar v, \rho - 50\bar \delta)$ or $\bar E$ is parabolic.} 
	Assume first that $\bar E$ is elliptic.
	Recall that $C_{\bar E}$ is $9\bar\delta$-quasi-convex.
	It follows that there exists a point $\bar x \in \zeta(X)$ in the $50\bar \delta$-neighborhood of $C_{\bar E}$.
	Let $x$ be a preimage of $\bar x$ in $\dot X$.
	Applying \autoref{res: SC - lifting quasi-convex} with the hull of $\bar E \cdot \bar x$ we get that there exists an elliptic subgroup $E$ of $N$ such that the map $G \twoheadrightarrow \bar G$ induces an isomorphism from $E$ onto $\bar E$ and for every $g \in E$, $\dist[\dot X]{gx}x  = \dist{\bar g\bar x}{\bar x}$.
	Assume now that $\bar E$ is parabolic.
	We denote by $\bar \xi$ the unique point of $\partial \bar E \subset \partial \bar X$.
	Let $x_0$ be a point of $X$.
	According to \autoref{res: bar X - quasi-geodesics avoiding apices - point in  partial X} there exits an $L_S \bar \delta$-local $(1, 11\bar \delta)$-quasi-geodesic $\bar \gamma : \R_+ \rightarrow \bar X$ joining $\bar x_0$ to $\bar \xi$ and avoiding the points of $\bar v(\mathcal Q)$.
	Recall that $\dot X \setminus v(\mathcal Q)$ is a covering space of $\bar X \setminus \bar v(\mathcal Q)$ (see \autoref{res: dot X covers bar X}).
	Therefore there exists a continuous path $\gamma : \R_+ \rightarrow \dot X$ starting at $x_0$ such that for every $t \in \R_+$, $\gamma(t)$ is a preimage of $\bar \gamma(t)$.
	Since the map $\dot X \setminus v(\mathcal Q) \rightarrow \bar X \setminus \bar v(\mathcal Q)$ is a local isometry (see \autoref{res: dot X - bar X - isometry far away from the apices}), $\gamma$ is an $L_S\bar\delta$-local $(1,11\bar \delta)$-quasi-geodesic of $\dot X$.
	In particular it defines a point $\xi = \lim_{t \rightarrow + \infty}\gamma(t)$ in the boundary at infinity of $\dot X$.
	It follows from \autoref{res: lifting parabolic subgroup - isomorphism} that the map $G \twoheadrightarrow \bar G$ induces an isomorphism from $\stab \xi \cap N$ onto $\stab {\bar \xi} \cap \bar N$.
	We denote by $E$ the preimage in $\stab \xi \cap N$ of $\bar E$.
	Applying \autoref{res: partial characteristic subset for parabolic}, for every $u \in E$,  there exists $t_0 \in \R_+$ such that for every $t \geq t_0$, $\dist[\dot X]{u\gamma(t)}{\gamma(t)} \leq 166\bar \delta$.
	
	\paragraph{}
	Finally, in both cases, there exists an elementary subgroup $E$ of $N$ which is not loxodromic and a point $x \in X$ with the following properties.
	\begin{itemize}
		\item The map $G \twoheadrightarrow \bar G$ induces an isomorphism from $E$ onto $\bar E$.
		\item For every $j \in \intvald 0m$, the preimage $g_j$ of $\bar g_j$ in $E$ satisfies $\dist[\dot X]{g_jx}x = \dist{\bar g_j \bar x}{\bar x} \leq 166\bar \delta$.
	\end{itemize}
	In particular for every $j \in \intvald 0{m-1}$ we have
	\begin{equation*}
		\dist{\bar g_j \bar h \bar x}{\bar h \bar x}  
		= \dist{\bar g_{j+1} \bar x}{\bar x} 
		= \dist[\dot X]{g_{j+1}x}x 
		\leq 166\bar \delta
	\end{equation*}
	Moreover, for every $\bar u \in \bar E$ there exists $\bar y$ in the $50 \bar \delta$-neighborhood of $\zeta(X)$ such that $\dist{\bar u \bar y}{\bar y} \leq 166 \bar \delta$.
	Let $(H,Y) \in \mathcal Q$ and $j \in \intvald 0{m-1}$.
	We denote by $v$ the apex of the cone $Z(Y)$ and $F$ the maximal finite normal subgroup of $\stab Y$.
	We claim that if $\bar g_j$ belongs to $\stab {\bar v}$ then $\bar g_j$ is the image of an element of $F$.
	Assume this is false.
	By \autoref{res: large rotation around an apex in bar X}, there exists $k \in \Z$ such that the axis of $\bar g_j^k$ is contained in the $6 \bar \delta$-neighborhood of $\{\bar v\}$.
	On the other hand, we explained that there exists $\bar y$ in the $50 \bar \delta$-neighborhood of $\zeta(X)$ such that $\dist{\bar u \bar y}{\bar y} \leq 166 \bar \delta$.
	Contradiction.
	
	\paragraph{}
	We now fix a preimage $h \in N$ of $\bar h$ such that  $\dist[\dot X] {hx}x \leq \dist{\bar h \bar x}{\bar x} + \bar \delta$.
	Let $\gamma : I \rightarrow \dot X$ be an $L_S\bar\delta$-local $(1,\bar \delta)$-quasi-geodesic joining $x$ to $hx$.
	The path $\bar \gamma : I \rightarrow \bar X$ induced by $\gamma$ is an $L_S\bar\delta$-local $(1,2\bar \delta)$-quasi-geodesic joining $\bar x$ to $\bar h \bar x$.
	We can now apply \autoref{res: SC - lifting group prelim} with the path $\gamma$ and the set $S = \{g_0, \dots g_{m-1}\}$.
	Thus for every $j \in \intvald 0{m-1}$, $\dist[\dot X]{g_jhx}{hx} = \dist{\bar g_j \bar h\bar x}{\bar h\bar x}$.
	We denote by $g$ the preimage in $E$ of $\bar g$ ($g=g_0$). 
	Let $j \in \intvald 0{m-1}$.
	We claim that $h^{-1}g_jh = g_{j+1}$.
	The proof is very similar to the one of \autoref{res: SC - loxodromic elementary subgroups of the quotient}.
	By choice of $h$ we have 
	\begin{equation*}
		\dist[\dot X]{g_{j+1}x}x = \dist{\bar g_{j+1} \bar x}{\bar x}= \dist{\bar g_j\bar h \bar x}{\bar h \bar x} \text{ and }\dist[\dot X]{h^{-1}g_jhx}x = \dist[\dot X]{g_jhx}{hx} = \dist{\bar g_j\bar h \bar x}{\bar h \bar x}.
	\end{equation*}
	Since $\bar x$ is moved by a small distance by $\bar g_{j+1}$ we get
	\begin{equation*}
		\dist[\dot X]{h^{-1}g_j^{-1}hg_{j+1}x}x \leq \dist[\dot X]{g_{j+1}x}x + \dist[\dot X]x{h^{-1}g_jhx} = 2\dist{\bar g_{j+1} \bar x}{\bar x} \leq 334\bar \delta.
	\end{equation*}
	However $g_{j+1}$ and $h^{-1}g_jh$ are two preimages of the same element of $\bar G$.
	Hence $h^{-1}g_j^{-1}hg_{j+1}$ belongs to $K$.
	By \autoref{res: dot X covers bar X} we get $h^{-1}g_jh = g_{j+1}$, which completes the proof of our claim.
	In particular for every $j \in \intvald 0m$, $h^{-j}gh^j$ belongs to $E$.
	Thus $g$, $ h^{-1} g h$, \dots, $ h^{-m} g h^m$ generate an elementary subgroup of $N$ which is not loxodromic.
	However we assumed that $m \geq \nu(N,X)$.
	Consequently $g$ and $h$ generate an elementary subgroup of $N$.
	By \autoref{res: SC - image of elementary subgroup}, $\bar g$ and $\bar h$ generate an elementary subgroup of $\bar N$.
	
	\paragraph{}In both cases $\bar g$ and $\bar h$ generate an elementary subgroup of $\bar N$.
	Thus $\nu(\bar N, \bar X) \leq m$.
\end{proof}

\begin{prop}
	\label{res: SC - lifting overlap of axes}
	Let $m$ be an integer.
	Let $\bar g_1, \dots, \bar g_m$ be a collection elements of $\bar G$ such that for every $j \in \intvald 1m$, $\len {\bar g_j} \leq L_S \bar \delta$.
	One of the following holds.
	\begin{enumerate}
		\item There exists $\bar v \in \bar v(\mathcal Q)$ such that for every $j \in \intvald 0m$, $\bar g_j$ belongs to $\stab{\bar v}$.
		\item There exist preimages $g_1, \dots, g_m$ in $G$  of $\bar g_1, \dots, \bar g_m$ such that for every $j \in \intvald 1m$, $\len {g_j} \leq \pi \sinh[(L_S+34)\bar \delta]$ and
		\begin{displaymath}
			A(\bar g_1,\dots, \bar g_m) \leq A(g_1, \dots, g_m) +\pi\sinh\left [\left(L_S+34\right)\bar \delta\right] + (L_S +45) \bar \delta.
		\end{displaymath}
	\end{enumerate}
\end{prop}

\rem 
Recall that $A(g_1, \dots, g_m)$ stands for
\begin{equation*}
	A(g_1, \dots, g_m) = \diam \left( A_{g_1}^{+ 13 \delta}\cap \dots \cap A_{g_m}^{+ 13 \delta}\right)
\end{equation*}
In the statement of the proposition all the metric objects are measured either with the distance of $X$ or $\bar X$, but not with the one of $\dot X$.

\begin{proof}
	Without loss of generality we can assume that the intersection of the $13\bar \delta$-neighborhoods of $A_{\bar g_1}, \dots, A_{\bar g_m}$ is not empty.
	Let us call $\bar Z$ this intersection.
	Assume that there exists $\bar v\in \bar v(\mathcal Q)$ and a point $\bar z \in \bar Z$ such that $\dist {\bar v}{\bar z} \leq \rho -  (L_S/2+17)\bar \delta$.
	By definition any $\bar g_j$ moves $\bar z$ by a distance smaller than $\len{\bar g_j} + 34\bar \delta \leq (L_S+34)\bar \delta$.
	It follows from the triangle inequality that all the $\bar g_j$ belongs to $\stab {\bar v}$, which provides the first case.
	
	\paragraph{}
	We now assume that for every $\bar v \in\bar  v(\mathcal Q)$, $\bar Z$ does not intersect the ball of center $\bar v$ and radius $ \rho -(L_S/2+17)\bar \delta$.
	By \autoref{res: intersection of quasi-convex}, $\bar Z$ is $7\bar \delta$-quasi-convex.
	Moreover, for every $j \in \intvald 1m$, $\bar g_j$ moves any point of $\bar Z$ by at most $(L_S+34)\bar \delta$.
	According to \autoref{res: SC - lifting quasi-convex}, there exists a subset $Z$ of $\dot X$ and a collection $g_1, \dots, g_m$ of preimages of $\bar g_1, \dots, \bar g_m$ satisfying the following properties.
	\begin{enumerate}
		\item The map $\zeta : \dot X \rightarrow \bar X$ induces an isometry from $Z$ onto $\bar Z$.
		\item For every $z \in Z$ for every $j \in \intvald 1m$ we have $\dist[\dot X]{g_jz}z = \dist{\bar g_j\bar z}{\bar z}$.
	\end{enumerate}
	We now denote by $\bar z$ and $\bar z'$ two points of $\bar Z$ such that
	\begin{displaymath}
		\dist{\bar z}{\bar z'} \geq A(\bar g_1,\dots, \bar g_m)  - \bar \delta. 
	\end{displaymath}
	The points $z$ and $z'$ stand for their preimages in $Z$.
	We write $x$ and $x'$ for respective projections of $z$ and $z'$ on $X$.
	By assumption, $\bar Z$ lies in the $(L_S/2+17)\bar \delta$-neighborhood of $\zeta(X)$.
	Thus $\dist[\dot X] xz, \dist[\dot X] {x'}{z'} \leq (L_S/2+17)\bar \delta$.
	In particular for every $j \in \intvald 1m$,
	\begin{displaymath}
		\mu \left(\dist{g_jx}x\right) \leq \dist[\dot X]{g_j x}{x} \leq \dist{\bar g_j \bar z}{\bar z} +  (L_S+34)\bar \delta \leq  2(L_S + 34)\bar \delta < 2 \rho.
	\end{displaymath}
	It follows that $\dist {g_jx}x \leq \pi \sinh [(L_S+34)\bar \delta]$ (see \autoref{res: map mu}).
	The same holds for $x'$.
	In particular,
	\begin{equation*}
		 \len {g_j} \leq \pi \sinh\left [\left(L_S+34\right)\bar \delta\right].
	\end{equation*}
	Moreover $x$ and $x'$ belong to the $C$-neighborhood of $A_{g_j}$ where $C = \pi \sinh [(L_S+34)\bar \delta]/2 + 3\bar \delta$ (see \autoref{res: axes is quasi-convex}).
	By \autoref{res: intersection of thickened quasi-convex},
	\begin{displaymath}
		\dist x{x'} \leq A(g_1, \dots, g_m) +\pi\sinh\left [\left(L_S+34\right)\bar \delta\right] + 10 \bar \delta
	\end{displaymath}
	On the other hand, the map $X \rightarrow \dot X$ shorten the distances.
	Therefore
	\begin{displaymath}
		\dist x{x'} \geq \dist[\dot X] x{x'} \geq \dist[\dot X] z{z'} - (L_S+34)\bar \delta \geq \dist {\bar z}{\bar z'} - (L_S+34)\bar \delta. 
	\end{displaymath}
	However by construction $\dist {\bar z}{\bar z'}  \geq A(\bar g_1,\dots, \bar g_m)  - \bar \delta$.
	The conclusion of the second case follows from the last two inequalities.
\end{proof}

\begin{coro}
\label{res: SC - control invariant A}
	The invariant $A(\bar N, \bar X)$ satisfies the following inequality
	\begin{equation*}
		A(\bar N, \bar X) \leq A(N, X) + (\nu + 4) \pi\sinh\left(2L_S\bar \delta\right),
	\end{equation*}
	where $\nu$ stands for $\nu = \nu(N,X)$.
\end{coro}

\begin{proof}
	Let $\bar \nu$ be the invariant $\bar \nu = \nu(\bar N, \bar X)$.
	We denote by $\mathcal A$ the set of $(\bar \nu + 1)$-uples $(\bar g_0, \dots, \bar g_{\bar \nu})$ of $\bar N$ such that $\bar g_0, \dots, \bar g_{\bar \nu}$ generate a non-elementary subgroup of $\bar N$ and for every $j \in \intvald 0{\bar \nu}$, $\len{\bar g_j} \leq L_S \bar \delta$.
	Let $(\bar g_0, \dots, \bar g_{\bar \nu})\in\mathcal A$.
	Since $\bar g_0, \dots, \bar g_{\bar \nu}$ do not generate an elementary subgroup of $\bar G$, there is no apex $\bar v \in \bar v(\mathcal Q)$ such that they all belong to $\stab {\bar v}$.
	According to \autoref{res: SC - lifting overlap of axes} there exist $g_0, \dots, g_{\bar \nu}$ respective preimages of $\bar g_0, \dots, \bar g_{\bar \nu}$ in $N$ such that
	\begin{enumerate}
		\item for every $j \in \intvald 0{\bar \nu}$,  $\len{g_j}\leq  \pi \sinh [(L_S+34)\bar \delta]$,
		\item $A(\bar g_0,\dots, \bar g_{\bar \nu}) \leq A(g_0, \dots, g_{\bar \nu}) + \pi\sinh [(L_S+34)\bar \delta] + (L_S +45) \bar \delta$.
	\end{enumerate}
	By \autoref{res: SC - image of elementary subgroup} the subgroup of $N$ generated by $g_0, \dots, g_{\bar \nu}$ is not elementary.
	In addition $\bar \nu \leq \nu(N,X)$ (see \autoref{res: SC - contol nu}).
	It follows from \autoref{res: overlap multiples axes} that 
	\begin{eqnarray*}
		A(\bar g_0,\dots, \bar g_{\bar \nu}) 
		& \leq & A(N,X) + (\nu + 3) \pi\sinh\left [\left(L_S+34\right)\bar \delta\right]  + (L_S +729)\bar \delta \\
		& \leq & A(N,X) + (\nu + 4) \pi\sinh\left(2L_S\bar \delta\right).
	\end{eqnarray*}
	This inequality holds for every $(\bar \nu +1)$-uple in $\mathcal A$, which provides the required conclusion.
\end{proof}

\begin{prop}
\label{res: SC - elements with only loxodromic lifts}
	We denote by $l$ the greatest lower bound on the stable translation length (in $X$) of loxodromic elements of $N$ which do not belong to some $\stab Y$ for $(H,Y) \in \mathcal Q$.
	Let $\bar g$ be an isometry of $\bar N$ which is not elliptic.
	If every preimage of $\bar g$ in $N$ is loxodromic then $\len[stable]{\bar g} \geq \min\left\{ \kappa l, \bar \delta\right\}$, where $\kappa = \bar \delta / 2 \pi \sinh (38\bar \delta)$.
\end{prop}

\begin{proof}
	Recall that for every $m \in \N$, we have $m\len[stable]{\bar g} \geq \len{\bar g^m} - 32\bar \delta$.
	Therefore it suffices to find an integer $m$ such that $\len{\bar g^m} \geq  m\min\left\{ \kappa l, \bar \delta\right\} +32\bar \delta$.
	We denote by $m$ the largest integer satisfying $m \min\left\{ \kappa l, \bar \delta\right\} \leq \bar \delta$.
	Assume that $\len{\bar g^m}$ is smaller than $m\min\left\{ \kappa l, \bar \delta\right\} +32\bar \delta$.
	In particular, $\len{\bar g^m} \leq 33\bar \delta$.
	In follows that for every $\bar v \in \bar v(\mathcal Q)$, the axis $A_{\bar g^m}$ of $\bar g^m$ does not intersect $B(\bar v, \rho - 17\bar \delta)$.
	Indeed if it was the case, $\bar g^m$ would fix $\bar v$ which contradicts the fact that $\bar g$ is not elliptic.
	By \autoref{res: SC - lifting quasi-convex}, there exists a subset $A$ of $\dot X$ such that the map $\zeta : \dot X \rightarrow \bar X$ induces an isometry from $A$ onto $A_{\bar g^m}$ and the projection $\pi : G \twoheadrightarrow \bar G$ induces an isomorphism from $\stab A$ onto $\stab {A_{\bar g^m}}$.
	We denote by $g$ the preimage of $\bar g$ in $\stab A$.
	By assumption $g$ is loxodromic, therefore $\len[stable]g \geq l$.
	Let $\bar x$ be a point of $A_{\bar g^m}$, $x$ the preimage of $\bar x$ in $A$ and $y$ a projection of $x$ on $X$.
	Recall that $\bar x$ lies in the $17\delta$-neighborhood of $\zeta(X)$, thus $\dist[\dot X] xy \leq 17\bar \delta$.
	Using the triangle inequality we get
	\begin{displaymath}
		\mu\left(\dist[X]{g^my}y\right) \leq \dist[\dot X]{g^my}y\leq \dist[\dot X]{g^mx}x + 34\bar \delta = \dist[\bar X]{\bar g^m \bar x}{\bar x} + 34 \bar \delta \leq \len{\bar g^m} + 42 \bar \delta \leq 75 \bar \delta.
	\end{displaymath}
	By \autoref{res: map mu},
	\begin{equation*}
		ml \leq m\len[stable] g\leq  \dist[X]{g^my}y \leq \pi \sinh (38 \bar \delta) \leq \frac {\bar \delta}{2\kappa},
	\end{equation*}
	which contradicts the maximality of $m$.
\end{proof}


\begin{coro}
\label{res: SC - lower bound injectivity radius}
	We denote by $l$ the greatest lower bound on the stable translation length (in $X$) of loxodromic elements of $N$ which does not belong to some $\stab Y$ for $(H,Y) \in \mathcal Q$.
	Then $\rinj{\bar N}{\bar X} \geq \min\left\{ \kappa l/8, \bar \delta\right\}$, where $\kappa = 2\rho/\pi\sinh\rho$
\end{coro}

%% file: 4_applications.tex

%
%

\section{Applications}
\label{sec: applications}

%
%

\subsection{Partial periodic quotients}
\label{sec: partial periodic groups}

The next proposition will play the role of the induction step in the proof of the main theorem.
\begin{prop}
\label{res: SC - induction lemma}
	There exist positive constants $\rho_0$,  $\delta_1$, $L_S$ such that for every integer $\nu_0$ there is an integer $n_0$ with the following properties.
	Let $G$ be a group acting by isometries on a $\delta_1$-hyperbolic length space $X$.
	We assume that this action is WPD and non-elementary.
	Let $N$ be a normal subgroup of $G$ without involutions.
	Let $n_1 \geq n_0$  and $n \geq n_1$ be an odd integer.
	We denote by $P$ the set of loxodromic elements $h$ of $N$ which are primitive as elements of $N$ such that $\len h \leq L_S\delta_1$.
	Let $K$ be the (normal) subgroup of $G$ generated by $\{h^n, h \in P\}$ and $\bar G$ the quotient of $G$ by $K$.
	We make the following assumptions.
	\begin{enumerate}
		\item \label{enu: SC - induction lemma - e}
		$e(N,X)$ divides $n$.
		\item \label{enu: SC - induction lemma - nu}
		$\nu(N,X) \leq \nu_0$.
		\item \label{enu: SC - induction lemma - A}
		$A(N,X) \leq (\nu_0+5)  \pi\sinh (2L_S\delta_1)$.
		\item \label{enu: SC - induction lemma - rinj}
		$\displaystyle \rinj NX \geq \delta_1 \sqrt {\frac {2L_S\sinh \rho_0}{n_1\sinh (38 \delta_1)}}$.
	\end{enumerate}
	Then there exists a $\delta_1$-hyperbolic length space $\bar X$ on which $\bar G$ acts by isometries.
	This action is WPD and non-elementary.
	The image $\bar N$ of $N$ in $\bar G$ has no involution.
	Moreover it satisfies  Assumptions~\ref{enu: SC - induction lemma - e}-\ref{enu: SC - induction lemma - rinj}.
	In addition, the map $G \rightarrow \bar G$ has the following properties.
	\begin{itemize}
		\item  For every $g \in G$, if $\bar g$ stands for its image in $\bar G$, we have 
		\begin{displaymath}
			\len[stable, espace= \bar X]{\bar g} \leq \frac 1{\sqrt {n_1}} \left(\frac {4\pi}{\delta_1}\sqrt{\frac {2\sinh \rho_0\sinh (38 \delta_1)}{L_S}}\right)\len[stable, espace= X]g. 
		\end{displaymath}
		\item For every non-loxodromic elementary subgroup $E$ of $G$, the map $G \rightarrow \bar G$ induces an isomorphism from $E$ onto its image $\bar E$ which is elementary and non-loxodromic.
		\item Let $\bar g$ be an elliptic (\resp parabolic) element of $\bar N$. 
		Either $\bar g^n = 1$ or $\bar g$ is the image of an elliptic (\resp parabolic) element of $N$.
		\item Let $u,u' \in N$ such that $\len u < \rho_0/100$ and $u'$ is elliptic. 
		If the respective images of $u$ and $u'$ are conjugated in $\bar G$ then so are $u$ and $u'$ in $G$.
	\end{itemize}

	\end{prop}

\voc Let $G$ be a group acting by isometries on a space $X$ and $N$ a normal subgroup of $G$.
Once $\nu_0$, $n_1$ and $n$ have been fixed, if $G$, $N$ and $X$ satisfy the assumption of the proposition including Points~\ref{enu: SC - induction lemma - e}-\ref{enu: SC - induction lemma - rinj}, we will write that $(G,N,X)$ \emph{satisfies the induction hypotheses for exponent $n$}.
The proposition says in particular that if $(G,N,X)$ satisfies the induction hypotheses for exponent $n$ then so does $(\bar G, \bar N,\bar X)$.

\begin{proof}
	The parameter $L_S$ is still the one that comes from the stability of quasi-geodesics (see \autoref{res: stability (1,l)-quasi-geodesic} and the remark after).
	The parameters $\rho_0$, $\delta_0$ and $\Delta_0$ are the one given by the small cancellation theorem (\autoref{res: SC - small cancellation theorem}).
	We set $\delta_1 = 64.10^4\boldsymbol \delta$.
	Let $\nu_0 \geq 0$.
	We now define the critical exponent $n_0$.
	To that end we consider a rescaling parameter $\lambda_n$ depending on an integer $n$
	\begin{displaymath}
		\lambda_n =\frac {4\pi}{\delta_1}\sqrt{\frac {2\sinh \rho_0\sinh (38 \delta_1)}{nL_S}}
	\end{displaymath}
	The sequence $(\lambda_n)$ converges to 0 as $n$ approaches infinity. 
	Therefore there exists an integer $n_0 \geq 100$ such that for every $n \geq n_0$
	\begin{eqnarray}
		\label{eqn: induction - delta}
		\lambda_n \delta_1 &  \leq & \delta_0 \\
		\label{eqn: induction - Delta}
		\lambda_n\left((\nu_0+5)  \pi\sinh (2L_S\delta_1)+  90\delta_1\right) & \leq & \min \left\{\Delta_0, \pi\sinh (2L_S\delta_1) \right\} \\
		\label{eqn: induction - rinj}
		\frac{\lambda L_S\delta_1^2}{4\pi \sinh (38\delta_1)} & < & \delta_1 \\
		\label{eqn: induction - non elem}
		\lambda_n \rho_0 & \leq & \rho_0
	\end{eqnarray}
	Let $n_1 \geq n_0$ and $n \geq n_1$ be an odd integer.
	For simplicity of notation we denote by $\lambda$ the rescaling parameter $\lambda = \lambda_{n_1}$.
	Let $G$ be a group acting by isometries on a metric space $X$ and $N$ a normal subgroup of $G$ such that $(G,N,X)$ satisfies the induction hypotheses for exponent $n$.
	We denote by $P$  the set of loxodromic elements $h$ of $N$ which are primitive as elements of $N$ such that $\len h \leq L_S\delta_1$.
	Let $K$ be the normal subgroup of $G$ generated by $\{h^n, h \in P\}$.
	Note that $P$ in invariant under conjugacy, thus $K$ is contained in $N$.
	We write $\bar G$ for the quotient of $G$ by $K$ and $\bar N=N/K$ for the image of $N$ in $\bar G$.
	We are going to prove that $\bar G$ is a small cancellation quotient of $G$.
	To that end we consider the action of $G$ on the rescaled space $\lambda X$.
	In particular it is a $\delta$-hyperbolic space, with $\delta = \lambda \delta_1 \leq \delta_0$.
	Unless stated otherwise, we will always work with the rescaled space $\lambda X$.
	We define the family $\mathcal Q$ by 
	\begin{displaymath}
		\mathcal Q = \set{\fantomB\left(\left< h^n\right>,Y_h\right)}{h \in P}.
	\end{displaymath}

	\begin{lemm}
	\label{res: SC - induction lemma - small cancellation}
		The family $\mathcal Q$ satisfies the following assumptions: $\Delta \left( \mathcal Q \right) \leq \Delta_0$ and $T\left(\mathcal Q\right) \geq 8\pi \sinh \rho_0$.	
	\end{lemm}
	
	\begin{proof}
		We start with the upper bound of $\Delta (\mathcal Q)$.
		Let $h_1$ and $h_2$ be two elements of $P$ such that $(\langle h_1^n\rangle,Y_{h_1}) \neq (\langle h_2^n\rangle,Y_{h_2})$.
		According to \autoref{res: Yg in quasi-convex g-invariant}, $Y_{h_1}$ and $Y_{h_2}$ are respectively contained in the $38\delta$-neighborhood of $A_{h_1}$ and $A_{h_2}$, thus by \autoref{res: intersection of thickened quasi-convex}
		\begin{displaymath}
			\diam\left( Y_{h_1}^{+ 5\delta} \cap Y_{h_2}^{+ 5\delta} \right) \leq \diam\left( A_{h_1}^{+17 \delta} \cap A_{h_2}^{+ 17\delta} \right) + 90\delta.
		\end{displaymath}
		According to \autoref{res: non elementary subgroup generated by two elements}, $h_1$ and $h_2$ generate a non-elementary subgroup of $N$.
		On the other hand, their translation lengths in $\lambda X$ are at most $L_S\delta$, thus
		\begin{eqnarray*}
			\diam\left( Y_{h_1}^{+ 5\delta} \cap Y_{h_2}^{+ 5\delta} \right) \leq A(N, \lambda X) + 90\delta 
			& \leq & \lambda A(N,X) + 90\lambda\delta_1 \\
			& \leq & \lambda((\nu_0+5)  \pi\sinh (2L_S\delta_1)+  90\delta_1).
		\end{eqnarray*}
		Thus by (\ref{eqn: induction - Delta}), $\Delta(\mathcal Q) \leq \Delta_0$.
		Let us focus now on $T(\mathcal Q)$.
		The injectivity radius of $N$ on $\lambda X$ is bounded below as follows
		\begin{displaymath}
			\rinj N{\lambda X} \geq \lambda \delta_1 \sqrt {\frac {2L_S\sinh \rho_0}{n_1\sinh (38 \delta_1)}} =  \frac{8\pi \sinh \rho_0}{n_1} \geq   \frac{8\pi \sinh \rho_0}n 
		\end{displaymath}
		In particular for every $h \in P$ we have $\len[stable]{h^n} = n \len[stable] h \geq 8\pi \sinh \rho_0$. 
		Hence $T(\mathcal Q) \geq 8\pi \sinh \rho_0$.
	\end{proof}
	
	\paragraph{}On account of the previous lemma, we can now apply the small cancellation theorem (\autoref{res: SC - small cancellation theorem}) to the action of $G$ on the rescaled space $\lambda  X$ and the family $\mathcal Q$.
	We denote by $\dot X$ the space obtained by attaching on $\lambda X$ for every $(H,Y) \in \mathcal Q$, a cone of radius $\rho_0$ over the set $Y$.
	The quotient of $\dot X$ by $K$ is the space $\bar X$.
	According to \autoref{res: SC - small cancellation theorem}, $\bar X$ is a $\delta_1$-hyperbolic length space and $\bar G$ acts by isometries on it.
	By \autoref{res: WPD action on bar X} and \autoref{res: bar G non elementary} this action is WPD and non-elementary.
	It follows from \autoref{res: SC - bar G no involution} that $\bar N$ has no involution
	We now prove that the action of $\bar N$ on $\bar X$ also satisfies Assumptions~\ref{enu: SC - induction lemma - e}-\ref{enu: SC - induction lemma - rinj}.
	
	\begin{lemm}
	\label{res: SC - induction - e and nu}
		The invariant $e(\bar N,\bar X)$ and $\nu(\bar N, \bar X)$ satisfies the following
		\begin{itemize}
			\item $e(\bar N,\bar X)$ divides $n$
			\item $\nu(\bar N, \bar X) \leq \nu_0$
		\end{itemize}
	\end{lemm}
	
	\begin{proof}
		By \autoref{res: SC - invariant exponent}, $e(\bar N, \bar X)$ divides $e(N,X)$.
		Thus the first point follows from Assumption~\ref{enu: SC - induction lemma - e} of the proposition.
		The second one is a consequence of \autoref{res: SC - contol nu} and Assumption~\ref{enu: SC - induction lemma - nu}
	\end{proof}

	\begin{lemm}
	\label{res: SC - induction - estimating parameters}
		The constant $A(\bar N, \bar X)$ is bounded above by $(\nu_0+5)  \pi\sinh (2L_S\delta_1)$ whereas $\rinj {\bar N}{\bar X}$ is bounded below as follows
		\begin{equation*}
			\rinj {\bar N}{\bar X} \geq \delta_1 \sqrt {\frac {2L_S\sinh \rho_0}{n_1\sinh (38 \delta_1)}}
		\end{equation*}
	\end{lemm}
	
	\begin{proof}
		We start with the upper bound of $A(\bar N, \bar X)$.
		According to \autoref{res: SC - control invariant A}, 
		\begin{eqnarray*}
			A(\bar N, \bar X) 
			& \leq & A(N, \lambda  X) + (\nu(N,X) + 4) \pi\sinh\left (2L_S\delta_1\right)   \\
			& \leq & A(N, \lambda  X) +(\nu_0+ 4) \pi\sinh\left (2L_S\delta_1\right).
		\end{eqnarray*}
		However the inequality~(\ref{eqn: induction - Delta}) gives
		\begin{displaymath}
			 A(N, \lambda  X) =  \lambda A(N,  X) \leq \lambda(\nu_0+5)  \pi\sinh\left (2L_S\delta_1\right) \leq \pi\sinh\left (2L_S\delta_1\right).
		\end{displaymath}
		Thus $A(\bar N, \bar X)$ is bounded above by $(\nu_0+5)  \pi\sinh\left (2L_S\delta_1\right)$.
		We now focus on the injectivity radius of $\bar N$.
		Let $g$ be a loxodromic isometry of $N$ which does not belong to the stabilizer of $Y_h$ where $h \in P$.
		Its asymptotic translation length in $\lambda  X$ is larger than $\lambda L_S\delta_1/2$.
		\autoref{res: SC - lower bound injectivity radius} combined with (\ref{eqn: induction - rinj}) gives
		\begin{displaymath}
			\rinj {\bar N}{\bar X} 
			\geq \min \left\{ \frac{\lambda L_S\delta_1^2}{4\pi \sinh (38\delta_1)},  \delta_1 \right\} 
			=  \frac{\lambda L_S\delta_1^2}{4\pi \sinh (38\delta_1)} 
			= \delta_1 \sqrt {\frac {2L_S\sinh \rho_0}{n_1\sinh (38 \delta_1)}}. \qedhere
		\end{displaymath}
	\end{proof}

	\paragraph{}Lemmas~\ref{res: SC - induction - e and nu} and  \ref{res: SC - induction - estimating parameters} show that $(\bar G,\bar N, \bar X)$ satisfies the induction hypotheses for exponent $n$.
	To finish the proof we focus on the properties on the map $G \rightarrow \bar G$.
	
	\begin{lemm}
		For every $g \in G$, we have 
		\begin{displaymath}
			\len[stable, espace= \bar X]{\bar g} \leq \frac 1{\sqrt {n_1}} \left(\frac {4\pi}{\delta_1}\sqrt{\frac {2\sinh \rho_0\sinh (38 \delta_1)}{L_S}}\right)\len[stable, espace= X]g.
		\end{displaymath}
	\end{lemm}

	\begin{proof}
		Let $g \in G$.
		The asymptotic translation length of $g$ in the rescaled space $\lambda X$ is $\len[stable, espace= \lambda X]g = \lambda \len[stable, espace= X]g$.
		On the other hand the map $\lambda  X \rightarrow \bar X$ shortens the distances, thus $\len[stable, espace= \bar X]{\bar g}\leq  \lambda \len[stable, espace= X]g$.
	\end{proof}
	
	\begin{lemm}
		Let $E$ be a non-loxodromic elementary subgroup of $G$.
		The map $G \rightarrow \bar G$ induces an isomorphism from $E$ onto its image $\bar E$ which is elementary and non-loxodromic.
	\end{lemm}
	
	\begin{proof}
		This lemma follows from \autoref{res: SC - image of elementary subgroup} and \autoref{res: SC - proj one-to-one on non-hyp elem sg}.
	\end{proof}
	
	\begin{lemm}
		Let $\bar g$ be an elliptic (\resp parabolic) element of $\bar N$. Either $\bar g^n = 1$ or $\bar g$ is the image of an elliptic (\resp parabolic) element of $N$.
	\end{lemm}

	\begin{proof}
		If $\bar g$ is parabolic, it follows from \autoref{res: lifting parabolic subgroup - isomorphism}.
		Assume now that $\bar g$ is elliptic.
		We denote by $\bar E$ the subgroup of $\bar G$ generated by $\bar g$. 
		According to \autoref{res: SC  - lifting elliptic subgroups}, there are two cases. 
		\begin{enumerate}
			\item In the first case, there exists $h \in P$ such that $\bar E$ is embedded in $\stab {Y_h}/ \langle h^n \rangle$.
			However $e(N,X)$ divides $n$.
			Therefore the order of any element of $\bar N$ in this group divides $n$ (see \autoref{res: invariant e}). 
			\item In the second case $\bar E$ is isomorphic to an elliptic subgroup $E$ of $G$.
			Hence $\bar g$ has an elliptic preimage in $G$. \qedhere
		\end{enumerate}
	\end{proof}	
	
	\begin{lemm}
		Let $u,u' \in N$ such that $\len u < \rho_0/100$ and $u'$ is elliptic. 
		If the respective images of $u$ and $u'$ are conjugated (in $\bar G$) so are $u$ and $u'$ in $G$.	
	\end{lemm}
	
	\begin{proof}
		This lemma follows directly from \autoref{res: elliptic conjugated in the quotient}.
	\end{proof}
	
	These last lemmas complete the proof of \autoref{res: SC - induction lemma}.
\end{proof}

\begin{theo}
\label{res : SC - partial periodic quotient}
	Let $X$ be a hyperbolic length space.
	Let $G$ be a group acting by isometries on $X$.
	We suppose that this action is WPD and non-elementary.
	Let $N$ be a normal subgroup of $G$ without involution.
	In addition we assume that $e(N,X)$ is odd, $\nu(N,X)$ and $A(N,X)$ are finite and $\rinj NX$ is positive.
	There is a critical exponent $n_1$ such that every odd integer $n \geq n_1$ which is a multiple of $e(N,X)$ has the following property.
	There exists a normal subgroup $K$ of $G$ contained in $N$ such that
	\begin{itemize}
		\item if $E$ is an elementary subgroup of $G$ which is not loxodromic, then the projection $G \twoheadrightarrow G/K$ induces an isomorphism from $E$ onto its image;
		\item every non-trivial element of $K$ is loxodromic;
		\item for every element $g \in N/K$, either $g^n=1$ or $g$ is the image a non-loxodromic element of $G$;
		\item there are infinitely many elements in $N/K$ which do not belong to the image of an elementary non-loxodromic subgroup of $G$;
		\item As a normal subgroup, $K$ is not finitely generated.
	\end{itemize}
\end{theo}

\rem For most of our examples we will simply take $N = G$.
However this more general statement is useful to avoid some problems coming from the $2$-torsion.

\begin{proof}
	The main ideas of the proof are the followings.
	Using \autoref{res: SC - induction lemma} we construct by induction a sequence of groups $G_0 \rightarrow G_1\rightarrow G_2 \rightarrow \dots$ where $G_{k+1}$ is obtained from $G_k$ by adding new relations of the form $h^n$ with $h \in N$.
	Then we chose for the quotient $G/K$ the direct limit of these groups.
	Let us put $\nu_0 = \nu(N,X)$ (which is finite by assumption).
	The parameters $\rho_0$, $L_S$, $\delta_1$ and $n_0$ are the one given by \autoref{res: SC - induction lemma}.
		
	\paragraph{Critical exponent.}
	The invariant $A(N,X)$ is finite. 
	By rescaling if necessary the space $X$ we can assume the followings
	\begin{itemize}
		\item $\delta \leq \delta_1$,
		\item $A(N,X) \leq (\nu_0+5)  \pi\sinh (2L_S\delta_1)$
	\end{itemize}
	By assumption $\rinj NX >0$.
	Therefore, there exists an integer $n_1 \geq n_0$ such that
	\begin{equation*}
		\rinj NX \geq  \delta_1 \sqrt {\frac {2L_S\sinh \rho_0}{n_1\sinh (38 \delta_1)}}.
	\end{equation*}
	Without loss of generality we can also assume that the constant $\lambda$ defined below is less than $1$.
	\begin{displaymath}
		\lambda =  \frac 1{\sqrt {n_1}} \left(\frac {4\pi}{\delta_1}\sqrt{\frac {2\sinh \rho_0\sinh (38 \delta_1)}{L_S}}\right).
	\end{displaymath}
	From now on we fix an odd integer $n \geq n_1$ which is a multiple of $e(N,X)$.
	
	\paragraph{Initialization.} We put $G_0 = G$, $N_0=N$ and $X_0 = X$.
	In particular $(G_0,N_0,X_0)$ satisfies the induction hypotheses for exponent $n$.
	
	\paragraph{Induction.} We assume that we already constructed the groups $G_k$, $N_k$ and the space $X_k$ such that $(G_k,N_k,X_k)$ satisfies the induction hypotheses for exponent $n$.
	We denote by $P_k$ the set of loxodromic elements $h \in N_k$ such that $\len[espace = X_k] h \leq L_S\delta_1$ which are primitive \emph{as elements of $N_k$}.
	Let $K_k$ be the normal subgroup of $G_k$ generated by $\{h^n, h \in P_k\}$.
	We write $G_{k+1}$ for the quotient of $G_k$ by $K_k$ and $N_{k+1}$ for the image of $N_k$ in $G_{k+1}$.
	In particular $N_{k+1}$ is a normal subgroup of $G_{k+1}$.
	By \autoref{res: SC - induction lemma}, there exists a metric space $X_{k+1}$ such that $(G_{k+1},N_{k+1},X_{k+1})$ satisfies the induction hypotheses for exponent $n$.
	Moreover the projection $G_k \twoheadrightarrow G_{k+1}$ fulfills the following properties.
	\begin{enumerate}
		\item \label{enu: partial periodic quotient - translation length}
		For every $g \in G_k$, if we still denote by $g$ its image in $G_{k+1}$ we have $\len[stable, espace = X_{k+1}] g \leq \lambda \len[stable, espace = X_k] g$.
		\item \label{enu: partial periodic quotient - elementary subgroup}
		For every non-loxodromic elementary subgroup $E$ of $G_k$, the map $G_k \twoheadrightarrow G_{k+1}$ induces an isomorphism from $E$ onto its image which is elementary and non-loxodromic.
		\item \label{enu: partial periodic quotient - elliptic element}
		For every elliptic or parabolic element $g \in N_{k+1}$, either $g^n=1$ or $g$ is the image of an elliptic or parabolic element of $N_k$.
		\item \label{enu: partial periodic quotient - conjugate of elliptic}
		Let $u,u' \in N_k$ such that $\len[espace=X_k]u< \rho_0/100$ and $u'$ is elliptic. If the respective images of $u$ and $u'$ are conjugated in $G_{k+1}$ so are $u$ and $u'$ in $G_k$.
	\end{enumerate}

	\paragraph{Direct limit.} 
	The direct limit of the sequence $(G_k)$ is a quotient $G/K$ of $G$.
	We claim that this group satisfies the announced properties.
	Let $g$ be an element of $G$.
	To shorten the notation we will still denote by $g$ its images in $G$, $G_k$ or $G/K$.
	
	\paragraph{Properties of $G/K$.} 
	Let $E$ be an elementary subgroup of $G$ which is not loxodromic.
	A proof by induction on $k$ shows that for every $k \in \N$, the map $G \twoheadrightarrow G_k$ induces an isomorphism from $E$ onto its image which is an elementary subgroup of $G_k$ either elliptic or parabolic.
	It follows that $G \twoheadrightarrow G/K$ induces an isomorphism from $E$ onto its image.
	This proves the first point of the theorem.
	
	\paragraph{}
	Let $g$ be a non-trivial element of $K$.
	Assume that contrary to our claim $g$ is not loxodromic.
	Then $\langle g\rangle$ is an elementary subgroup of $G$ either elliptic or parabolic.
	Therefore the map $G \twoheadrightarrow G/K$ induces an isomorphism from $\langle g \rangle$ onto its image.
	In particular $g$ is not trivial in $G/K$, and thus cannot belong to $K$. 
	Contradiction.
	
	\paragraph{}
	A proof by induction on $k$ shows that if $g$ is a non-loxodromic element of $N_k$ then either $g^n=1$ or $g$ is the image of a non-loxodromic element of $N$.
	Let $g$ be an element of $N/K$ which is not the image of a non-loxodromic element of $N$.
	We still denote by $g$ a preimage of $g$ in $N$.
	In particular $g$ is loxodromic.
	It follows from the construction of the sequence $(G_k)$ that for every $k \in \N$, we have $\len[stable, espace= X_k] g \leq \lambda^k\len[stable, espace=X] g$.
	Recall that $\lambda <1$.
	There exists an integer $k$ such that $\len[stable, espace= X_k] g < \rinj {N_k}{X_k}$.
	As an element of $G_k$ the isometry $g$ is not loxodromic.
	Consequently, as an element of $N_k$, $g^n = 1$.
	The same holds in $G/K$.
	
	\paragraph{}
	We now focus on the last point.
	Denote by $P$ for the set of all loxodromic elements of $N$ which are not identified in $G/K$ with a non-loxodromic element of $N$.
	Assume that the image of $P$ in $N/K$ is finite.
	In particular there exists a finite subset $S$ of $P$ such that $P$ lies in $S\cdot K$.
	Using a similar argument as previously we see that there exists $s \in \N$ such that every element of $S$ is non-loxodromic in $N_s$ and $P_s$ is not empty.
	Fix $g \in P$ a preimage in $N$ of an element of $P_s$.
	By construction $g$ is loxodromic in $N_s$ with $\len[espace = X_s] g \leq L_S\delta_1 \leq \rho_0/100$ and elliptic in $N_{s+1}$.
	However $P$ is a subset of $S\cdot K$.
	Therefore there exits $t > s$ such that $g$ belongs to $S$ as an element of $N_t$.
	An induction using the Property~\ref{enu: partial periodic quotient - conjugate of elliptic} about conjugates shows that $g$ is actually conjugated to an element of $S$ in $N_s$.
	However in $N_s$, $g$ is loxodromic whereas all elements of $S$ are elliptic.
	Contradiction.
	
	\paragraph{}For every $k \in \N$, the action of $G_k$ on $X_k$ is non-elementary. 
	It follows that the sequence $(G_k)$ does not ultimately stabilize.
	Thus $K$ is infinitely generated as a normal subgroup.
\end{proof}

%
%

\subsection{Acylindrical action on a hyperbolic space}
\label{sec: acylindrical action}

	\paragraph{}Our main source of examples comes from groups acting acylindrically on a hyperbolic space.
	We recall and prove here a few properties of this action.
	They will be useful to satisfy the assumptions of \autoref{res : SC - partial periodic quotient}.
	In this section, $X$ is a $\delta$-hyperbolic length space endowed with an action by isometries of a group $G$.
	
	\begin{defi}
	\label{def: acylindrical action}
		The action of $G$ on $X$ is \emph{acylindrical} if for every $l\geq0$ there exist $d\geq0$ and $N >0$ such that for all $x,x' \in X$ with $\dist x{x'} \geq d$ there are at most $N$ isometries $u \in G$ satisfying $\dist {ux}x \leq l$ and $\dist {ux'}{x'} \leq l$.
	\end{defi}
	
	\paragraph{}
	Note that if a group $G$ acts acylindrically on a hyperbolic space, this action is also WPD (see \autoref{def: WDP}).
	However the acylindricity condition is much stronger. 
	In particular the parameters $d$ and $N$ are uniform.
	They only depend on $l$ and not on the points $x$ and $x'$.
 	A proper and co-compact action on a hyperbolic space is acylindrical.
	An other example is the action of the mapping class group of a surface on its complex of curves.
	More examples are detailed in \autoref{sec: examples}.
	From now on, we will assume that $G$ acts acylindrically on $X$.

	\begin{lemm}[Bowditch]{\rm \cite[Lemma 2.2]{Bowditch:2008bj}} \quad
	\label{res: acylindrical action gives positive injectivity radius}
		The injectivity radius $\rinj GX$ is positive.
	\end{lemm}
	
	\begin{lemm}
	\label{res: acylindrical action gives finite nu}
		The invariant $\nu(G,X)$ is finite.
	\end{lemm}

	\begin{proof}
		By acylindricity, there exist positive constants $d$ and $N$ with the following property.
		For every $x,y \in X$ with $\dist xy \geq d$ there are at most $N$ elements $u \in G$ satisfying $\dist {ux}x \leq 166\delta$ and $\dist {uy}y \leq 166\delta$.
		According to \autoref{res: acylindrical action gives positive injectivity radius} the injectivity radius of $G$ on $X$ is positive.
		We can fix $M$ such that  $M \rinj GX \geq d$.
		Let $m$ be an integer such that $m \geq N + M$.
		Let $g,h \in G$ with $h$ loxodromic.
		Assume that $g, h^{-1}gh,\dots,h^{-m}gh^m$ generate an elementary subgroup of $G$ which is not loxodromic.
		According to \autoref{res: characteristic subset summary} and \autoref{res: partial characteristic subset for parabolic} there exists a point $x \in X$ such that for every $j \in \intvald 0m$, $\dist{h^{-j}gh^jx}x \leq 166\delta$.
		In particular for every $j \in \intvald 0N$ we have
		\begin{equation*}
			\dist{h^{-j}gh^jx}x \leq 166\delta \quad \text{and}\quad \dist{h^{-j}gh^j(h^Mx)}{h^Mx} \leq 166\delta.
		\end{equation*}
		However by choice of $M$, $\dist{h^Mx}x \geq d$.
		It follows then from acylindricity that the set
		\begin{equation*}
			\set{h^{-j}gh^j}{j \in \intvald 0N}
		\end{equation*}
		contains at most $N$ elements.
		Therefore there exists $j \in \intvald 1N$ such that $h^{-j}gh^j = g$.
		Hence $g$ stabilizes $\{ h^-,h^+\}$ where $h^-$ and $h^+$ are the points of the boundary $\partial X$ fixed by $h$.
		In particular, $g$ and $h$ generate an elementary subgroup of $G$. 
		Consequently, $\nu(G,X)$ is bounded above by $N+M$.
	\end{proof}
	
	We now focus on the invariant $A(G,X)$.
	Recall first that given $m$ elements $g_1, \dots, g_m$ of $G$ the quantity $A(g_1, \dots, g_m)$ is defined by 
	\begin{equation*}
		A(g_1, \dots,g_m) = \diam \left( A_{g_1}^{+13 \delta} \cap \dots \cap  A_{g_m}^{+13 \delta}\right).
	\end{equation*}

	\begin{lemm}
	\label{res: acylindrical action gives finite A - prelim}
		Let $m \in \N$.
		There exist  $\ell \in \N$ and $A>0$ with the following property.
		Let $g_1, \dots ,g_m$ be $m$ elements of $G$ which generate a non-elementary subgroup.
		If $A(g_1, \dots ,g_m) > A$ then there exists a loxodromic element which is the product of at most $\ell$ elements of $\{g_1, \dots ,g_m\}$ or their inverses.		
	\end{lemm}
	
	\begin{proof}
		Since $G$ acts acylindrically on $X$ there exists $N \in \N$ and $d>0$ with the following property.
		For every $x, x' \in X$, if $\dist x{x'} \geq d$ then there are at most $N$ elements $u \in G$ such that $\dist{ux}x \leq 66\delta$ and $\dist{ux'}{x'} \leq 66 \delta$.
		We now put $\ell = N+1$ and $A = d + (66\ell+10)\delta$.
		Let $g_1, \dots, g_m$ be $m$ elements of $G$ which generate a non-elementary subgroup $H$ such that $A(g_1, \dots ,g_m) > A$.
		We denote by $S$ the set of elements of $G$ that can be written as a product of at most $\ell$ elements of $\{g_1, \dots ,g_m\}$ or their inverses.
		Assume that, contrary to our claim, no element of $S$ is loxodromic.
		In particular, $g_1, \dots, g_m$ are not loxodromic, thus their translation length it at most $32 \delta$  (see \autoref{res: translation lengths}).
		
		\paragraph{}Let $h \in S$. 
		Let $x$ be a point in the intersection of the respective $13\delta$-neighborhoods of the axis $A_{g_1},\dots, A_{g_m}$.
		For every $j \in \intvald 1m$, $\dist{g_jx}x \leq 66\delta$.
		It follows from the triangle inequality that $\dist{hx}x \leq 66\ell \delta$. 
		According to \autoref{res: axes is quasi-convex}~\ref{enu: axes - deplacement donne distance a l'axe}, $x$ lies in the $(33\ell + 3)\delta$-neighborhood of $A_h$.
		It follows that 
		\begin{equation*}
			\diam\left( \bigcap_{h \in S}A_h^{+(33\ell + 3)\delta}\right) \geq A(g_1, \dots ,g_m) >A.
		\end{equation*}
		Applying \autoref{res: intersection of thickened quasi-convex}, we get that
		\begin{equation*}
			\diam\left( \bigcap_{h \in S}A_h^{+13\delta}\right) > A - (66\ell + 10)\delta \geq d.
		\end{equation*}
		In particular, there exist two points $x, x' \in X$ with $\dist x{x'} \geq d$ such that for every $h \in S$, $x$ and $x'$ belong to the $13\delta$-neighborhood of $A_h$.
		By assumption the elements of $S$ are not loxodromic, thus for every $h \in S$, $\dist{hx}x \leq 66\delta$ and $\dist{hx'}{x'} \leq 66 \delta$.
		By choice of $N$ and $d$, the set $S$ contains at most $N$ elements.
		However $\ell = N+1$.
		It follows that every element of $S$ which is exactly the product of $\ell$ elements of $\{g_1, \dots ,g_m\}$ or their inverses can be written as a shorter product.
		In particular, any element of the subgroup $H$ generated by $\{g_1, \dots ,g_m\}$ can be written as a product of at most $N$ elements of $\{g_1, \dots ,g_m\}$ or their inverses.
		Thus $H$ is finite.
		It contradicts the fact that $H$ is non-elementary.
	\end{proof}

	\begin{lemm}
	\label{res: acylindrical action gives finite A}
		The invariant $A(G,X)$ is finite.
	\end{lemm}
	
	\begin{proof}
		We need first to define many parameters.
		For simplicity of notation we put $\nu = \nu(G,X)$ which is finite according to \autoref{res: acylindrical action gives finite nu}.
		As in \autoref{sec:group invariants}, we denote by $\mathcal A$ the set of $(\nu+1)$-uples $(g_0,\dots,g_\nu)$ such that $g_0, \dots, g_\nu$ generate a non-elementary subgroup of $G$ and for all $j \in \intvald 0\nu$, $\len{g_j} \leq L_S\delta$. 
		According to \autoref{res: acylindrical action gives finite A - prelim}, there exist  $\ell \in \N$ and $A>0$ with the following property.
		For every $(g_0, \dots, g_\nu) \in \mathcal A$, if $A(g_0, \dots ,g_\nu) > A$ then there exists a loxodromic element which is the product of at most $\ell$ elements of $\{g_0, \dots ,g_\nu\}$ or their inverses.
		By \autoref{res: acylindrical action gives positive injectivity radius}, $\rinj GX$ is positive, thus there is an integer $m$ such that $m \rinj GX > L_S\delta$.
		Finally, by acylindricity, there exist $N \in \N$ and $d>0$ such that for every $x, y \in X$, if $\dist xy \geq d$ then there are at most $N$ elements $u \in G$ satisfying $\dist{ux}x \leq  (L_S+74) \delta$ and $\dist{uy}y \leq  (L_S+74) \delta$.
		We claim that
		\begin{equation*}
			A(G,X) \leq  \max\{A, d +(N+1) m\ell (L_S+34)\delta + (N+54)\delta\}.
		\end{equation*}
		Assume that our assertion is false.
		There exists $(g_0, \dots, g_\nu) \in \mathcal A$ such that 
		\begin{equation*}
			A(g_0, \dots ,g_\nu) > \max\{A, d +(N+1) m\ell (L_S+34)\delta + (N+54)\delta\}.
		\end{equation*}
		In particular, $A(g_0, \dots ,g_\nu) > A$.
		By choice of $A$ and $\ell$ there exists a loxodromic element which is the product of at most $\ell$ elements of $\{g_1, \dots ,g_\nu\}$ or their inverses.
		Taking its $m$-th power we obtain an element $h \in G$ with the following properties.
		\begin{enumerate}
			\item $h$ is the product of at most $m\ell$ elements of $\{g_1, \dots ,g_\nu\}$ or their inverses.
			\item $\len h \geq m \rinj GX> L_S\delta$.
		\end{enumerate}
		Let $\gamma \colon \R \rightarrow X$ be a $\delta$-nerve of $h$ and $T$ its fundamental length.
		Let $x$ be a point in the intersection of the respective $13\delta$-neighborhoods of the axis $A_{g_0}, \dots, A_{g_\nu}$.
		By definition for every $j \in \intvald 0\nu$, $\dist{g_jx}x \leq (L_S+34)\delta$.
		It follows from the triangle inequality that $\dist{hx}x \leq m\ell (L_S+34)\delta$.
		Hence 
		\begin{equation*}
			T \leq \len h + \delta  \leq m\ell (L_S+34)\delta +\delta.
		\end{equation*}
		Moreover, according to \autoref{res: axes is quasi-convex}~\ref{enu: axes - deplacement donne distance a l'axe}, the distance between $x$ and $A_h$ is at most $m\ell (L_S/2+17)\delta +3\delta$.
		Since $\len h > L_S \delta$, the axis $A_h$ lies in the $10\delta$-neighborhood of $\gamma$.
		Thus $x$ belongs to the $D$-neighborhood of $\gamma$ where $D = m\ell (L_S/2+17)\delta +13\delta$.
		In particular
		\begin{equation*}
			\diam \left(\gamma^{+D} \cap A_{g_0}^{13\delta} \cap \dots \cap  A_{g_\nu}^{13\delta} \right) \geq A(g_0, \dots ,g_\nu).
		\end{equation*}
		By \autoref{res: intersection of thickened quasi-convex}, we get that for every $j \in \intvald 0\nu$,
		\begin{equation*}
			\diam \left(\gamma^{+12\delta} \cap A_{g_j}^{+13\delta} \right) 
			\geq A(g_0, \dots ,g_\nu) - 2D -4\delta
			> d + Nm\ell (L_S+34)\delta +(N+24)\delta
		\end{equation*}
		Let $j \in \intvald 0\nu$.
		According to the previous inequality there exists points $x = \gamma(s)$ and $x' = \gamma(s')$ in the $25 \delta$-neighborhood of the axis of $g_j$ such that 
		\begin{equation}
		\label{eqn: acylindrical action gives finite A}
			\dist x{x'} 
			\geq d + Nm\ell (L_S+34)\delta +N\delta
			\geq d + NT.
		\end{equation}
		By replacing if necessary $h$ by $h^{-1}$ we can assume that $s \leq s'$.
		By stability of quasi-geodesics, for all $t \in \intval s{s'}$, $\gro x{x'}{\gamma(t)} \leq 6\delta$.
		Since the $25\delta$-neighborhood of $A_{g_j}$ is $2\delta$-quasi-convex (see \autoref{res: neighborhood quasi-convex}), it follows that $\gamma(t)$ lies in the $33\delta$-neighborhood of $A_{g_j}$.
		Thus $\dist{g_j\gamma(t)}{\gamma(t)} \leq (L_S+ 74)\delta$.
		According to (\ref{eqn: acylindrical action gives finite A}), there exists $t \in \intval s{s'}$ such that $\dist{\gamma(t)}x = d$.
		We put $y = \gamma(t)$.
		Note that 
		\begin{equation*}
			\dist{s'}t \geq \dist y{x'} \geq \dist x{x'} - \dist xy \geq NT.
		\end{equation*}
		Let $k \in \intvald 0N$.
		By construction $h^kx = \gamma(s+kT)$ and $h^ky = \gamma(t+kT)$.
		Using our previous remark, we see that $s+kT$ and $t+kT$ belongs to $\intval s{s'}$.
		Thus
		\begin{equation*}
			\max\left\{\dist {g_jh^kx}{h^kx}, \dist {g_jh^ky}{h^ky} \right\}\leq (L_S+74) \delta.
		\end{equation*}
		In other words, for every $k \in \intvald 0N$, $\dist{h^{-k}g_jh^kx}x \leq  (L_S+74) \delta$ and $\dist{h^{-k}g_jh^ky}y \leq  (L_S+74) \delta$.
		However $\dist xy \geq d$.
		By choice of $d$ and $N$, there exists $k \in \intvald 1N$ such that $g_j$ and $h^k$ commutes.
		Since $h$ is loxodromic, $g_j$ fixes pointwise $\{h^-,h^+\} \subset \partial X$.
		Hence $g_j$ belongs to the maximal elementary subgroup containing $h$.
		This statement holds for every $j \in \intvald 0\nu$.
		Consequently $g_0, \dots, g_\nu$ do not generate a non-elementary subgroup.
		Contradiction.
	\end{proof}
	
	\paragraph{}
	In view of \autoref{res: acylindrical action gives positive injectivity radius}, \autoref{res: acylindrical action gives finite nu} and \autoref{res: acylindrical action gives finite A}, \autoref{res : SC - partial periodic quotient} leads to the following result.
	\begin{theo}
	\label{res: SC - partial periodic quotient - acylindrical case}
		Let $X$ be a hyperbolic length space.
		Let $G$ be a group acting by isometries on $X$.
		We assume that the action of $G$ is acylindrical and non-elementary.
		Let $N$ be a normal subgroup of $G$ without involution.
		Assume that $e(N,X)$ is odd.
		There exists a critical exponent $n_1$ such that every odd integer $n \geq n_1$ which is a multiple of $e(N,X)$ has the following property.
		There exists a normal subgroup $K$ of $G$ contained in $N$ such that
		\begin{itemize}
			\item if $E$ is an elementary subgroup of $G$ which is not loxodromic, then the projection $G \twoheadrightarrow G/K$ induces an isomorphism from $E$ onto its image;
			\item for every element $g \in N/K$, either $g$ is the image a non-loxodromic element of $N$ or $g^n=1$;
			\item every non-trivial element of $K$ is loxodromic;
			\item there are infinitely many elements in $N/K$ which do not belong to the image of an elementary non-loxodromic subgroup of $G$.
			\item As a normal subgroup, $K$ is not finitely generated.
		\end{itemize}
	\end{theo}

%
%

\subsection{Examples}
\label{sec: examples}

\paragraph{Mapping class groups.}
Let $\Sigma$ be a compact surface of genus $g$ with $p$ boundary components.
In the rest of this paragraph we assume that its complexity $3g+p -3$ is larger than $1$.
The \emph{mapping class group} $\mcg \Sigma$ of $\Sigma$ is the group of orientation preserving self homeomorphisms of $\Sigma$ defined up to homotopy.
A mapping class $f \in \mcg \Sigma$ is
\begin{enumerate}
	\item \emph{periodic}, if it has finite order;
	\item \emph{reducible}, if it permutes a collection of essential non-peripheral curves (up to isotopy);
	\item \emph{pseudo-Anosov}, if there exists an homotopy in the class of $f$ that preserves a pair of transverse foliations and rescale these foliations in an appropriate way.
\end{enumerate}
It follows from Thurston's work that any element of $\mcg \Sigma$ falls into one these three categories \cite[Theorem 4]{Thu88}.
The \emph{complex of curves} $X$ is a simplicial complex associated to $\Sigma$.
It has been first introduced by W.~Harvey \cite{Harvey:1981tg}.
A $k$-simplex of $X$ is a collection of $k+1$ curves of $\Sigma$ that can be disjointly realized.
In \cite{Masur:1999hc}, H.~Masur and Y.~Minsky proved that this new space is hyperbolic.
By construction, $X$ is endowed with an action by isometries of $\mcg \Sigma$.
Moreover B.~Bowditch showed that this action is acylindrical \cite[Theorem 1.3]{Bowditch:2008bj}.
This is an example of a group acting acylindrically but not properly on a hyperbolic space.
Indeed the stabilizer of a point, i.e. the set of mapping classes preserving a curve, is far from being finite.
This action provides an other characterization of the elements of $\mcg \Sigma$.
An element of $\mcg \Sigma$ is periodic or reducible (\resp pseudo-Anosov) if and only it is elliptic (\resp loxodromic) for the action on the complex of curves \cite{Masur:1999hc}.

\begin{theo}
\label{res: partial periodic quotient mcg}	
Let $\Sigma$ be a compact surface of genus $g$ with $p$ boundary components such that $3g +p - 3 >1$.
	There exist integers $\kappa$ and $n_0$ such that for every odd exponent $n \geq n_0$ there is a quotient $G$ of $\mcg \Sigma$ with the following properties.
	\begin{enumerate}
		\item If $E$ is a subgroup of $\mcg \Sigma$ that does not contain a pseudo-Anosov element, then the projection $\mcg \Sigma \twoheadrightarrow G$ induces an isomorphism from $E$ onto its image.
		\item Let $f$ be a pseudo-Anosov element of $\mcg \Sigma$.
		Either $f^{\kappa n} = 1$ in $G$ or there exists a periodic or reducible element $u \in \mcg \Sigma$ such that $f^\kappa = u$ in $G$.
		In particular, for every pseudo-Anosov $f \in \mcg \Sigma$, there exists a non-pseudo-Anosov element $u \in \mcg \Sigma$ such that $f^{\kappa n} = u$ in $G$.
		\item There are infinitely many elements in $G$ which are not the image of a periodic or reducible element of $\mcg \Sigma$.
	\end{enumerate}
\end{theo}

\begin{proof}
	We would like to apply \autoref{res : SC - partial periodic quotient} with the mapping class group $\mcg \Sigma$ acting on the complex of curve $X$ of $\Sigma$. 
	However $\mcg \Sigma$ does contains elements of order $2$.
	To avoid this difficulty we consider a normal torsion-free finite-index subgroup $N$ of $\mcg \Sigma$.
	We write $\kappa$ for the index of $N$ in $\mcg \Sigma$.
	This groups acts acylindrically on $X$ thus $\rinj NX$ is positive and $\nu (N,X)$ is finite.
	Since $N$ has no torsion, $e(N,X) = 1$.
	Note also that for every $f \in \mcg \Sigma$, $f^\kappa$ belongs to $N$.
	Thus the theorem follows from \autoref{res: SC - partial periodic quotient - acylindrical case}.
	\end{proof}

\paragraph{Amalgamated product.}
Let $G$ be a group.
A subgroup $H$ of $G$ is \emph{malnormal} if for every $g \in G$, $gHg^{1} \cap H = \{1\}$ unless $g$ belongs to $H$.
The following theorem is known from specialists in the field.
However it has not be published so far.

\begin{theo}
\label{res: quotient amalgamated products}
	Let $A$ and $B$ be two groups without involution.
	Let $C$ be a subgroup of $A$ and $B$ malnormal in $A$ or $B$.
	There is an integer $n_1$ such that for every odd exponent $n \geq n_1$ there exists a group $G$ with the following properties.
	\begin{enumerate}
		\item The groups $A$ and $B$ embed into $G$ such that the diagram below commutes.
		\begin{center}
			\begin{tikzpicture}[description/.style={fill=white,inner sep=2pt},] 
				\matrix (m) [matrix of math nodes, row sep=2em, column sep=2.5em, text height=1.5ex, text depth=0.25ex] 
				{ 
					C	& B		\\
					A	& G		\\
				}; 
				\draw[>=stealth, ->] (m-1-1) -- (m-1-2);
				\draw[>=stealth, ->] (m-2-1) -- (m-2-2);
				\draw[>=stealth, ->] (m-1-1) -- (m-2-1);
				\draw[>=stealth, ->] (m-1-2) -- (m-2-2);
			\end{tikzpicture} 
		\end{center}
		\item For every $g \in G$, if $g$ is not conjugated to an element of $A$ or $B$ then $g^n=1$.
		\item There are infinitely many elements in $G$ which are not conjugated to an element of $A$ or $B$.
	\end{enumerate}
\end{theo}

\begin{proof}
	We denote by $X$ the Bass-Serre tree associated to the amalgamated product $A*_CB$ (see for instance \cite{Ser77}).
	By construction $A*_CB$ acts by isometries on $X$.
	An element $h \in A*_CB$ is elliptic for this action if and only if it is conjugate to an element of $A$ or $B$.
	It is loxodromic otherwise.
	In particular $A*_CB$ does not contain any element of order $2$.
	Moreover $A$ and $B$ are elliptic subgroups.
	Since $C$ is malnormal in $A$ or $B$ the stabilizer of any path of length at least $3$ is trivial.
	It follows that the action of $H$ on $X$ is acylindrical.
	On the other hand, any elementary loxodromic subgroup is cyclic infinite, hence $e(A*_CB,X) = 1$.
	The theorem follows from \autoref{res: SC - partial periodic quotient - acylindrical case}.
\end{proof}

\paragraph{Hyperbolic groups.}
Let $G$ be a group acting properly co-compactly on a hyperbolic length space.
In particular $G$ is a hyperbolic group.
Moreover this action is acylindrical.
In this particular case, the invariant $e(G,X)$ can be characterized algebraically.
Indeed the elementary loxodromic subgroup of $G$ are exactly the ones containing $\Z$ as a finite index subgroup.
Therefore we simply write $e(G)$ for $e(G,X)$.

\paragraph{}
If $G$ is torsion-free, there exists an integer $n_0$ such that for every odd exponent $n \geq n_0$ the quotient $G/G^n$ is infinite.
This result was first proved by A.Y.~Ol'shanskii \cite{Olc91}. 
The work of T.~Delzant and M.~Gromov provides an alternative prove of the same result \cite{DelGro08} (see also \cite{Coulon:2013tx}).
Our study allow us to add some harmless torsion in the original group $G$, 
We recover here a particular case of a theorem proved by A.Y.~Ol'shanskii and S.V.~Ivanov in \cite{IvaOlc96} (their result works for also for hyperbolic groups with 2-torsion).

\begin{theo}
\label{res: torsion quotient of hyperbolic groups}
	Let $G$ be a non-elementary hyperbolic group without involution such that $e(G)$ is odd.
	There exist integers $\kappa$ and $n_1$ such that for every odd integer $n \geq n_1$, the quotient $G/G^{\kappa n}$ is infinite.
\end{theo}

\begin{proof}
	Since $G$ is hyperbolic, its action on its Cayley graph $X$ is proper and co-compact.
	In particular it is acylindrical.
	Moreover it contains only a finite number of conjugacy classes of elliptic elements (see \cite[Lemme 3.5]{CooDelPap90}).
	Since $G$ has no involution, there exists an odd integer $\kappa$, multiple of $e(G)$ such that for every elliptic element $u$ of $G$, the order of $u$ divides $\kappa$.
	Hence we can apply \autoref{res: SC - partial periodic quotient - acylindrical case} with $G=N$.
	There exists an integer $n_1$ such that for every odd exponent $n \geq n_1$ there exists an infinite quotient $G/K$ of $G$ with the following property.
	For every loxodromic element $g \in G$ either $g^{\kappa n}=1$ in $G/K$ or there exists an elliptic element $u \in G$ such that $g=u$ is $G/K$.
	However for every elliptic element $u \in G$, we have $u^\kappa = 1$.
	It follows that $G/K$ is an infinite quotient of $G/G^{\kappa n}$, hence $G/G^{\kappa n}$ is infinite.
\end{proof}

\rem One can actually prove that the quotient $G/K$ that appears in the proof is exactly $G/G^{\kappa n}$.
However this is not needed here.

\paragraph{Relatively hyperbolic groups.}
The notion of a group being hyperbolic relative to a collection of subgroups was introduced by Gromov in \cite{Gro87}.
This class extends the one of hyperbolic groups and covers various examples like fundamental groups a negatively curved manifold with finite volume, HNN extensions over finite groups, geometrically finite Kleinian groups, etc.
Since Gromov's original paper, several different definitions have emerged, see for instance \cite{Bowditch:2012ga,Far98}. 
These definitions have been shown to be almost equivalent \cite{Bowditch:2012ga,Szczepanski:1998fo,Hruska:2010iw}.
For our purpose we will use the following one.

\begin{defi}{\rm \cite[Definition 3.3]{Hruska:2010iw}} \quad
\label{def: relatively hyperbolic}
	Let $G$ be a group and $\{ H_1, \dots, H_m\}$ be a collection of subgroups of $G$.
	We say that $G$ is \emph{hyperbolic relative to} $\{H_1, \dots, H_m\}$ if there exists a proper geodesic hyperbolic space $X$ and a collection $\mathcal Y$ of disjoint open horoballs satisfying the following properties.
	\begin{enumerate}
		\item $G$ acts properly by isometries on $X$ and $\mathcal Y$ is $G$-invariant.
		\item If $U$ stands for the union of the horoballs of $\mathcal Y$ then $G$ acts co-compactly on $X \setminus U$.
		\item $\{H_1, \dots, H_m\}$ is a set of representatives of the $G$-orbits of $\set{\stab Y}{Y \in \mathcal Y}$.
	\end{enumerate}
\end{defi}

The action of $G$ on the space $X$ given by \autoref{def: relatively hyperbolic} is not acylindrical.
Indeed the subgroups $H_j$ can be parabolic. 
This cannot happen with an acylindrical action \cite[Lemma 2.2]{Bowditch:2008bj}.
More generally, the non-loxodromic elementary subgroups of $G$ are exactly the finite subgroups of $G$ and the ones which are conjugated to a subgroup of some $H_j$.
As in the case of hyperbolic groups, the invariant $e(G,X)$ can be characterized algebraically.
Indeed a subgroup $E$ of $G$ is loxodromic if and only if $\Z$ is a finite-index subgroup of $E$ and $E$ is not conjugated to a subgroup of some $H_j$.
Therefore we simply write $e(G)$ for $e(G,X)$.
Note that this notation implicitly depends on the collection $\{H_1, \dots, H_m\}$ though.

\paragraph{}
As in the case of groups with an acylindrical action, one can prove that $\rinj GX$ is positive whereas $\nu(G,X)$ and $A(G,X)$ are finite.
\autoref{res : SC - partial periodic quotient} gives the following result.
\begin{theo}
\label{res: SC - partial periodic quotient - acylindrical case}
	Let $G$ be a group without involution and $\{H_1, \dots, H_m\}$ be a collection of subgroups of $G$.
	Assume that $G$ is hyperbolic relatively to $\{H_1, \dots, H_m\}$ and $e(G)$ is odd.
	There is a critical exponent $n_1$ such that every odd integer $n \geq n_1$ which is a multiple of $e(G)$ has the following property.
	There exists a quotient $G/K$ of $G$ such that
	\begin{itemize}
		\item if $E$ is a finite subgroup of $G$ or conjugated to some $H_j$, then the projection $G \twoheadrightarrow G/K$ induces an isomorphism from $E$ onto its image;
		\item for every element $g \in G/K$, either $g^n=1$ or $g$ is the image a non-loxodromic element of $G$;
		\item there are infinitely many elements in $G/K$ which do not belong to the image of an elementary non-loxodromic subgroup of $G$.
	\end{itemize}
\end{theo}

\paragraph{Other examples.}
In \cite{Osin:2013te}, D.~Osin investigates the class of groups that admit a non-elementary acylindrical action on a hyperbolic space.
He called them \emph{acylindrically hyperbolic groups}.
It turns out that this class is very large.
Here are a few examples in addition to the one we already studied.
\begin{enumerate}
	\item If a group $G$ is not virtually cyclic and admits an action on a hyperbolic space with at least one loxodromic element satisfying the WPD property, then $G$ is acylindrically hyperbolic.
	In particular for every $r \geq 2$, the outer automorphism group $\out{\free r}$ of the free group $\free r$ of rank $r$ is acylindrically hyperbolic.
	Indeed given any automorphism $\phi \in \out{\free r}$ which is irreducible with irreducible powers (iwip), M. Bestvina and M. Feighn constructed a hyperbolic $\out{\free r}$-complex where $\phi$ satisfies the WPD property \cite{Bestvina:2010ey}.
	\item If $G$ contains a proper infinite hyperbolically embedded subgroup (see \cite{Dahmani:2011vu} for a precise definition) $G$ is acylindrically hyperbolic.
	One example is the \emph{Cremona group} $\mathbf {Bir}(P^2_\C)$. 
	It is the group of birationnal transformations of the projective planes.
	It has been shown by S.~Cantat and S.~Lamy that $\mathbf {Bir}(P^2_\C)$ admits an action on a hyperbolic space  with many loxodromic elements \cite{Cantat:2013fl}.
	F.~Dahmani, V.~Guirardel and D.~Osin used then these data to prove that $\mathbf {Bir}(P^2_\C)$ contains virtually cyclic hyperbolically embedded subgroups \cite{Dahmani:2011vu}.
	\item In \cite{Sisto:2011uc} A.~Sisto proved that if $G$ is a group acting properly on a proper $\operatorname{CAT}(0)$ space, then then every rank $1$ element of $G$ is contained in a hyperbolically embedded virtually cyclic subgroup. which provides other examples of acylindrically hyperbolic groups.
	In particular every Right-Angle Artin Group which is not cyclic, or directly decomposable is acylindrically hyperbolic.
	\item In \cite{Minasyan:2013wp}, A.~Minasyan and D.~Osin used actions on trees to provide other examples of acylindrically hyperbolic group.
	Among others, they gave the following results.
	For every field $k$, the group $\aut{k[x,y]}$ of automorphisms of the polynomial algebra $k[x,y]$ is acylindrically hyperbolic.
	Any one relator group with at least three generators is acylindrically hyperbolic.
	
\end{enumerate}
For all these examples we can apply \autoref{res: SC - partial periodic quotient - acylindrical case} provided we can deal with the even torsion.
However we do not necessarily have an intrinsic characterization for the type (elliptic or loxodromic) of the elements of $G$ for the corresponding action on $X$.
For instance, it is not known if there exists an acylindrical action of $\out{\free r}$ on a hyperbolic space such that the loxodromic elements are exactly the iwip automorphisms of $\free r$.